%% file: PLHA.tex
\include{pre.tex}

\begin{document}

\title{%
On the Provability Logic of $\HA$%
}
\author{
Mojtaba Mojtahedi\thanks{
Department of Mathematics, Statistics and Computer Science, 
College of Sciences, University of Tehran, Iran;
and Department of Mathematics WE16, Ghent University, Krijgslaan 281-S8, B9000 Ghent, 
Belgium.  
Email: 
\url{mojtahedy@gmail.com}, Home page: \url{http://mmojtahedi.ir}
}}
\maketitle

\begin{abstract}
We axiomatize the provability logic of $\HA$ and prove its decidability. Furthermore, we axiomatize the preservativity and relative admissibility relations for several modal logics extending $\ikfourcp$. A principal technical tool is the introduction of a new type of semantics, termed \emph{provability models}, for modal logics extending $\iglcp$. This semantics combines elements of standard Kripke semantics with provability in propositional modal logics.
\end{abstract}
\tableofcontents
\listoftables

\section{Introduction}
Provability logic is a propositional modal logic in which the formula $\Box A$ is interpreted to mean: ``$A$ is provable in a given formal theory $\sft$''. The first explicit interpretation of a modal operator in this manner appears in \citep{Godel33}, where provability is used to interpret the modal operator $\Box$ in the classical modal logic ${\sf S4}$, in order to embed intuitionistic propositional logic $\ipc$ into ${\sf S4}$. This result, together with G\"odel's arithmetization of syntax and proof for the incompleteness theorems \citep{Godel}, mark the origins of the field of provability logic. Since then, considerable research has been conducted in this area, and many problems remain open. We refer the reader to \citep{VisBek,ArtBekProv,sep-logic-provability} for comprehensive surveys.

A celebrated result in provability logic is the characterization of the provability logic of Peano Arithmetic $\PA$ \citep{Solovay,Lob}. More precisely, \citep{Solovay,Lob} prove that $\GL\vdash A$ if and only if, for every arithmetical $\PA$-interpretation $\apa$, we have $\PA\vdash \apa(A)$, where $\GL$ is the G\"odel-L\"ob logic, defined as ${\sf K4}$ plus L\"ob's principle $\Box(\Box A\to A)\to\Box A$. Here, $\apa$ is called a $\PA$-interpretation if it satisfies the following conditions:
\begin{itemize}
\item $\apa(a)$ is an arbitrary first-order sentence in the language of arithmetic for every atomic proposition $a$.
\item $\apa$ commutes with the Boolean connectives $\vee$, $\wedge$, and $\to$.
\item $\apa(\Box A)$ is an arithmetization (a formalization in the first-order language of arithmetic) of the statement: ``$\apa(A)$ is provable in $\PA$''.
\end{itemize}
For further details, see \citep{Smorynski-Book,Boolos}. This result is known to be robust; it can be generalized to other sufficiently strong first-order \textit{classical} theories such as ${\sf I\Delta_0}+{\sf exp}$, ${\sf ZF}$, and $\ZFC$.

Another significant direction in provability logic considers provability interpretations in Heyting Arithmetic ($\HA$), the intuitionistic fragment of $\PA$. Early results in this area include \citep{Myhill,Friedman75}, which demonstrate that $\Box (B\vee C)\to(\Box B\vee \Box C)$ does not belong to the provability logic of $\HA$. However, \citep{Leivant} shows that the axiom schema $\Box (B\vee C)\to\Box (\Box B\vee C)$ does belong to this logic. Subsequent work by \citep{VisserThes,Visser82} studies the provability logic of $\HA$ and characterizes its letterless fragment. \citep{Visser02,Iemhoff.Preservativity,IemhoffT} investigate a generalization of provability called \emph{preservativity}, which serves as the intuitionistic analogue of interpretability \citep{VisserInterpretability,Visser-Interpretability}. The propositional language for preservativity, denoted $\lcalp$ in this paper, includes a binary modal operator $\rhd$ with the following interpretation for $A\rhd B$:
\begin{center}
For every $\Sigma_1$-sentence $S$, if $\HA\vdash S\to A$, then $\HA\vdash S\to B$.
\end{center}
Albert Visser axiomatized a logic called $\iph$ and, together with Dick de Jongh, proved its soundness for the arithmetical interpretation described above. Rosalie Iemhoff conjectured that $\iph$ is also complete for this interpretation. While the results of the present paper reinforce Iemhoff's conjecture, the arithmetical completeness of $\iph$ remains open. \citep{Sigma.Prov.HA,Jetze-Visser,Sigma.Prov.HA*} characterize the provability logic of $\HA$ and $\HA^*$ (a self-completion of $\HA$ introduced in \citep{Visser82}) for $\Sigma_1$-substitutions. In a sense, \citep{Rev-SPLHA} shows that the provability logic of $\HA$ for $\Sigma_1$-substitutions, denoted $\lles$ here, is essentially $\iglca$. More precisely, \citep{Rev-SPLHA} defines a translation $(.)^\Boxh$ that embeds $\lles$ into $\iglca$. \citep{mojtahedi2021hard} characterizes the $\Sigma_1$-provability logics of $\HA$ and $\HA^*$ relative to $\PA$ and the standard model $\nat$.

In this paper, we axiomatize the provability logic of $\HA$ as $\iglcp$ extended with all axioms of the form $\Box A\to\Box B$ for every pair $A$, $B$ satisfying the following property (denoted $A\priglcdsnb B$ in this paper):
\begin{center}
For every $E\in\cdsnb$, if $\iglcp\vdash E\to  A$, then $\iglcp\vdash E\to B$.
\end{center}
The precise definition of $\cdsnb$ is somewhat technical and can be found in \Cref{notation-set}. Roughly speaking, $\cdsnb$ is the set of all propositions that are projective relative to $\NNIL$ and are self-complete.

\subsection*{Road map}
\Cref{preli-sec} contains all elementary and general definitions, along with related facts. Two principal required results are also included in \Cref{2salient}. In \Cref{secPres}, we extend the results of \citep{PLHA0} and axiomatize $\priglcdsnb$, together with several other preservativity and admissibility relations. \Cref{Prov-semant} introduces a Kripke-style semantics, called \emph{provability models}, for which we prove soundness and completeness for the provability logic of $\HA$. Finally, \Cref{sec-completeness-reduction} uses provability models to reduce arithmetical completeness to its $\Sigma_1$ version, following a method similar to that employed in \citep{reduction}.

%

\section{Preliminary definitions and facts}\label{preli-sec}
This section presents elementary definitions and results. We first define propositional languages (\Cref{sec-languages}) and substitutions (\Cref{sec-sub,sec-sub-1}). Subsequently, several axiom schemata and propositional logics are defined (\Cref{propo-logics}). \Cref{sec-cto} defines two complexity measures, $\cto(A)$ and $\ctob(A)$. Kripke semantics for intuitionistic modal logics are defined in \Cref{sec-Kripke}. For the sake of self-containment, we also prove Kripke completeness for $\iglcp$ and $\iglca$ in \Cref{sec-Kripke}. \Cref{sec-Box-trans} defines G\"odel's translation \citep{Godel33}. Some notation concerning sets of propositions is introduced in \Cref{notation-set}. \Cref{sec-NNIL-def} defines the set $\NNIL$ of propositions and states some of their properties. \Cref{pres-admis} defines relative admissibility and preservativity and proves elementary facts about them. \Cref{glb} introduces the greatest lower bound relative to a set $\Gamma$ of propositions and proves some of its elementary properties. In \Cref{sec-fix}, we prove a simultaneous fixed-point theorem for $\iglcp$. Finally, \Cref{2salient} states two salient results from \citep{Sigma.Prov.HA,PLHA0} that are crucial for characterizing the provability logic of $\HA$.

\subsection{Propositional language}\label{sec-languages}
The non-modal propositional language $\lcalz$ includes the connectives $\vee$, $\wedge$, $\to$, $\bot$, a countably infinite set of atomic variables $\varr:=\{x_1,x_2,\ldots\}$, and a countably infinite set of atomic parameters $\parr:=\{p_1,p_2,\ldots\}$. The inclusion of parameters in the propositional language is for technical reasons. As will be seen later, it is convenient to have atomic symbols in the propositional language with the intended interpretation as $\Sigma_1$-sentences. Consequently, in the axiomatizations defined in \Cref{propo-logics}, we always keep this intended interpretation in mind, which leads to the inclusion of the axiom $p_i\to \Box p_i$ in $\ikfourcp$. A further consideration regarding this intended meaning for parameters is that we cannot substitute them arbitrarily. The only permitted substitutions for a parameter are those that do not violate its intended $\Sigma_1$-interpretation. However, except for the results in \Cref{sec-completeness-reduction}, we may assume that the set of parameters of the language is empty, in which case the axiom schema $\CPp:=p_i\to \Box p_i$ is automatically removed.

Negation $\neg$ is defined as $\neg A:= A\to\bot$, and $\top:=\bot\to\bot$. We use the notation $\lcalz(X)$ for the set of all Boolean combinations of propositions in $X$; i.e., $\lcalz(X)$ is the smallest set containing $X\cup\{\bot\}$ that is closed under conjunction, disjunction, and implication. The modal language $\lcalb$ is defined as $\lcalz$ augmented with the unary modal operator $\Box$. The language $\lcalp$ denotes the propositional language $\lcalz$ augmented with a binary modal operator $\rhd$. Whenever we consider the language $\lcalp$, we assume that $\Box B:=\top\rhd B$. In this sense, $\lcalp$ is an extension of $\lcalb$.

The union $\varr\cup\parr$ is denoted by $\atom$, the set of atomic propositions. Additionally, we define
$$
\boxed:=\{\Box B: B\in\lcalb\}
\quad\text{and}\quad
\parb:=\parr\cup\boxed
\cup\{\bot\}
\quad\text{and}\quad
\atomb:=\parb\cup\varr.$$

\subsection{Propositional substitutions}\label{sec-sub}
A (propositional) substitution $\theta$ is a function on the propositional language that commutes with all connectives. More precisely, $\theta$ satisfies the following conditions:
\begin{itemize}
\item $\theta(a)$ is a proposition in the language $\lcalp$ for every $a\in\atom$.
\item $\theta(B\circ C)=\theta(B)\circ\theta(C)$ for every $\circ\in\{\vee,\wedge,\to\}$.
\item $\theta(\bot)=\bot$.
\item $\theta(  B\rhd C):=\theta(B)\rhd\theta(C)$.
\end{itemize}
\textit{By default, we assume that all substitutions are the identity on $\parr$.} However, there are places where we need to substitute parameters as well; these will be made explicit to the reader. For a substitution $\theta$, the function $\ov\theta$ is defined identically to $\theta$, except that it acts as the identity on boxed propositions:
\begin{itemize}
\item $\th{(B\circ C)}=\th(B)\circ\th(C)$ for every $\circ\in\{\vee,\wedge,\to\}$.
\item $\th(\bot)=\bot$.
\item $\th(B\rhd C):=B\rhd C$.
\end{itemize}
We call $\th$ an \textit{outer} substitution. Note that $\th{}$ and $\theta$ coincide in the case of the non-modal language.

\subsection{Arithmetical substitutions}\label{sec-sub-1}
An arithmetical substitution is a function $\alpha$ on the set of atomic variables and parameters $\atom$ such that $\alpha(a)$ is a first-order arithmetical sentence for every $a\in\atom$, and $\alpha(a)\in\Sigma_1$ for every $a\in\parr$. Moreover, $\alpha$ is called a $\Sigma_1$-substitution if $\alpha(a)\in\Sigma_1$ for every $a\in\atom$.
\\
An arithmetical substitution $\alpha$ can be extended to $\lcalp$ as follows:
\begin{itemize}
\item $\aha(a):=\alpha(a)$ for every $a\in\atom$, and $\aha(\bot)=\bot$.
\item $\aha$ commutes with Boolean connectives: $\vee$, $\wedge$, and $\to$.
\item $\aha(A\rhd B)$ is defined as an arithmetization of the following statement:
\begin{center}
For every $E\in\Sigma_1$, if $\HA\vdash E\to \aha(A)$, then $\HA\vdash E\to \aha(B)$.
\end{center}
\end{itemize}
Note that the above definition of $\aha$ is compatible with the well-known provability interpretation for $\Box$ when one assumes $\Box B:=\top\rhd B$.

A strong variant $\ahap$ is defined similarly:
\begin{itemize}
\item $\ahap(a):=\alpha(a)$ for every $a\in\atom$, and $\ahap(\bot)=\bot$.
\item $\ahap$ commutes with Boolean connectives: $\vee$, $\wedge$, and $\to$.
\item $\ahap(A\rhd B)$ is defined as $\varphi$ together with its provability statement in $\HA$; i.e., $\ahap(A\rhd B):=\varphi\wedge\Box_{_{\sf HA}}\varphi$, where $\varphi$ is an arithmetization of the following statement:
\begin{center}
``For every $E\in\Sigma_1$, if $\HA\vdash E\to \ahap(A)$, then $\HA\vdash E\to \ahap(B)$." 
\end{center}
\end{itemize}

\subsection{Propositional logics and theories}\label{propo-logics}
We consider $\IPC$ as the intuitionistic propositional logic over the modal language $\lcalb$; i.e., a set of propositions in $\lcalb$ that is closed under modus ponens (\Ax{$ A $}\Ax{$ A\to B $}\BI{$ B $}\DP) and includes all of the following axiom schemata:
\begin{itemize}
	\item $ A\to (B\to A) $,
\item  $(A\to (B\to C))\to ((A\to B)\to (A\to C))$,
\item  $A\to (B\to (A\wedge B))$,
\item  $(A\wedge B)\to A$,   $(A\wedge B)\to B$,
\item  $(A\to C)\to ((B\to C)\to ((A\vee B)\to C))$,
\item  $A\to (A\vee B)$,  $B\to (A\vee B)$.
\end{itemize}
By default, we use $\vdash$ for derivability in $\IPC$.

A \textit{theory} $\sft$ is a set of formulas that includes all axioms of $\IPC$ listed above and is closed under modus ponens.

The following axiom schemata are defined:
\begin{itemize}[leftmargin=1.5cm]
\item[\ul{\sf K}:] $ \Box(A\to B)\to (\Box A \to\Box B) $.
\item[\ul{\sf 4}:] $ \Box A\to\Box\Box A $.
\item[\ul{\sf L}:] $ \Box(\Box A\to A)\to\Box A $. (The {\sf L\"ob}'s axiom)
\item[$\underline{\cpp}$:] $ p\to\Box p $ for every $ p\in\parr $.
\item[$\underline{\cpa}$:] $ a\to\Box a $ for every $ a\in\atom $.
\item[$\ulvisgt $:] $ \Box A\to \Box B$ for every $  A\prtg B $, where $\prtg$ is defined in \Cref{pres-admis}.
\item[$\ulHa$:]  $ \ulvisp{\cdsnb}{\iglcp} $, where $\cdsnb$ is as defined in \Cref{notation-set}.
\item[$\ulHs$:] $ \ulvis{\snnil}{\iglca} $, where $\snnil$ is as defined in \Cref{notation-set}. 
\item[$\ulHb$:] $ \ulvis{\snnilb}{\iglcp} $, where $\snnilb$ is as defined in \Cref{notation-set}.
\end{itemize}
For an axiom schema $\underline{\sf X}$, let $\overline{\sf X}$ denote 
$\Box \underline{\sf X}$, and let ${\sf X}$ denote 
$\underline{\sf X}\wedge\overline{\sf X}$. Given a logic ${\sf L}$ and axiom schemata ${\sf X_1,\ldots,X_n}$, the logic ${\sf LX_1\ldots X_n}$ is defined as ${\sf L}$ plus the axioms ${\sf X_1,\ldots,X_n}$. We then define the following modal logics:
\begin{itemize}
	\item {\sf i}: $\IPC$ with   the necessitation of all axioms in $\IPC$\footnote{The necessitation of a formula $A$, is just $\Box X$} together with $\cpp$.
	\item $ \iglcp:=\ikfourcp{\sf L}$.
\end{itemize}
Note that in this setting, $\iglcp$ and $\ikfourcp$ are closed under necessitation: \Ax{$ A $}\UI{$ \Box A $}\DP; however, this inference rule is not primarily assumed in their proof systems. Moreover, throughout this paper, the necessitation rule is admissible for the modal systems (systems in the language $\lcalb$) that we consider (except for $\IPC$, which does not have this property).

\begin{remark}
As an observant reader might already noticed, in our extended language with parameters, 
all modal logics are augmented with some additional completeness axiom for parameters. 
In other words, for example, $\iglcp$ also include the axiom $p\to\Box p$
for atomic parameters. Nevertheless, this is just a conservative extension of standard 
$\iglcp$. This means that for every $A\in\lcalb$ without any parameters, 
$\iglcp\vdash A$  iff $A$ is a theorem of standard $\iglcp$. The same explanation also holds for other logics like $\iglcph$ that are being studied in this paper. 

So why we do consider such extended format of logics? The reason is only due to technical more convenience, later in \cref{sec-completeness-reduction}, where we prove 
that $\iglcph\nvdash A$ implies $\lles\nvdash \alpha(A)$ for some propositional substitution $\alpha$. Actually, the substitution $\alpha$ includes several parameters from the extended language. 
\end{remark}


\subsection{Complexity measures $\cto(A)$ and $\ctob(A)$}
\label{sec-cto}
Given $A\in\lcalb$, define $\cto(A)$ as the maximum number of nested implications outside boxes, and $\ctob(A)$ as the maximum of the number of parameters and boxed subformulas in $A$ and $\max\{\cto(B):\Box B\in\sub{A}\}$. More precisely, we define $\cto(A)$ inductively as follows:
\begin{itemize}
\item $\cto(A):=0$ for $A=\Box B$ or $A\in\atom$ or $A=\bot$.
\item $\cto(A\circ B):=\max \{\cto(A),\cto(B)\}$ for $\circ\in\{\vee,\wedge\}$.
\item $\cto(A\to B):=1+\max\{\cto(A),\cto(B)\}$.
\end{itemize}
Then define
\emli
$$\ctob(A):=\max(\{\cto(B):\Box B\in\sub A\}\cup \{n_A\}),$$
where $n_A$ is defined as the number of elements in $\sub A\cap \parb$. Recall that $\parb:=\parr\cup\{\Box B: B\in\lcalb\}\cup\{\bot\}$.

A notable feature of the complexity measure $\cto(A)$ is that there are only finitely many propositions $A$ with $\cto(A)\leq n$:
\begin{lemma}\label{Remark-Finiteness-c(A)}
Modulo $\IPC$-provable equivalence, there are finitely many propositions $A\in \lcalz(X)$ with $\cto(A)\leq n$, where $X$ is a finite set of atomic or boxed propositions. Moreover, one can effectively compute the finite set of such propositions.
\end{lemma}
\begin{proof}
By induction on $n$, we define an upper bound $f(n)$ for the number of propositions $A\in\lcalz(X)$ with $\cto(A)\leq n$. The computability of such a set of propositions is left to the reader.
\begin{enumerate}[leftmargin=*]
\item $f(0):$ Observe that any $A$ with $\cto(A)=0$ is $\IPC$-equivalent to a disjunction of conjunctions of propositions in $X$. Hence, $f(0)=2^{2^m}$ is an obvious upper bound, where $m$ is the number of propositions in $X$.
\item $f(n+1):$ For every implication $B\to C$ with $\cto(B\to C)\leq n+1$, we have $\cto(B),\cto(C)\leq n$, and hence $f(n)^2$ is an upper bound for the number of inequivalent such propositions. Then, since
\mipc
every proposition is a disjunction of conjunctions of propositions in $X$ or implications, the following definition provides an upper bound:
\[f(n+1):=2^{2^{[m+f(n)^2]}}.\qedhere\]
\end{enumerate}
\end{proof}

\subsection{Kripke models for intuitionistic modal logics}\label{sec-Kripke}
A Kripke model for intuitionistic modal logic combines features of Kripke models for intuitionistic logic and for classical modal logic. As expected, it contains two relations: one ($\pce$) for intuitionistic implication and another ($\R$) for the modal operator $\Box$ or $\rhd$. More precisely, a Kripke model is a quadruple $\kcal=(W,\prec,\R,\V)$ with the following properties:
\begin{itemize}
	\item $ W\neq\emptyset $.
	\item $(W, \prec)$ is a partial order (transitive and irreflexive). We write $\pce$ for the reflexive closure of $\prec$.
	\item $\V$ is the valuation on atomics; i.e., $V\subseteq W\times\atom$.
	\item $w\pce u$ and $w\V a$ imply $u\V a$ for every $w,u\in W$ and $a\in\atom$.
	\item $({\pce};{\R})\subseteq {\R}$; i.e., $w\pce u\R v$ implies $w\R v$. This condition ensures that the previous property holds for all modal propositions, not only for $a\in\atom$.
	\item $w\R u$ and $w\V p$ imply $u\V p$ for every $w,u\in W$ and $p\in\parr$.
\end{itemize}

The valuation relation $\V$ can be extended to all modal propositions as follows:
\begin{itemize}
	\item $\kcal,w\Vdash a$ iff $w\V a$, for $a\in\atom$.
	\item $\kcal,w\Vdash A\wedge B$ iff $\kcal,w\Vdash A$ and $\kcal,w\Vdash B$.
	\item $\kcal,w\Vdash A\vee B$ iff $\kcal,w\Vdash A$ or $\kcal,w\Vdash B$.
	\item $\kcal,w\Vdash A\to B$ iff for every $u\sce w$, if $\kcal,u\Vdash A$, then $\kcal,u\Vdash B$.
	\item $\kcal,w\Vdash A\rhd B$ iff for every $u\sqsupset w$ with $\kcal,u\Vdash A$, we have $\kcal,u\Vdash B$.
	\item $\kcal,w\Vdash \Box A$ iff for every $u\sqsupset w$, we have $\kcal,u\Vdash A$.
\end{itemize}
We also define the following strengthening of $\Vdash$:
\emli
$$\kcal,w\vdashp A \quad\text{iff}\quad \text{there is some }u\R w\text{ such that }\kcal,w'\Vdash A \text{ for every }w'\sqsupset u.$$
Define $u \Rvar v$ ($u\Rvarr v$) iff there is some $u'$ such that $u\sqsubseteq u'\pce v$ ($u\sqsubset u'\pce v$). Notice that, in a transitive frame, $\Rvar$ is just the transitive reflexive closure of the union of the two relations ${\R}\cup{\pce}$. In other words, $w \Rvar u$ iff one can reach $u$ from $w$ by any sequence of the accessibility relations.
Then define the following notions for Kripke models:
\begin{itemize}
	\item \textit{Finite:} if $W$ is a finite set. 
	\item \textit{Transitive:} if $\R$ is transitive; i.e., $u\R v\R w$ implies $u\R w$.
	\item \textit{Rooted:} if there is some node $w_0\in W$ such that $w_0\Rvar w$ for every $w\in W$.
	\item \textit{Conversely well-founded:} if there is no infinite ascending sequence $w_1\R w_2\R\ldots$. Note that this condition implies irreflexivity of $\R$.
	\item \textit{$\pce$-Tree:} if for every $w\in W$, the set $\{u\in W: u\pce w\}$ is a finite linearly ordered set (by $\pce$).
	\item \textit{$\R$-Tree:} if for every $v,u,w\in W$ with $w\R v$ and $u\R v$, either $w\Rvar u$ or $u\Rvar w$.
	\item \textit{Transcendental:} if $u\Rvarr v$ and $w\pce v$, then $w=v$.
	\item \textit{Good:} if all of the above properties hold.
	\item \textit{Atomic ascending:} if $\kcal\Vdash \cpa$.
\end{itemize}
Given two Kripke models $\kripke$ and $\kripkep$, we say that $\kcal'$ is an intuitionistic submodel of $\kcal$ (notation $\kcal'\leq\kcal$) iff $W=W'$, ${\sqsubset}={\sqsubset'}$, $V=V'$, and ${\pce'}\subseteq{\pce}$. A class $\scrk$ of Kripke models has the intuitionistic submodel property if $\kcal'\leq\kcal\in\scrk$ implies $\kcal'\in\scrk$. A modal logic $\sft$ is said to have the intuitionistic submodel property iff it is sound and complete for some class $\scrk$ of good Kripke models with the intuitionistic submodel property.
\begin{theorem}\label{Kripke-completeness-igl}
	$\iglcp$ is sound and complete for good Kripke models. Also, $\iglca$ is sound and complete for good atomic ascending Kripke models.
\end{theorem}
\begin{proof}
	The soundness parts are straightforward and are left to the reader. The second statement can be easily derived from the first and is also left to the reader.
	\\
	Let $\iglcp\nvdash A$ for some $A\in\lcalb$. We must find a good Kripke model $\kripke$ such that $\kcal\nVdash A$. Using canonical models (see \citep[Prop.~4.3.2]{IemhoffT} or \citep{Iemhoff}), one can find a finite, transitive, $\cpp$, and conversely well-founded Kripke model such that $\kcal\nVdash A$. Then, by \Cref{unravelling}, one obtains the desired result.
\end{proof}
\begin{lemma}\label{unravelling}
Let $\kcal=(W,\pce,\R,\V)$ be a finite irreflexive Kripke model. Then for every $w_0\in W$, there exists a finite rooted (with root $\langle w_0\rangle$) transcendental tree Kripke model $\tcal=(W',\pce',\R',\V')$ that is equivalent to $\kcal_{w_0}$; i.e., there exists a function $e:W'\longrightarrow W$ such that $e(\langle w_0\rangle)=w_0$ and for every $w\in W'$ and $A\in\lcalb$, we have $\tcal,w\Vdash A$ iff $\kcal,e(w)\Vdash A$.
\end{lemma}
\begin{proof}
Define $\tcal:=(W',\pce',\R',\V')$ as follows.
\begin{itemize}[leftmargin=*]
		\item $W' :=$ the set of finite sequences (excluding the empty sequence) $\vv{w}:=\langle w_0,s_1,w_1,s_2,w_2,s_3\ldots,s_n,w_n\rangle$ with the following properties:
		\begin{itemize}
			\item $w_i\in W$ and $s_i\in\{\prec,\R\}$ for every $1\leq i\leq n$.
			\item $w_i\mathrel{s_{i+1}} w_{i+1}$ for every $0\leq i < n$. This means that if $s_{i+1}={\prec}$, then $w_i\prec w_{i+1}$, and if $s_{i+1}={\R}$, then $w_i\R w_{i+1}$.
		\end{itemize}
		\item Define the function $e:W'\longrightarrow W$ as follows. For $\vv{w}:=\langle w_0,s_1,w_1,s_2,w_2,s_3\ldots, s_n,w_n\rangle$, define $e(\vv{w}):=w_n$. Also define $e(\langle w_0\rangle):=w_0$.
		\item Let $\vv{v}:=\langle v_0,t_1,v_1,t_2,v_2,t_3\ldots,t_m,v_m\rangle$ and $\vv{w}:=\langle w_0,s_1,w_1,s_2,w_2,s_3\ldots,s_n,w_n\rangle$. Then define $\vv{v} \pce' \vv{w}$ iff the following conditions are met:
		\begin{itemize}
			\item $\vv{v}$ is an initial segment of $\vv{w}$.
			\item For every $m<i\leq n$, we have $s_i={\prec}$.
		\end{itemize}
		Also define $\vv{v} \R' \vv{w}$ iff the following conditions are met:
		\begin{itemize}
			\item $\vv{v}$ is an initial segment of $\vv{w}$.
			\item $n>m$.
			\item $s_n={\R}$.
		\end{itemize}
		\item $\vv{w}\V' a$ iff $e(\vv{w})\V a$.
\end{itemize}
Verifying that this $\tcal$ fulfills all required conditions is left to the reader.
\end{proof}

\subsection{The G\"odel translation $\sggt{(.)}$}
\label{sec-Box-trans}
The following translation is a variant of G\"odel's celebrated translation for embedding $\IPC$ into ${\sf S4}$ \citep{Godel33}.
\begin{definition}\label{Definition-Box translation}
	For every proposition $A\in\lcalb$, define $\ggt A$ and $\sggt A$ inductively as follows:
	\begin{itemize}
\item $\ggt A:=A$, $\sggt{A}:= \Boxdot A$, for $A\in\varr$.
		\item $\ggt A:=\sggt{A}:= A$ for $A\in\parb$.
		\item $\ggt{(B\circ C)}:=\ggt{B}\circ \ggt{C}$ and $\sggt{(B\circ C)}:=\sggt{B}\circ \sggt{C}$ for $\circ\in\{\vee,\wedge\}$.
		\item $\ggt{(B\rightarrow C)}:=\Boxdot(\ggt B\rightarrow \ggt C)$ and $\sggt{(B\rightarrow C)}:=\Boxdot(\sggt B\rightarrow \sggt C)$.
	\end{itemize}
	$A\in\lcalb$ is called \textit{self-complete} if there exists some $B\in\lcalb$ such that $A=\sggt B$:
\emli
	$$\sfs:= \{\ggt B: B\in\lcalb\} \quad\text{and}\quad \sfsb:=\{\sggt B: B\in\lcalb\}.$$
	Note that in the presence of $\cpa$, the two translations $\ggt A$ and $\sggt A$ are equivalent. Also, for $A\in\lcalb(\parb)$, we have $\sggt A=\ggt A$. In the rest of the paper, we may freely interchange between the two translations whenever they are equivalent.
	$A\in\lcalb$ is called \textit{$\sft$-complete} if $\tvdash A\to \Box A$:
\emli
	$$\tcom:=\{A\in\lcalb: \tvdash A\to\Box A\}.$$
	Note that for every $\sft\supseteq \ikfourcp$, we have $\sfs\subseteq \tcom$. Whenever no confusion is likely, we may omit the superscript $\sft$ in the notation $\tcom$ and simply write $\sfc$.
\end{definition}
\begin{theorem}\label{igl-closure-box}
 $\iglcp$, and $\iglca$ are closed under $\sggt{(.)}$; i.e., 
 for every $A\in\lcalb$, $\iglcp\vdash A$ \uparan{$\iglca\vdash A$} implies $\iglcp\vdash \sggt A$ \uparan{$\iglca\vdash \sggt A$}.
\end{theorem}
\begin{proof}
Straightforward induction on the proof of $\iglcp\vdash A$ ($\iglca\vdash A$), left to the reader.
\end{proof}

\subsection{Relative projectivity for the modal language}
\label{Gam-projec-modal}
Motivated by the algebraic notion of projectivity, \citep{Ghilardi97} introduced the notion of projectivity for propositional logics. Then \citep{Ghil99,Ghil2000modal} utilized this notion for intuitionistic and classical transitive modal logics. Obvious connections between unification and admissibility of inference rules were already known. Nevertheless, \citep{Iemhoff-admissibility} used results from \citep{Ghil99} to characterize the admissible rules of intuitionistic logic (their decidability was already discovered in \citep{Rybakov_1987}). In \citep{PLHA0}, we relativized the non-modal notion of projectivity based on the set of No Nested Implications to the Left ($\NNIL$), which has been shown to play an important role in the study of intuitionistic logic (see \citep{Visser-Benthem-NNIL,Visser02}). Essentially, the results in \citep{PLHA0} are imitations of the methods introduced in \citep{Ghil99,Iemhoff-admissibility}.

In this subsection, we extend the propositional notion of relative projectivity to the modal language.
\\
Let $\lcalz(\parb)$ denote the set of Boolean combinations of parameters and boxed propositions. Let $A\in\lcalb$ and $\Gamma\subseteq\lcalb$. A substitution $\theta$ is called \textit{$A$-projective} (in $\sft$) if
\emli
\begin{equation}\label{Eq-star}
\text{For all atomic $a$,} \quad \tvdash A\to (a\lr \theta(a)).
\end{equation}
A substitution $\theta$ is a \textit{$\Gamma$-fier} of $A\in\lcalb$ if
\emli
$$
\tvdash \th(A)\in \Gamma(\parb) \quad \text{i.e., $\th(A)$ is $\sft$-equivalent to some $A'\in\Gamma\cap\lcalz(\parb)$.}
$$
Recall that $\th$ is a nonstandard notion that essentially substitutes only those variables that are not in the scope of boxes. In this case, we use the notation $A\xtratht \Gamma$. If $\Gamma$ is a singleton $\{A'\}$, we write $A\xtratht A'$ instead of $A\xtratht\{A'\}$.
The substitution $\theta$ is a unifier for $A$ if it is a $\{\top\}$-fier for $A$. We say that a substitution $\theta$ \textit{projects} $A$ to $\Gamma$ in $\sft$ (notation: $A\xrattht \Gamma$) if $\theta$ is $A$-projective in $\sft$ and $A\xtratht \Gamma$.
If $\Gamma=\{A'\}$, we simplify $A\xrattht \{A'\}$ to $A\xrattht A'$.
We say that $A$ is \textit{$\Gamma$-projective} in $\sft$ (notation $A\xrat{}{\sft} \Gamma$) if there exists some $\theta$ such that $A\xrat{\theta}{\sft} \Gamma$. Also, $\dar\sft\Gamma$ denotes the set of all propositions that are $\Gamma$-projective in $\sft$. Whenever $\sft$ can be inferred from context, we may omit it and simply write $\darl\Gamma$. We say that $A$ is \textit{projective} if it is $\{\top\}$-projective.

\begin{remark}\label{local-proj}
If $A\xrat{}{\sft} A'\in\lcalz(\parb)$, then there exists some $\tau$ that is the identity on every atomic $a\nin\subo A$ and such that $A\xrat{\tau}{\sft} A'$.
\end{remark}
\begin{proof}
Let $A\xrat{\theta}{\sft} A'\in\lcalz(\parb)$ and define $\tau$ as follows:
\emli
$$\tau(a):=\begin{cases}
\theta(a) \quad &: a\in\subo A\\
a &: \text{otherwise}
\end{cases}$$
Then obviously $A\xrat{\tau}{\sft} A'$.
\end{proof}

\begin{remark}\label{Remark-modal-proj}
Let $\sft$ be a modal logic closed under outer substitutions and containing $\IPC$. Then for every $\Gamma$-projective proposition $A$ in $\sft$, there exists a unique (modulo $\sft$-provable equivalence) $A^\dagger\in\Gamma$ such that $A\xrat{}{\sft}A^\dagger$. Such $A^\dagger$ is called the \textit{$\Gamma$-projection} of $A$ in $\sft$. Moreover, $\tvdash A\to A^\dagger$.
\end{remark}
\begin{proof}
Let $A\xrat{\theta}{}A'$ and $A\xrat{\tau}{}A''$ with $A',A''\in\Gamma$. From the $A$-projectivity of $\theta$ and $\tau$, for every atomic $a$, we have $A\vdasht \theta(a) \lr\tau(a)$. Hence $A\vdasht \th(A)\lr\ov\tau(A)$, and then $A\vdasht A'\lr A''$. By applying $\th$ to both sides of this derivation, we have $\th(A)\vdasht \th(A')\lr\th(A'')$. Since $A',A''\in\lcalz(\parb)$, $\th(A')=A'$ and $\th(A'')=A''$, and thus $\vdasht A'\to A''$. Similarly, we have $\vdasht A''\to A'$.

Next, we show $\tvdash A\to A^\dagger$. Let $A\xrattht A^\dagger$. Then $A\vdasht A\lr \th(A)$, which implies $A\vdasht A\lr A^\dagger$, and hence $\tvdash A\to A^\dagger$.
\end{proof}

\begin{lemma}\label{Lem-red-proj-rela}
Let $\Gamma$ be closed under conjunctions, and let $\sft$ be $\Gamma(\parb)$-axiomatizable over $\IPC$\footnote{This means that there exists a set $\Delta\subseteq\Gamma\cap\lcalz(\parb)$ which, over $\IPC$, axiomatizes $\sft$; i.e., every member of $\Delta$ is a theorem of $\sft$, and for every theorem $A$ of $\sft$, there is a finite subset $\Delta_0\subseteq \Delta$ such that $\IPC\vdash \bigwedge\Delta_0\to A$.}. Then for every $A\xrattht A^\dagger$, there exists some $E\in\Gamma(\parb)$ such that $\sft\vdash E$ and $(E\wedge A)\xrightarrowtail{\IPC}{\theta} (E\wedge A^\dagger)$. Furthermore, for every $A\in\dar\sft \Gamma$, there exists some $B\in\dar\IPC\Gamma$ such that $\sft\vdash A\lr B$. In other words, modulo $\sft$-provable equivalence, $\dar\sft \Gamma\subseteq \dar\IPC \Gamma$.
\end{lemma}
\begin{proof}
We only need to prove the first statement; the rest is a direct consequence. For the first statement, take $E$ as the conjunction of all $\sft$-axioms in $\Gamma(\parb)$ that are used in \cref{Eq-star} and in the equivalence $\tvdash \th(A)\in \Gamma(\parb)$. Note that we need to use \Cref{local-proj} to ensure that only finitely many axioms are required for \cref{Eq-star}.
\end{proof}

\subsection{Notations on sets of propositions}\label{notation-set}
In the remainder of this paper, we deal with several sets of modal propositions. To simplify notation, we write ${\sf X_1\ldots X_n}$ for ${\sf X_1}\cap\ldots\cap{\sf X_n}$ when ${\sf X_i}$ are sets of propositions. For example, we write $\snnil$ for the set of propositions that are $\nnil$ (as defined in \Cref{sec-NNIL-def}) and self-complete (as defined in this section). Also, $\fsubeq$ indicates the finite subset relation. Given $A\in\lcalb$, let $\sub{A}$ be the set of all subformulas of $A$, and let $\subo{A}$ be the set of all subformulas of $A$ that are outside boxes:
\begin{itemize}
 	\item $\subo{a}:=\{a\}$ for $a\in\atom$.
 	\item $\subo{\Box A}:=\{\Box A\}$.
 	\item $\subo{B\circ C}:=\{B\circ C\}\cup\subo{B}\cup\subo{C}$ for $\circ\in\{\vee,\wedge,\to\}$.
\end{itemize}
Also, for an arbitrary set $\Gamma$ of propositions, define
\begin{itemize}
\item $\Gamma^\vee:=\{\bigvee \Delta:\Delta\fsubeq \Gamma \text{ and }\Delta\neq\emptyset\}$.
\item $\dar{\sft}{\Gamma}:=$ the set of all $\Gamma$-projective propositions in the logic $\sft$.
	Whenever $\sft$ may be inferred from context, we may omit $\sft$ and simply write $\darl \Gamma$ (see \Cref{Gam-projec-modal} for more details).
	\item $\Gamma(X)$ denotes the set $\Gamma\cap\lcalz(X)$. Also let $\Gamma(\Box):=\Gamma(\parb)$.
\end{itemize}
Also define
\begin{itemize}
	\item $\sfs:= \{\sggt B: B\in\lcalb\}$.
	\item $\tcom:=\{A\in\lcalb: \tvdash A\to\Box A\}$.
	If $\sft$ may be inferred from context, we may omit the argument $\sft$ from the notation.
	\item $\prim\sft:=\Prim\sft:=$ the set of all $\sft$-prime propositions; i.e., the set of propositions $A$ such that for every $B,C$ with $\tvdash A\to (B\vee C)$, either $\tvdash A\to B$ or $\tvdash A\to C$.
	If the logic $\sft$ can be inferred from context, we may omit the argument $\sft$ from the notation.
	\item Given a set $\Gamma$ of propositions, define $\nfz\Gamma$ as the set of propositions $B\in\lcalb$ such that either $B\in\Gamma$ or $\Boxdot B\in\Gamma$. Then define the set $\nf\Gamma$ of propositions in \textit{$\Gamma$-Normal Form} as follows:
\emli
$$\nf\Gamma:=\{A\in\lcalb: \forall\, \Box B\in\sub A\ B\in\nfz\Gamma\}.$$
\end{itemize}
Finally, we assume that $(.)^\vee$ has the lowest precedence. This means that
\emli
$$
\cdsnbv:=(\sfc\darl\snnil)^\vee.
$$

\begin{lemma}\label{cdspnb-prime}
If $\Gamma$ is a set of $\sft$-prime propositions, then ${\downarrow}\Gamma$ is also a set of $\sft$-prime propositions.
\end{lemma}
\begin{proof}
	Let $E\in{\downarrow}\Gamma$ such that $\tvdash E\to (B\vee C)$. Also assume that $E\xrat{\theta}{\sft}E^\dagger\in\Gamma$. Hence $\tvdash E^\dagger\to(\th(B)\vee\th(C))$, and since $E^\dagger$ is $\sft$-prime, either $\tvdash E^\dagger\to\th(B)$ or $\tvdash E^\dagger\to \th(C)$. Hence either $\sft,E\vdash \th(E\to B)$ or $\sft,E\vdash \th(E\to C)$. Since $\theta$ is $E$-projective, we have either $\sft,E\vdash E\to B$ or $\sft,E\vdash E\to C$.
\end{proof}

\subsection{$\NNIL$ propositions}\label{sec-NNIL-def}
The class of \textit{No Nested Implications to the Left} ($\NNIL$) formulas for the non-modal language $\lcalz$ was introduced in \citep{Visser-Benthem-NNIL} and further explored in \citep{Visser02}. Here we summarize the necessary results from \citep{Visser-Benthem-NNIL,Visser02}. For simplicity, we may write ${\sf N}$ for $\NNIL$. A crucial result of \citep{Visser02} is an algorithm that, given $A\in\lcalz$, returns its best $\NNIL$ approximation $A^*$ from below; i.e., $\vdash A^*\ra A$, and for all $\NNIL$ formulas $B$ such that $\vdash B\ra A$, we have $\vdash B\ra A^*$. Moreover, for all $\Sigma_1$-substitutions $\sigma$, we have $\HA\vdash \sigma_{_{\sf HA}}(\Box A\lr \Box A^*)$ \citep{Visser02}. The classes $\NNIL$ and $\NI$ of propositions in $\lcalb$ are defined inductively:
\begin{itemize}
\item $A\in\NNIL$ and $A\in\NI$ for every $A\in\atomb$.
\item $B\circ C\in \NNIL$ if $B,C\in\NNIL$. Also $B\circ C\in \NI$ if $B,C\in\NI$ ($\circ\in\{\vee,\wedge\}$).
\item $B\to C\in\NNIL$ if $B\in\NI$ and $C\in\NNIL$.
\end{itemize}
\begin{theorem}\label{int-modal-submod-nnil}
Let $A\in\NNIL$ and let $\kcal'\leq\kcal$ be two Kripke models.
Then $\kcal\Vdash A$ implies $\kcal'\Vdash A$.
\end{theorem}
\begin{proof}
Let $\kcal=(W,\pce,\R,\V)$ and $\kcal'=(W,\pce',\R,\V)$. First, by induction on $A\in\NI$, show that for every $w\in W$, $\kcal,w\Vdash A$ iff $\kcal',w\Vdash A$. Then, by induction on $A\in\NNIL$, show that for every $w\in W$, if $\kcal,w\Vdash A$, then $\kcal',w\Vdash A$.
\end{proof}

Recall that (by \cref{notation-set}) $\NNIL(X)$ is the set of $\NNIL$ propositions that are also Boolean combinations of formulas in $X$.
\begin{lemma}\label{Remark-NNIL-finiteness}
Modulo $\IPC$-provable equivalence, $\NNIL(X)$ is finite whenever $X$ is a finite set of atomic or boxed formulas. Moreover, the set of all $\IPC$-provably equivalent formulas in this set is decidable.
\end{lemma}
\begin{proof}
Decidability follows from the decidability of $\IPC$ and the following argument for the finiteness of $\NNIL(X)$; it is left to the reader. To show that $\NNIL(X)$ is finite, we will find an upper bound $f(n,m)$ for the number of $(\IPC+Y)$-inequivalent propositions in $\NNIL(X)$, where $n$ is the number of elements in $X$ and $m$ is the number of elements in $Y\subseteq X$ (meaning $0\leq m\leq n$).

First, observe that each proposition in $\NNIL(X)$ can be written as $\bigvee\bigwedge C$, where $C$ either belongs to $X$ or is an implication $D\to E$ with $D\in X$ and $E\in\NNIL(X)$. We call such $C$ a component. Hence, $f(n,m)$ is at most $2^{2^{g(n,m)}}$, where $g(n,m)$ is the number of $\IPC+Y$-inequivalent components in $\NNIL(X)$ with $n$ and $m$ as the number of elements in $X$ and $Y$, respectively. Obviously, $g(n,n)=1$. For $m<n$, observe that $g(n,m)\leq (n-m)f(n,m+1)+ n-m$, because one may assume that each component $C$ is either of the form $E\to A$ for some $E\in X\setminus Y$ and some $A$ in $\NNIL(X)$ (modulo $(\IPC+Y+E)$-inequivalence), or $C\in X\setminus Y$. Hence, the following (primitive) recursive function is an upper bound for the number of all formulas in $n$ atomics:
\emli
\[
f(n,n):= 1
\quad \quad , \quad
\quad
f(n,m):=2^{2^{(n-m)(f(n,m+1)+1)}}.\qedhere
\]
\end{proof}

\subsection{Admissibility and preservativity}\label{pres-admis}
Given a logic $\sft$, we say that an inference rule \Ax{$A$}\UI{$B$}\DP is \textit{admissible} to $\sft$ if $\tvdash \theta(A)$ implies $\tvdash\theta(B)$ for every substitution $\theta$. Characterizing all admissible rules for classical logic is trivial: \Ax{$A$}\UI{$B$}\DP is admissible iff $A\to B$ is classically valid. However, the case for intuitionistic logic $\IPC$ or modal logics like ${\sf K4}$ is non-trivial (see \cite{Iemhoff-admissibility,iemhoff2009proof,jevrabek2005admissible, Rybakov87,Rybakov_1987,Rybakov_Book,goudsmit2014unification, iemhoff2016consequence}). In this paper, we deal with a generalization of admissibility: admissibility relative to a set $\Gamma$. This generalization was considered in \citep{PLHA0} for the propositional language, and here we extend it to the modal language.
Given a logic $\sft$ and a set $\Gamma$ of propositions, define
\begin{center}
	$A\argt B$ iff for every substitution $\theta$ and every $C\in\Gamma(\parb)$: $\tvdash \th(C\to A)$ implies $\tvdash \th(C\to B)$.
\end{center}

Also, we define the binary relation $\prtg$, the \textit{preservativity} relation, as follows:
\emli
$$A\prtg B \quad \text{ iff } \quad \forall\,E\in\Gamma(\tvdash E\to A \Rightarrow \tvdash E\to B).$$
\citep{Iemhoff.Preservativity} studies preservativity in the first-order language of arithmetic and axiomatizes it via the binary modal operator $\rhd$. Also, \citep{Visser02} studies preservativity for the propositional non-modal language and, among other results, axiomatizes $\pres{\ipc}{\NNIL}$.
\begin{remark}
By definition, it can be inferred that $A\argt B$ implies $A\prtg B$ whenever $\Gamma\subseteq\lcalz(\parb)$; however, the converse may not hold. As a counterexample, let $A$ and $B$ be two different variables, $\Gamma:=\{\top\}$, and $\sft=\IPC$. Then we have $A\prtg B$ but not $A\argt B$.
\end{remark}

\begin{lemma}\label{adsm-pres}
Let $\sft$ be a logic that is closed under outer substitutions. Then $A\argt B$ implies $A\prtg B$.
\end{lemma}
\begin{proof}
Let $A\argt B$ and $E\in{\downarrow}\Gamma$ with $\tvdash E\to A$. Assume that $E\xrat\theta\sft E^\dagger\in \Gamma$. Since $\sft$ is closed under outer substitutions, we get $\tvdash E^\dagger \to \th(A)$. Then $A\argt B$ implies $\tvdash E^\dagger\to\th(B)$. Hence $\tvdash E\to \th(E\to B)$, and because $E\xrat\theta\sft E^\dagger$, we get $\tvdash E\to B$.
\end{proof}

Later in this paper, we axiomatize $\prtg$ and $\argt$ for several pairs $(\sft,\Gamma)$. Before continuing, let us introduce some basic axioms and rules.

Given a logic $\sft$, the logic $\BART$ proves statements $A\rhd B$ for $A$ and $B$ in the language of $\sft$ and has the following axioms and rules:
\\[4mm]
\textbf{Axioms}
\begin{itemize}[leftmargin=1.5cm]
	\item[${\sf Ax}:$] \quad $A\rhd B$, for every $\tvdash A\to B$.
\end{itemize}
\textbf{Rules}
\begin{center}
	\bgroup
	\def\arraystretch{3}
	\begin{tabular}{c c}
		\Ax{$A\rhd B$}
		\Ax{$A\rhd C$}
		\RLa{Conj}
		\BI{$A\rhd B\wedge C$}
		\DP \quad \quad \quad \quad \quad \quad \quad \quad \quad
		&
		\Ax{$A\rhd B$}
		\Ax{$B\rhd C$}
		\RLa{Cut}
		\BI{$A\rhd C$}
		\DP \quad \quad \quad
	\end{tabular}
	\egroup
\end{center}
The above axiom and rules are not interesting because $\BART\vdash A\rhd B$ iff $\tvdash A\to B$. However, we define several interesting additional rules and axioms as follows. Let $\Delta\subseteq\lcalb$ and define
\begin{itemize}[leftmargin=1.5cm]
	\item[$\Les$:] \quad $A\rhd \Box A$ for every $A\in\lcalb$.
	\item[$\Lem$:] \quad $A\rhd \Box A$ for every $A\in\lcalz(\parb)$.
	\item[$\VAB$:] \quad $A\rhd \ov{\theta}(A)$, for every substitution $\theta$.
	\item[$\viss\Delta:$]\quad $B\to C\rhd \bigvee_{i=1}^{n+m} \itv{B}{E_i}\Delta$, where $B=\bigwedge_{i=1}^n (E_i\to F_i)$ and $C=\bigvee_{i=n+1}^{n+m} E_i$,
\end{itemize}
\emli
$$\itv{A}{B}\Delta:=\begin{cases}
B \quad &: B \in\Delta \\
A\to B &: \text{otherwise}
\end{cases}
$$
\begin{center}
	\bgroup
	\def\arraystretch{3}
	\begin{tabular}{c c}
		\quad \quad \quad \Ax{$B\rhd A$}
		\Ax{$C\rhd A$}
		\RLa{Disj}
		\BI{$B\vee C\rhd A$}
		\DP
		&
		\quad \quad \quad
		\Ax{$A\rhd B$}
		\Ax{($C\in\Delta$)}
		\RLa{$\montd$}
		\LLa{}
		\BI{$C\to A\rhd C\to B$}
		\DP
	\end{tabular}
	\egroup
\end{center}
The above axioms and rules have been introduced previously in \citep{IemhoffT,Visser02}, except for $\VAB$, which appears to be new. Also, $\viss\Delta$ and $\montd$ are generalizations of those introduced in \citep{IemhoffT,Visser02}. Finally, define
\emli
$$\ARD\sft\Delta:=\BART+\text{Disj}+\mont(\Delta)+\viss\Delta,$$
$$\ARDP\sft\Delta:= \ARD\sft\Delta+\Le \quad\text{and}\quad \ARDM\sft\Delta:= \ARD\sft\Delta+\Lem.$$
\begin{remark}\label{ASC}
$\sft\subseteq\sft'$ and $\Delta\subseteq\Delta'$ implies $\ARD\sft\Delta\subseteq\ARD{\sft'}{\Delta'}$.
\end{remark}
\begin{proof}
By induction on the complexity of a proof $\ARD\sft\Delta\vdash A\rhd A'$, one must show $\ARD{\sft'}{\Delta'}\vdash A\rhd A'$. We only treat the case $\viss\Delta$ here and leave the rest to the reader. So assume that $A=B\to C$ and $A'=\bigvee_{i=1}^{n+m} \itv{B}{E_i}\Delta$, where $B=\bigwedge_{i=1}^n (E_i\to F_i)$ and $C=\bigvee_{i=n+1}^{n+m} E_i$. Since for every $D,F\in \lcalb$, we have $\vdash \itv{D}{F}{\Delta'}\to \itv{D}{F}{\Delta}$, we get $\ARD{\sft'}{\Delta'}\vdash \bigvee_{i=1}^{n+m} \itv{B}{E_i}{\Delta'}\rhd \bigvee_{i=1}^{n+m} \itv{B}{E_i}\Delta$. On the other hand, by $\viss{\Delta'}$, we have $\ARD{\sft'}{\Delta'}\vdash A\rhd \bigvee_{i=1}^{n+m} \itv{B}{E_i}{\Delta'}$. Thus, Cut implies the desired result.
\end{proof}
\begin{theorem}[\textbf{Soundness for Preservativity}]\label{gen-pres-sound}
	$\BART$ is sound for preservativity interpretations; i.e., $\BART\vdash A\rhd B$ implies $A\prtg B$ for every set $\Gamma$ of propositions and every logic $\sft$. Moreover,
	\begin{enumerate}
		\item if $\Gamma$ is $\sft$-complete, then $\Les$ is sound,
		\item if $\Gamma$ is $\sft$-prime, then Disj is also sound,
		\item if $\Gamma$ is closed under $\Delta$-conjunctions (i.e., $A\in\Gamma$ and $B\in\Delta$ implies $A\wedge B\in \Gamma$, up to $\sft$-provable equivalence), then $\montd$ is sound.
		\item if $\sft$ has the intuitionistic submodel property, $\Gamma\subseteq\NNIL$, and $\Delta\subseteq\atomb$, then $\viss\Delta$ is sound.
		\item if $\Gamma\subseteq\lcalz(\parb)$ and $\sft$ is closed under outer substitutions, then $\VAB$ is also sound.
	\end{enumerate}
\end{theorem}
\begin{proof}
Easy induction on the complexity of a proof $\BART\vdash A\rhd B$, left to the reader. The validity of item 4 is provided by \Cref{itv-closure-pres}.
\end{proof}

\begin{theorem}[\textbf{Soundness for Admissibility}]\label{gen-admis-sound}
	$\BART$ is sound for admissibility interpretations; i.e., $\BART\vdash A\rhd B$ implies $A\argt B$ for every set $\Gamma$ of propositions and every logic $\sft$ that is closed under outer substitutions. Moreover,
	\begin{enumerate}
		\item if $\Gamma$ is $\sft$-complete, then $\Lem$ is sound,
		\item if $\Gamma$ is $\sft$-prime, then Disj is also sound.
		\item if $\Gamma$ is closed under outer substitutions of $\Delta$-conjunctions (i.e., $A\in\Gamma$ and $B\in\Delta$ implies $A\wedge \th (B)\in \Gamma$, up to $\sft$-provable equivalence), then $\montd$ is sound.
		\item if $\sft$ has the intuitionistic submodel property, $\Gamma\subseteq\NNIL$, and $\Delta\subseteq\parb$, then $\viss\Delta$ is sound.
	\end{enumerate}
\end{theorem}
\begin{proof}
Easy induction on the complexity of a proof $\BART\vdash A\rhd B$, left to the reader. The validity of item 4 is provided by \Cref{itv-closure-pres}.
\end{proof}
\begin{lemma}\label{itv-closure-pres}
Let $\sft$ have the intuitionistic submodel property, $\Gamma\subseteq\NNIL$, and $\Delta\subseteq\atomb$. Then
$B\to C\prtg \bigvee_{i=1}^{n+m} \itv{B}{E_i}\Delta$, where $B=\bigwedge_{i=1}^n (E_i\to F_i)$ and $C=\bigvee_{i=n+1}^{n+m} E_i$. Moreover, if $\Delta\subseteq\parb$, then $B\to C\artg \bigvee_{i=1}^{n+m} \itv{B}{E_i}\Delta$.
\end{lemma}
\begin{proof}
We argue by contradiction to show that $B\to C\prtg \bigvee_{i=1}^{n+m} \itv{B}{E_i}\Delta$, and leave the similar argument for $B\to C\artg \bigvee_{i=1}^{n+m} \itv{B}{E_i}\Delta$ to the reader. Fix some class $\scrmt$ of rooted Kripke models with the intuitionistic submodel property for which $\sft$ is sound and complete. Let $E\in\Gamma$ be such that $\sft\nvdash E\to (\bigvee_{i=1}^{n+m} \itv{B}{E_i}\Delta)$. Hence, there exists some $\kripke\in\scrmt$ with root $w_0$ such that $\kcal,w_0\Vdash E$ and $\kcal,w_0\nVdash \bigvee_{i=1}^{n+m} \itv{B}{E_i}\Delta$. Let $I$ be the set of indices $i$ such that $E_i\in\Delta$. Also, let $J$ be the complement of $I$. Thus, for every $i\in I$, we have $\kcal,w_0\nVdash E_i$, and for every $j\in J$, there exists some $w_j\sce w_0$ such that $\kcal,w_j\Vdash B$ and $\kcal,w_j\nVdash E_j$. Define $\pce'$ on $W$ as follows:
\emli
$${\pce'}:={\pce}\setminus (\{w_0\}\times \{v\in W: \neg\exists j\in J(w_j\pce v)\})$$
and define $\kcal':=(W,\pce',\R,\V)$. Then, since $E\in\NNIL$ and $\kcal,w_0\Vdash E$, \Cref{int-modal-submod-nnil} implies $\kcal',w_0\Vdash E$. Moreover, it is not difficult to observe that $\kcal',w_0\nVdash E_i$ for every $i\in I\cup J$. Hence, $\kcal',w_0\Vdash B$ and $\kcal',w_0\nVdash C$. Thus, $\kcal',w_0\nVdash E\to (B\to C)$. Since $\scrmt$ is assumed to have the intuitionistic submodel property, we may conclude $\kcal'\in\scrmt$, and hence $\sft\nvdash E\to (B\to C)$.
\end{proof}

\begin{theorem}\label{vee-pres}
${\prtg}={\prtgv}$ and ${\artg}={\artgv}$.
\end{theorem}
\begin{proof}
	We only show $A\prtg B$ iff $A\prtgv B$ and leave similar arguments for the other statements to the reader. The right-to-left direction holds since $\Gamma\subseteq \Gamma^\vee$. For the other direction, assume that $A\prtg B$ and let $E\in\Gamma^\vee$ such that $\tvdash E\to A$. Then $E=\bigvee_i E_i$ with $E_i\in\Gamma$. Hence, for every $i$, we have $\tvdash E_i\to A$. Then $A\prtg B$ implies $\tvdash E_i\to B$. Thus $\tvdash E\to B$, as desired.
\end{proof}
\subsection{Greatest lower bounds}\label{glb}
Given a set $\Gamma\cup\{A\}\subseteq\lcalb$ and a logic $\sft$, we say that $B$ is a \textit{$(\Gamma,\sft)$-lower bound} for $A$ if:
\begin{enumerate}
	\item $B\in\Gamma$,
	\item $\tvdash B\to A$.
\end{enumerate}
Moreover, we say that $B$ is the \textit{$(\Gamma,\sft)$-greatest lower bound} ($(\Gamma,\sft)$-glb) for $A$ if for every $(\Gamma,\sft)$-lower bound $B'$ for $A$, we have $\tvdash B'\to B$. Note that, up to $\sft$-provable equivalence, such a glb is unique, and we denote it by $\ap{\Gamma}{\sft}{A}$.
\\
We say that $(\Gamma,\sft)$ is \textit{downward compact} if every $A\in\lcalb$ has a $(\Gamma,\sft)$-glb $\ap\Gamma\sft A$. If $\ap\Gamma\sft A$ can be effectively computed, we say that $(\Gamma,\sft)$ is \textit{recursively downward compact}. A main result in \citep{Visser02} states that $(\NNIL,\IPC)$ is recursively downward compact (see Section 7 in \citep{Visser02}). We say that $(\Gamma,\sft)$ is \textit{(recursively) strong downward compact} if it is (recursively) downward compact and for every $\Box B\in\sub{\ap\Gamma\sft A}$, either $\Box B\in\sub A$ or $B\in \nfz\Gamma$.

\begin{question}
	One may similarly define the notion of least upper bound and upward compactness. What can we say about the (recursive) upward compactness of $(\NNIL,\IPC)$?
\end{question}


\begin{theorem}\label{Gamma-approx-preserv}
	$B$ is the $(\Gamma,\sft)$-glb for $A$ iff
	\begin{itemize}
		\item $B\in\Gamma$,
		\item $\tvdash B\to A$,
		\item $A\prtg B$.
	\end{itemize}
	Hence, we have $A\prtg \ap\Gamma\sft A$.
\end{theorem}
\begin{proof}
	Left to the reader.
\end{proof}

\begin{corollary}\label{Cor-Gamma-approx-preserv}
	If $\ap\Gamma\sft A$ exists, then for every $B\in\lcalb$, we have
\emli
$$\tvdash \ap\Gamma\sft A \to B\quad \text{ iff } \quad A\prtg B.$$
\end{corollary}
\begin{proof}
	First, assume that $\tvdash \ap\Gamma\sft A$. Also, let $E\in\Gamma$ such that $\tvdash E\to A$. \Cref{Gamma-approx-preserv} implies $A\prtg\ap\Gamma\sft A$, and hence $\tvdash E\to \ap\Gamma\sft A$. Then, by $\tvdash \ap\Gamma\sft A\to B$, we get $\tvdash E\to B$, as desired.
	\\
	For the other direction, let $A\prtg B$. By definition, we have $\ap\Gamma\sft A\in\Gamma$ and $\tvdash \ap\Gamma\sft A\to A$. Hence, by $A\prtg B$, we get $\tvdash \ap\Gamma\sft A\to B$, as desired.
\end{proof}

\begin{question}
	As we saw in \Cref{Gamma-approx-preserv}, the glb may be expressed via the preservativity relation $\prtg$. One may consider its variant which is best suited for lub's:
	$$A\prtgp B \quad \text{ iff } \quad \forall\,E\in\Gamma(\tvdash A\to E \Rightarrow \tvdash B\to E).$$
	We ask for an axiomatization of $\prtgp$ when $\sft=\IPC$ and $\Gamma=\NNIL$.
\end{question}

We also inductively define $\iap\Gamma\sft A$ (for downward compact $(\Gamma,\sft)$):
\begin{itemize}
\item $\iap\Gamma\sft a=a$, for atomic $a$.
\item $\iap\Gamma\sft{ \ }$ commutes with $\{\vee,\wedge,\to\}$.
\item $\iap\Gamma\sft{\Box A}:=\Box\ap\Gamma\sft{\iap\Gamma\sft A}$.
\end{itemize}
Note that in the above definition, we assumed $(\Gamma,\sft)$ to be downward compact to guarantee the existence of $\ap\Gamma\sft{\iap\Gamma\sft A}$.
\begin{lemma}\label{sdc}
If $(\Gamma,\sft)$ is strong downward compact, $\sft\supseteq\ik$, and $\sft$ is closed under necessitation, then for every $A\in\lcalb$, we have $\iap\Gamma\sft A\in\nf{\Gamma}$ and $\sft\visgt\vdash A\lr \iap\Gamma\sft A$.
\end{lemma}
\begin{proof}
We use induction on the complexity of $A$. All cases are trivial except for $A=\Box B$. Then, by definition, we have $\iap\Gamma\sft {\Box B}=\Box\ap\Gamma\sft{\iap\Gamma\sft B}$. By the induction hypothesis, we have $\iap\Gamma\sft B\in\nf\Gamma$ and $\sft\visgt\vdash B\lr \iap\Gamma\sft B$. Moreover, by \cref{Gamma-approx-preserv}, we have ${\iap\Gamma\sft B}\prtg \ap\Gamma\sft{\iap\Gamma\sft B}$. Hence, $\Box {\iap\Gamma\sft B}\to \Box \ap\Gamma\sft{\iap\Gamma\sft B} \in\visgt$. Also, by the definition of $\ap\Gamma\sft{ \,}$, we have $\sft\vdash \ap\Gamma\sft{\iap\Gamma\sft B} \to {\iap\Gamma\sft B}$, and since $\sft$ is closed under necessitation and $\sft\supseteq \ik$, we have $\sft\vdash \Box \ap\Gamma\sft{\iap\Gamma\sft B} \to \Box {\iap\Gamma\sft B}$. Thus, $\sft\visgt\vdash \iap\Gamma\sft {\Box B} \lr \Box {\iap\Gamma\sft B}$. Therefore, by the induction hypothesis (using $\sft\supseteq\ik$ and its closure under necessitation), $\sft\visgt\vdash \iap\Gamma\sft {\Box B} \lr \Box { B}$.

It remains only to show that $\ap\Gamma\sft{\iap\Gamma\sft B} \in\nf\Gamma$, which holds by $\iap\Gamma\sft B\in\nf\Gamma$ (induction hypothesis) and the \textit{strong} downward compactness of $(\Gamma,\sft)$.
\end{proof}

\subsection{Modal logics with binary modal operator}
\label{sec-preser}
Modal logics with a binary modal operator have been studied in the provability logic literature for at least two intended meanings of $A\rhd B$: (1) $\sft+B$ is interpretable in $\sft+A$ \citep{Visser-Interpretability,Berarducci,Velt-Jongh,goris2011new}, and (2) $\sft\vdash E\to A$ implies $\sft\vdash E\to B$ for every $\Sigma_1$ sentence $E$ \citep{Iemhoff.Preservativity,Iemhoff2005}. The first is considered for classical theories like $\sft$, and the second for intuitionistic theories like $\HA$. \citep{IemhoffT,Iemhoff.Preservativity} introduces the logic $\iph$ (defined slightly differently in this section) over the language $\lcalp$. Note that both definitions (the one presented here and the one defined by Iemhoff) coincide if we consider the language without parameters ($\parr=\emptyset$). Visser and de Jongh prove \citep{Iemhoff.Preservativity} that $\iph$ is sound for arithmetical interpretations in $\HA$ and conjecture that $\iph$ is also complete for such interpretations. We also conjecture that $\iphs$ (as defined in this subsection) is the $\Sigma_1$-preservativity logic of $\HA$ for $\Sigma_1$-substitutions; i.e., $\iphs\vdash A$ iff for every $\Sigma_1$-substitution $\sigma$, we have $\HA\vdash \sha (A)$ (see \Cref{sec-sub-1}).

Define $\LARD\sft\Delta$ as a logic in the language $\lcalp$ with the following axioms and rules:
\subsubsection*{Axioms}
\begin{itemize}[leftmargin=1.7cm]
	\item[\sft:] \quad All theorems of $\sft$.
	\item[$\viss\Delta:$]\quad $B\to C\rhd \bigvee_{i=1}^{n+m} \itv{B}{E_i}\Delta$, where $B=\bigwedge_{i=1}^n (E_i\to F_i)$ and $C=\bigvee_{i=n+1}^{n+m} E_i$.
	\item[$\amontd$:] \quad $A\rhd B\to (C\to A)\rhd (C\to B)$ for every $C\in\Delta$.
	\item[$\Le$:] \quad $A\rhd \Box A$ for every $A$.
	\item[ADisj:] \quad $(B\rhd A \wedge C\rhd A)\to (B\vee C)\rhd A$.
	\item[AConj:] \quad $[(A\rhd B)\wedge(A\rhd C)]\to(A\rhd (B\wedge C))$.
	\item[ACut:] \quad $[(A\rhd B)\wedge(B\rhd C)]\to(A\rhd C)$.
\end{itemize}
\subsubsection*{Rules}
\begin{itemize}[leftmargin=1.5cm]
	\item[MP:] \quad $A$, $A\to B$ / $B$.
	\item[PNec:] \quad $A \to B$ / $A\rhd B$.
\end{itemize}
Also define $\LARDP\sft\Delta$ as $\LARD\sft\Delta$ plus the following axiom:
\begin{itemize}[leftmargin=1.5cm]
\item[${\sf 4p}$:] \quad $(B\rhd C)\to\Box (B\rhd C)$.
\end{itemize}
Then define $\iph:=\LARD\iglcp\parb$ and $\iphp:=\LARDP\iglcp\parb$. Finally, define $\iphs:=\LARD\iglca\atomb$.

\begin{remark}\label{Lob-pres}
The L\"ob preservativity principle $(\Box A\to A)\rhd A$ is derivable in $\iph$. This axiom was listed in the original axiomatization of $\iph$ in \citep{Iemhoff.Preservativity,Iemhoff2005}.
\end{remark}
\begin{proof}
Reason inside $\iph$. By $\Le$, we have $(\Box A\to A)\rhd\Box(\Box A\to A)$. 
Also, by L\"ob's axiom in $\iglcp$ and necessitation, we get 
$\Box(\Box A\to A)\rhd \Box A$. Thus, by ACut, we have $(\Box A\to A)\rhd \Box A$. 
Since $(\Box A\to A)\rhd (\Box A\to A)$, by the AConj axiom, we have $(\Box A\to A)\rhd (\Box A\wedge (\Box A\to A))$. By necessitation, we have $(\Box A\wedge (\Box A\to A))\rhd A$, and thus ACut implies $(\Box A\to A)\rhd A$.
\end{proof}
\begin{theorem}\label{iph-sound}
$\iphp$ is sound for strong arithmetical interpretations in $\HA$; i.e., for every arithmetical substitution $\alpha$ and $A\in\lcalp$ with $\iphp\vdash A$, we have $\HA\vdash \aha^+(A)$ (see \Cref{sec-sub-1}).
\end{theorem}
\begin{proof}
\citep{IemhoffT,Iemhoff.Preservativity} proves that $\iph\vdash A$ implies $\HA\vdash\aha(A)$. The same proof works for $\iph$ and $\ahap(A)$ as well. Also, the validity of ${\sf 4p}$ for strong interpretations is obvious.
\end{proof}

\subsection{Simultaneous fixed-point theorem}\label{sec-fix}
It is well-known that the G\"odel-L\"ob logic $\GL$ proves the fixed-point theorem; i.e., for every $A\in\lcalb$ and atomic $a$ such that $a$ only occurs inside the scope of boxes, there exists a unique (up to $\GL$-provable equivalence) fixed point for $A$ with respect to $a$; i.e., there exists a proposition $D$ such that
$$\GL\vdash A[a:D]\lr D.$$
Moreover, one may choose $D$ such that it contains only atomics appearing in $A$ other than $a$. A well-known extension of this result is the simultaneous fixed-point theorem, of which we state an intuitionistic version in the following theorem.

As we will see in this paper, we mainly deal with outer substitutions $\th$; i.e., we do not substitute variables inside the scope of boxes. The fixed-point theorem helps us (in the proof of \Cref{Reduction-llea-lleb}, which is a major step in the arithmetical completeness of $\iglcph$) to convert usual substitutions to outer substitutions.

\citep{Iemhoff2005} proves the fixed-point theorem for $\iglcp$. In the following theorem, we extend it to a simultaneous version, analogous to the classical case.
\begin{theorem}\label{Theorem-simultan-fixed-point}
Let $\bs{E}:=\{E_1,\ldots,E_m\}$ and $\bs{a}=\{a_1,\ldots,a_m\}$ such that every occurrence of $a_i$ in $E_j$ is in the scope of some $\Box$. Then there exists a substitution $\tau$ that is the simultaneous fixed point of $\abold$ with respect to $\bs{E}$ in $\iglcp$; i.e.,
\begin{itemize}
\item $\iglcp\vdash \tau(E_i)\lr \tau(a_i)$ for every $1\leq i\leq m$.
\item $\tau$ is the identity on every atomic $a\nin\abold$.
\item $\suba{\tau(a_i)}\subseteq(\suba{\bs{E}}\setminus\bs{a})$.
\end{itemize}
\end{theorem}
\begin{proof}
The classical syntactic proof for this result with $n=1$ is valid for its intuitionistic counterpart. We refer the reader to \citep[theorem 1.3.5]{Smorynski-Book} or \citep[section 8]{Boolos}. Then we may use induction on $m$ to prove the general case as follows. As induction hypothesis, assume that we already have the statement of this theorem for $m$ and prove it for $m+1$. So assume that $\bs{E}:=\{E_1,\ldots,E_m,E_{m+1}\}$ and $\abold:=\{a_1,\ldots,a_m,a_{m+1}\}$ are given. By the induction hypothesis, there exists some substitution $\tau'$ such that
\begin{itemize}
\item $\iglcp\vdash \tau'(E_i)\lr \tau'(a_i)$ for every $1\leq i\leq m$.
\item $\tau'$ is the identity on every atomic $a\nin(\abold\setminus a_{m+1})$.
\item $\suba{\tau'(a_i)}\subseteq\{a_{m+1}\}\cup(\suba{\bs{E}}\setminus\bs{a})$ for $1\leq i\leq m$.
\end{itemize}
Then there exists a fixed point $D$ for $\tau'(E_{m+1})$ with respect to $a_{m+1}$; i.e.:
$$\iglcp\vdash D\lr \tau'(E_{m+1})[a_{m+1}:D].$$
Finally, define $\tau$ as follows:
$$
\tau(a):=\begin{cases}
(\tau'(a))[a_{m+1}:D]\quad &:a=a_i\text{ for } 1\leq i\leq m\\
D &: a=a_{m+1}\\
a &:\text{otherwise}
\end{cases}.
$$
Then it is not difficult to observe that $\tau$ satisfies all required conditions for the simultaneous fixed point.
\end{proof}

\subsection{Two crucial results}\label{2salient}
There are two results that are crucial for the arithmetical completeness of the provability logic of $\HA$: (1) the characterization of the $\Sigma_1$-provability logic of $\HA$ (\Cref{Theorem-Sigma-Provability-HA}), and (2) \Cref{PLHA0}, which implies the characterization of $\adsm{\ipc}{\Gamma}\ $ and $\pres{\ipc}{\Delta}\ $ for $\Gamma:=\NNILpar$ and $\Delta :=\dar{\ipc}\Gamma$ (\Cref{cor-PLHA0}).
\begin{theorem}\label{Theorem-Sigma-Provability-HA}
	The $\Sigma_1$-provability logic of $\HA$ is $\lles$; i.e., $\lles\vdash A$ iff for every $\Sigma_1$-substitution $\alpha$, we have $\HA\vdash \aha(A)$.
\end{theorem}
\begin{proof}
Observe that by \Cref{summery-ikfourca-snnil}, 
$\lles$ is the same as ${\sf iH}_\sigma$ as defined in 
\citep{Sigma.Prov.HA}. Then we have desired result by Theorems 6.3 and 6.5 in \citep{Sigma.Prov.HA}.
\end{proof}


\begin{theorem}\label{PLHA0}
Given $A\in\lcalz$, there exists a finite set $\Pi\subseteq\dnnil(\subp A)$ such that
\begin{enumerate}
\item $\IPC\vdash \bigvee\Pi\to A$.
\item $\ARD\IPC\parr\vdash A\rhd\bigvee\Pi$ (for the definition of $\ARD\IPC\parr$, see \Cref{pres-admis}).
\item $\Pi$ is computable as a function of $A$. Moreover, for every $D\in\Pi$, the substitution $\theta$ with $D\xrat\theta\ipc D^\dagger\in\nnil(\subp A)$ is computable.
\item $\cto(B)\leq \cto(A)+1+\#\subp{A}$ for every $B\in\Pi$.
\item $\suba B\subseteq\suba A$ for every $B\in\Pi$.
\end{enumerate}
\end{theorem}
\begin{proof}
Theorems 3.12, 3.27, and 4.15 from \citep{PLHA0}.
\end{proof}
\begin{corollary}\label{cor-PLHA0}
Let $\sft:=\ipc$, $\Gamma:=\NNILpar$, and $\Delta:=\dar{\ipc}\Gamma$. Then for every $A,B\in\lcalz$,
\begin{center}
$A\pres{\sft}{\Delta}B$ iff $A\adsm{\sft}{\Gamma}B$ iff $\ARD\IPC\parr\vdash A\rhd B$.
\end{center}
\end{corollary}
\begin{proof}
$\ARD\IPC\parr\vdash A\rhd B$ implies $A\adsm{\sft}{\Gamma}B$: One must use induction on the complexity of a proof $\ARD\IPC\parr\vdash A\rhd B$. All cases are easy except for Disj and $\VAR$, for which we refer the reader to Lemma 4.5 in \citep{PLHA0}. A similar reasoning for modal logics is presented in \Cref{Lem-con12-AR} in this paper.
\\
$A\adsm{\sft}{\Gamma}B$ implies $A\pres{\sft}{\Delta}B$: \Cref{adsm-pres}.
\\
$A\pres{\sft}{\Delta}B$ implies $\ARD\IPC\parr\vdash A\rhd B$: Let $A\pres{\sft}{\Delta}B$ and let $\Pi$ be a set of propositions as provided by \Cref{PLHA0}. Then for every $E\in\Pi$, we have $\ipc\vdash E\to A$, and since $\Pi\subseteq\Delta$, we get $\ipc\vdash E\to B$. Thus, $\ipc\vdash\bigvee \Pi\to B$ and $\ARD\IPC\parr\vdash \bigvee\Pi\rhd B$. Since $\ARD\IPC\parr\vdash A\rhd\bigvee\Pi$, Cut implies $\ARD\IPC\parr\vdash A\rhd B$.
\end{proof}

\section{Preservativity and relative admissibility}\label{secPres}
This section studies the preservativity relation $\prtg$ and the relative admissibility relation $\artg$ for several pairs $(\Gamma,\sft)$, where $\Gamma\subseteq\lcalb$ is a set of propositions and $\sft$ is an intuitionistic modal logic. The primary application is to use these relations for the axiomatization of the provability logic of $\HA$ and to prove its decidability. More precisely, we establish the following results in this section. The main result in \Cref{modal-prime} is \Cref{prime-fact}, which shows that $\NNIL$ propositions can be equivalently represented as disjunctions of prime $\NNIL$ propositions. In the subsequent subsections, we characterize the following relations:
\begin{itemize}
\item \Cref{Sec-NNIL-Preserv}:
${\ARNBST}={\prtsn}$, whenever $\sft$ is \types.
\item \Cref{Pres-logic-ARNBT1}:
${\ARNBTM}={\prtdnb}= {\artnb}$, whenever $\sft$ is \typea.
\item \Cref{Pres-logic-ARNBT3}:
${\AARNBT}={\prtdsnb}={\artsnb}$, whenever $\sft$ is \typea.
\item \Cref{Pres-logic-ARNBT2}:
${\ARNBT}={\prtcdsnb}$, whenever $\sft$ is \typea.
\item \Cref{Pres-logic-ARNBTP}:
${\ARNBTP}={\prtsnb}$, whenever $\sft$ is \typea.
\item If $\sft$ is decidable, then all the above-mentioned relations are also decidable.
\end{itemize}
\noindent\textbf{Important convention:}
Throughout this section, we fix a set $\sft\supseteq\ipc$ of modal propositions that is closed under modus ponens, and a class $\scrmt$ of rooted Kripke models for which $\sft$ is sound and complete.
Thus, in the remainder of this section, we may omit the superscript $\sft$ from notations; e.g., we may use the shorter notations $\sfc$, $\sfP$, and $\dar{\,}\Gamma$ instead of $\sfc^\sft$, $\sfP^\sft$, and $\dar{\sft}\Gamma$, respectively.
\\
We also define three types of modal logics as follows.
\begin{itemize}
\item \typez: $\sft\supseteq\ikfourcp$ and $\scrmt$ has the extension property (see \Cref{Modal-ext-property}).
\item \typea: $\sft\supseteq\ikfourcp$, $\sft$ is closed under necessitation, $\scrmt$ has the extension property (see \Cref{Modal-ext-property}) and the intuitionistic submodel property (see \Cref{sec-Kripke}), and $\sft$ is closed under outer substitutions; i.e., $A\in \sft$ implies $\th(A)\in\sft$ for every substitution $\theta$.
\item \types: $\sft\supseteq\ikfourca$, $\sft$ is closed under necessitation, and $\scrmt$ has the extension property (see \Cref{Modal-ext-property}) and the intuitionistic submodel property (see \Cref{sec-Kripke}).
\end{itemize}
In the remainder of this section, the notation \uparan{$\typez$} in theorem statements means that we assume $\sft$ satisfies the conditions of $\typez$. We use similar notation for $\typea$ and $\types$.
\begin{remark}
$\ikfourcp$ and $\iglcp$ are \typea. Also, $\ikfourca$ and $\iglca$ are \types. The finite model property for $\ikfourcp$, $\iglcp$, $\ikfourca$, and $\iglca$ implies their decidability.
\end{remark}
\begin{proof}
Left to the reader.
\end{proof}
In the remainder of this paper, we may use the above remark without explicit mention.

\subsection{Prime factorization for $\NNIL$}
\label{modal-prime}
In this section, we prove that every $\NNIL$ proposition $A\in\lcalb$ can be decomposed into a disjunction $\bigvee_iA_i$ of $\sft$-prime $\NNIL$ propositions (\Cref{prime-fact}). We follow a route similar to that in Section 3.5 of \citep{PLHA0} for the non-modal language. To show this, we require two other equivalent notions: $\sft$-component and $\sft$-extendible. We first show that every $A$ can be decomposed into $\sft$-components (\Cref{Lem2-01}), then show that every $\sft$-component is $\sft$-extendible (\Cref{Lem2-0}), and finally show that every $\sft$-extendible proposition is $\sft$-prime (\Cref{extendible-to-prime}). We will see in \Cref{Prime-ext-comp} that these three notions are equivalent.

Recall that $\sfP$ denotes the set of all $\sft$-prime propositions; i.e.,
$$\sfP:=\{A\in\lcalb: \tvdash A\to(B\vee C)\text{ implies } \tvdash A\to B \text{ or } \tvdash A\to C\}.$$

\begin{theorem}\label{prime-fact}
\uparan{\typez}
Up to $\ikfourcp$-provable equivalence,
we have $\nnilb=\sfP\nnilb^\vee$ and $\snnilb=\sfs\sfP\nnilb^\vee$. Also, if $\sft\supseteq \ikfourca$, then up to $\ikfourca$-provable equivalence we have $\sfn=\sfP\nnil^\vee$ and $\snnil=\sfs\sfP\nnil^\vee$.
\end{theorem}
\begin{proof}
	Direct consequence of \Cref{Lem2-01,Lem2-0,extendible-to-prime}.
\end{proof}
\begin{corollary}\label{prspn=prsn}
\uparan{\typez}
${\prtsnb}={\prtspnb}$, ${\artsnb}={\artspnb}$, ${\artnb}={\artpnb}$, and if $\sft\supseteq \ikfourca$, then ${\prtsn}={\prtspn}$.
\end{corollary}
\begin{proof}
	Direct consequence of \Cref{prime-fact,vee-pres}.
\end{proof}

In the remainder of this subsection, we present technical lemmas needed for the proof of the above theorem.

\subsubsection{$\sft$-components}
Given $A\in\NNIL$, we say that $A$ is an \textit{$\sft$-component} if $A=\bigwedge \Gamma\wedge\bigwedge\Delta $ with the following properties:
\begin{itemize}
	\item Every $B\in \Gamma$ is atomic or boxed.
	\item Every $B\in\Delta$ is an implication $C\to D$ for some atomic or boxed $C$ such that $\sft\nvdash \bigwedge\Gamma\to C$.
\end{itemize}
Define $\NNILP:=\{A\in\NNIL: E\to F\in\subo{A} \text{ implies } E\in\atomb\}$. In other words, $\NNILP$ includes all $\NNIL$ propositions such that every antecedent occurring outside of boxes is either atomic or boxed. Obviously, every $A\in\NNIL$ ($\SNNIL$) can be converted to some $A'\in\NNILP$ ($A'\in\SNNILP$) via derivability in $\IPC$ ($\ikfourcp$).
\begin{lemma}\label{Lem2-01-0}
	Given $A\in\NNILP$ \uparan{$A\in\SNNILP$}, there exists a finite set $\Gamma_A\subseteq\NNILP$ \uparan{$\Gamma_A\subseteq\SNNILP$} of $\sft$-components such that $\tvdash A\lr \bigvee\Gamma_A$ and $\suboab{\Gamma_A}\subseteq\suboab A$.
\end{lemma}
\begin{proof}
We use induction on $\suboab A$ (the set of atomic or boxed formulas occurring outside the scope of boxes in $A$), ordered by $\subset$, to find a finite set $\Gamma_A\subseteq\NNILP$ ($\Gamma_A\subseteq\SNNILP$) of $\sft$-components with $\suboab{\Gamma_A}\subseteq\suboab A$ and $\tvdash \bigvee\Gamma_A\lr A$.

As the induction hypothesis, assume that for every $\sft'\supseteq\IPC$ and $B\in\NNILP$ ($B\in\SNNILP$) with $\suboab B\subsetneqq \suboab A$, there exists a finite set $\Gamma_B\subseteq\NNILP$ ($\Gamma_B\subseteq\SNNILP$) of $\sft'$-components such that $\sft'\vdash B\lr \bigvee\Gamma_B$ and $\suboab{\Gamma_B}\subseteq\suboab B$.
For the induction step, assume that $A\in\NNILP$ ($A\in\SNNILP$) is given. Using derivability in $\ipc$, one can easily find finite sets $\Gamma_i$ and $\Delta_i$ such that
\begin{itemize}
\item $\ipc\vdash A\lr \bigvee_{i=1}^n A_i$, where $A_i:=\bigwedge\Gamma_i\wedge\bigwedge\Delta_i$.
\item $\Delta_i$ contains only atomic or boxed propositions.
\item $\Gamma_i$ contains implications with atomic or boxed antecedents.
\item $\suboab{A_i}\subseteq \suboab A$.
\item $A_i\in\NNILP$ ($A_i\in\SNNILP$).
\end{itemize}
It suffices to decompose each $A_i$ into $\sft$-components. If $\sft\nvdash \bigwedge\Delta_i\to E$ for every antecedent $E$ of an implication in $\Gamma_i$, then $A_i$ is already an $\sft$-component, and we are done. Otherwise, there exists some $E\to F\in\Gamma_i$ such that $\tvdash \bigwedge\Delta_i\to E$. Let $A_i':=A_i[E:\top]$; i.e., replace every outer occurrence of $E$ in $A_i$ with $\top$. Also, let $\sft':=\sft+E$. Hence $\suboab{A'_i}\subsetneqq \suboab{A}$, and by the induction hypothesis, we may decompose $A'_i$ into $\sft'$-components:
$$\sft,E\vdash A'_i\lr \bigvee_j B_j$$
It is not difficult to observe that if $B_j$ is an $\sft'$-component, then $B'_j:=E\wedge B_j$ is an $\sft$-component. Moreover, $\sft\vdash E\wedge A'_i\lr \bigvee_j B'_j$, and since $\IPC\vdash (E\wedge A'_i)\lr (E\wedge A_i)$ and $\tvdash A_i\to E$, we get $$\sft\vdash A_i\lr \bigvee_j B'_j$$
Thus, we have decomposed $A_i$ into $\sft$-components $B'_j$ with $\suboab{B'_j}\subseteq\suboab{A}$.
\end{proof}
\begin{corollary}\label{Lem2-01}
	Every $A\in\NNIL$ can be decomposed into $\sft$-components; i.e., there exists a finite set $\Gamma_A$ of $\sft$-components such that $\tvdash A\lr \bigvee\Gamma_A$. Moreover:
	\begin{enumerate}
	\item if $A\in\nnilb$, then $\Gamma_A\subseteq\nnilb$,
	\item if $A\in\snnil$ and $\sft\supseteq\ikfourcp$, then $\Gamma_A\subseteq\snnil$,
	\item if $A\in\snnilb$ and $\sft\supseteq\ikfourcp$, then $\Gamma_A\subseteq\snnilb$.
	\end{enumerate}
\end{corollary}
\begin{proof}
Easy consequence of \Cref{Lem2-01-0}, left to the reader. The condition $\sft\supseteq\ikfourcp$ is only needed to convert $A\in\SNNIL$ to some $A'\in\SNNIL^+$.
\end{proof}

\subsubsection{$\sft$-Extension property}\label{Modal-ext-property}
Let $\scrk$ and $\scrk'$ be two sets of rooted Kripke models. In the following, we define a Kripke model $\sum(\scrk,\scrk')$. Roughly speaking, it is obtained by adding a fresh root below $\scrk$ with $\R$-access to the roots of $\scrk'$. More precisely, define $\sum(\scrk,\scrk'):=(W,\pce,\R,\V)$ as follows. Assume that the sets of nodes of Kripke models in $\scrk\cup\scrk'$ are disjoint. Moreover, assume that every $\kcal\in\scrk\cup\scrk'$ is of the form $\kcal:=(W_\kcal,\pce_\kcal,\R_\kcal,\V_\kcal)$.
\begin{itemize}
\item $W:=\{w_0\}\cup\bigcup_{\kcal\in\scrk\cup\scrk'}W_\kcal$, where $w_0\nin W_\kcal$ for every $\kcal\in\scrk\cup\scrk'$.
\item $\pce:=(\bigcup_{\kcal\in\scrk\cup\scrk'}\pce_\kcal)\cup \{(w_0,w):\exists\kcal\in\scrk \ (w^\kcal_0\pce_\kcal w)\}$, where $w^\kcal_0$ is the root of $\kcal$.
\item $\R:= (\bigcup_{\kcal\in\scrk\cup\scrk'}\R_\kcal)\cup \{(w_0,w): \exists\kcal\in\scrk' \ (w^\kcal_0\sqsubseteq_\kcal w) \text{ or } \exists\kcal\in\scrk \ (w^\kcal_0\R_\kcal w)\}$.
\item $\V:=\bigcup_{\kcal\in\scrk\cup\scrk'}\V_\kcal$. Note that this means the fresh root $w_0$ has an empty valuation.
\end{itemize}
Two Kripke models are said to be \textit{variants} of each other if they share the same rooted frame and their valuations are identical except possibly at the root, where they may differ.
\\
A class $\scrm$ of rooted Kripke models is said to have the \textit{extension property} if for every finite set $\scrk\subseteq\scrm$, there exists a finite set of rooted Kripke models $\scrk'$ such that a variant of $\sum(\scrk,\scrk')$ belongs to $\scrm$. A proposition $A$ is said to have the \textit{$\scrm$-extension property} if $\scrm\cap\Mod A$ has the extension property, where $\Mod A$ is the class of all Kripke models of $A$.

\begin{lemma}\label{Lem2-0} \uparan{\typez}
Let $A\in\lcalb$ be an $\sft$-component such that $\sft\vdash a\to\Box a$ for every $a\in\suboa{A}$. Then $A$ has the $\scrmt$-extension property.
\end{lemma}
\begin{proof}
	Let $\scrk\subseteq\scrmt$ be a finite set of Kripke models of a component $A=\bigwedge \Gamma\wedge\bigwedge\Delta$. Let
	$$\Gamma=\{a_1,\ldots,a_n,\Box E_1,\ldots,\Box E_m\} \quad \text{ and } \quad \Delta=\{b_1\to G_1,\ldots,b_k\to G_k,\Box F_1\to H_1, \ldots \Box F_l\to H_l\},$$
	where $a_i,b_j\in\atom$. For every $i\leq l$ such that $\scrk\nVdash \Box F_i$, since $\sft\nvdash \bigwedge\Gamma\to \Box F_i$, there exists some $\kcal'_i\in\scrmt$ such that $\kcal'_i\Vdash \Gamma$ and $\kcal'_i\nVdash \Box F_i$. Thus, there exists some node $w$ such that $\kcal'_i,w\Vdash \Gamma'$, where $\Gamma':=\Gamma\cup\{E_1,\ldots,E_m\}$, and $\kcal'_i,w\nVdash F_i$ (note that here we are using $\kcal'_i\Vdash a_i\to\Box a_i$). Define $\kcal_i:=(\kcal_i')_w$; i.e., the Kripke model generated by $w$ in $\kcal'_i$. Let $\scrk':=\{\kcal_1,\ldots,\kcal_l\}$, and finally define $\kcal$ as a variant of $\sum(\scrk,\scrk')$ with the following valuation at the root:
	\emli
	$$\kcal,w_0\Vdash a\quad \text{iff}\quad a\in\Gamma.$$
	Then it is not difficult to observe that $\kcal,w_0\Vdash A$.
\end{proof}

\begin{lemma}\label{extendible-to-prime}
	Every $A\in\lcalz$ with the $\scrmt$-extension property is $\sft$-prime.
\end{lemma}
\begin{proof}
	Let $\sft\nvdash A\to B$ and $\sft\nvdash A\to C$. Then there exist Kripke models $\kcal_1,\kcal_2\in\scrmt$ such that $\kcal_1,\kcal_2\Vdash A$, $\kcal_1\nVdash B$, and $\kcal_2\nVdash C$. Since $A$ has the $\scrmt$-extension property, there exists a Kripke model $\kcal\in\scrmt$ that is an extension of $\{\kcal_1,\kcal_2 \}$ and $\kcal\Vdash A$. Since $\kcal_1\nVdash B$, we get $\kcal\nVdash B$. Similarly, $\kcal\nVdash C$. Hence $\kcal\nVdash A\to (B\vee C)$, and thus $\sft\nvdash A\to (B\vee C)$.
\end{proof}

\begin{theorem}\label{Prime-ext-comp}\uparan{\typez}
	Let 
	$A\in\NNIL$ with $\sft\vdash a\to \Box a$ for every $a\in\suboa{A}$. The following items are equivalent:
	\begin{enumerate}
		\item $A$ has the $\scrmt$-extension property.
		\item $A$ is $\sft$-prime.
		\item $A$ is an $\sft$-component (up to $\sft$-provable equivalence).
	\end{enumerate}
\end{theorem}
\begin{proof}
	\begin{itemize}
		\item $1\to 2$: \Cref{extendible-to-prime}.
		\item $2\to 3$: Given $A$, first decompose it into $\sft$-components, then use its $\sft$-primality to deduce that $A$ must be $\sft$-equivalent to one of its $\sft$-components.
		\item $3\to 1$: \Cref{Lem2-0}.
	\end{itemize}
\end{proof}


\subsection{$\snnil$-Preservativity}
\label{Sec-NNIL-Preserv}
 \citep{Visser02} axiomatizes the binary relation of $\NNIL$-preservativity for the non-modal language and $\ipc$ (\Cref{Theorem-Vis-NNIL-pres}). Here we perform a similar task for the modal language (\Cref{Theorem-NNIL-Pres}) and show that $\ARNBST$ axiomatizes $\prtsn$ when $\sft$ is $\types$. Moreover, we also show that $(\SNNIL,\sft)$ is recursively downward compact when $\sft$ is \types.

\begin{theorem}\label{Theorem-Vis-NNIL-pres}
	For every $A,B\in\lcalb$,
	$\ARNBSI\vdash A\rhd B$ iff $A \prin  B$. Moreover, for every $A\in\lcalb$, there exists some $A^*\in\NNIL$ that is effectively computable, $\IPC\vdash A^*\to A$, $\ARNBSI\vdash A\rhd A^*$, and $\suboab{A^*}\subseteq\suboab A$.
\end{theorem}
\begin{proof}
\citep{Visser02}.
\end{proof}

\begin{theorem}\label{SN-approx-prop}\uparan{\typet}
Given $A\in\lcalb$, we have $A^\sfh:=\sggt{(A^*)}=\ap\snnil\sft A$, and hence $(\snnil,\sft)$ is recursively strong downward compact (see \Cref{glb}).
\end{theorem}
\begin{proof}
First, we show that $A^\sfh:=\sggt{(A^*)}$ is the $(\snnil,\sft)$-glb for $A$, where $A^*$ is as provided by \Cref{Theorem-Vis-NNIL-pres}. By \Cref{Theorem-Vis-NNIL-pres}, $A^*\in\NNIL$ and $\IPC\vdash A^*\to A$. Also, we have $A^\sfh=\sggt{(A^*)}\in\snnil$ and $\ikfourca\vdash A^\sfh\to A^*$. \Cref{Cor-HNF-brts} implies $\ARNBST\vdash A\rhd A^\sfh$, and then \Cref{Lem-con1} implies $A\prtsn A^\sfh$. Thus, \Cref{Gamma-approx-preserv} implies $A^\sfh=\ap\snnil\sft A$.
\\
\Cref{Theorem-Vis-NNIL-pres} implies that $A^*$ is computable, and hence $\ap\snnil\sft A=A^\sfh$ is computable. Next, we reason for strongness. \Cref{Theorem-Vis-NNIL-pres} implies that $\suboab{A^*}\subseteq\suboab{A}$, and hence for every boxed subformula $\Box B$ of $A^\sfh:=\sggt{(A^*)}$, either $\Box B\in\sub A$ or $B\in \nfz\snnil$.
\end{proof}

\begin{theorem}\label{Theorem-NNIL-Pres}
\uparan{\typet}
$\ARNBST$ is sound and complete for $\snnil$-preservativity in $\sft$. More precisely, for every $A,B\in\lcalb$:
\emli	$$ \ARNBST\vdash A\rhd B\quad \text{ iff }\quad A\prtsn B.$$
\end{theorem}
\begin{proof}
	The left-to-right direction (soundness) holds by \Cref{Lem-con1}.
	\\
	For the other direction, let $A\prtsn B$. \Cref{SN-approx-prop,Cor-Gamma-approx-preserv} implies $\tvdash A^\sfh\to B$, and hence $\ARNBST\vdash A^\sfh\rhd B$. Also, by \Cref{Cor-HNF-brts}, we have $\ARNBST\vdash A\rhd A^\sfh$, and then Cut implies the desired result.
\end{proof}

\begin{corollary}\label{Coro-Pres-HNF}\uparan{\typet}
	$\sft$ is closed under $(.)^\sfh$; i.e., if $\tvdash A$, then $\tvdash A^\sfh$.
\end{corollary}
\begin{proof}
	Let $\tvdash A$, and hence $\tvdash \top\to A$. \Cref{Cor-HNF-brts,ASC} implies $\ARNBST\vdash A\rhd A^\sfh$. Then \Cref{Theorem-NNIL-Pres} implies $A\prtsn A^\sfh$, and hence $\tvdash \top\to A^\sfh$.
\end{proof}

\begin{lemma}\label{Cor-HNF-brts}
	For every $A\in\lcalb$, we have $\ARNBSK\vdash A\rhd A^\sfh$.
\end{lemma}
\begin{proof}
	First, note that by \Cref{Theorem-Vis-NNIL-pres}, we have $\ARNBSI \vdash A\rhd A^*$, and by \Cref{ASC}, $\ARNBSK \vdash A\rhd A^*$. On the other hand, \Cref{Lem-Box-translation-brts} implies $\ARNBSK\vdash A^*\rhd \sggt{(A^*)}$. Then Cut implies the desired result.
\end{proof}

\begin{lemma}\label{Lem-Box-translation-brts}
	For every $A\in\NNIL$, we have $\ARNBSK\vdash A\rhd \sggt A$.
\end{lemma}
\begin{proof}
	Use induction on the complexity of $A$:
	\begin{itemize}[leftmargin=*]
		\item $A$ is atomic. Then $\sggt A=A\wedge \Box A$, and by \Le we have $\ARNBSK\vdash A\rhd \Box A$. Since $\ARNBSK\vdash A\rhd A$, Conj implies the desired result.
		\item $A$ is boxed. Then $\sggt A=A$, and hence $\ikfourca\vdash A\to \sggt A$. Thus, $\ARNBSK\vdash A\rhd \sggt A$.
		\item $A=B\wedge C$. By the induction hypothesis, $\ARNBSK\vdash B\rhd \sggt B$ and $\ARNBSK\vdash C\rhd \sggt C$. On the other hand, $\tvdash (B\wedge C)\to B$ and $\tvdash (B\wedge C)\to C$. Hence $\ARNBSK\vdash (B\wedge C)\rhd B$ and $\ARNBSK\vdash (B\wedge C)\rhd C$, which by Cut and Conj yields $\ARNBSK\vdash (B\wedge C)\rhd (\sggt B\wedge \sggt C)$. Since $\sggt{(B\wedge C)}=\sggt B\wedge \sggt C$, we have the desired result.
		\item $A=B\vee C$. By the induction hypothesis, $\ARNBSK\vdash B\rhd \sggt B$ and $\ARNBSK\vdash C\rhd \sggt C$. On the other hand, $\tvdash B\to (B\vee C)$ and $\tvdash C\to (B\vee C)$. Hence $\ARNBSK\vdash B\rhd (B\vee C)$ and $\ARNBSK\vdash C\rhd (B\vee C)$, which by Disj and Cut yields $\ARNBSK\vdash (B\vee C)\rhd (\sggt B\vee \sggt C)$. Since $\sggt{(B\vee C)}=\sggt B\vee \sggt C$, we have the desired result.
		\item $A=B\to C$. By the induction hypothesis, we have $\ARNBSK\vdash C\rhd \sggt C$. Since $A\in\NNIL$, we have $B\in\NI$, and \Cref{Lem01} implies $\ARNBSK\vdash (B\to C)\rhd(B\to \sggt C)$. Since $\ikfourca\vdash B\lr \sggt B$, we get $\ARNBSK\vdash (B\to \sggt C)\rhd (\sggt B\to \sggt C)$. Finally, by \Le we get $\ARNBSK\vdash (\sggt B\to \sggt C)\rhd \Box(\sggt B\to \sggt C)$, and hence $\ARNBSK\vdash (B\to C)\rhd\sggt{(B\to C)}$, as desired. \qedhere
	\end{itemize}
\end{proof}

\begin{lemma}\label{Lem01}
	The following rule is admissible to $\ARNBSK$:
	\Ax{$A\rhd B$}
	\RLa{$E\in\NI$}
	\UI{$E\to A\rhd E\to B$}
	\DP.
\end{lemma}
\begin{proof}
	Use induction on the complexity of $E$.
	\begin{itemize}
		\item If $E$ is atomic or boxed, then it is the same as Montagna's rule.
		\item $E=E_1\wedge E_2$. Let $\ARNBSK\vdash A\rhd B$. Then by the induction hypothesis, we have $\ARNBSK\vdash (E_2\to A)\rhd(E_2\to B)$. Again by the induction hypothesis, we have $\ARNBSK\vdash (E_1\to (E_2\to A))\rhd (E_1\to (E_2\to B))$, and thus by Cut and $\ikfourca$, we get $\ARNBSK\vdash E\to A\rhd E\to B$.
		\item $E=E_1\vee E_2$. Let $\ARNBSK\vdash A\rhd B$. Then by the induction hypothesis, we have $\ARNBSK\vdash (E_1\to A)\rhd(E_1\to B)$ and $\ARNBSK\vdash (E_2\to A)\rhd(E_2\to B)$. Hence $\ARNBSK\vdash ((E_1\to A) \wedge (E_2\to A))\rhd ((E_1\to B) \wedge (E_2\to B))$. Thus $\ARNBSK\vdash E\to A\rhd E\to B$, as desired. \qedhere
	\end{itemize}
\end{proof}

\begin{theorem}[\textbf{Soundness}]\label{Lem-con1}
\uparan{\typet}
	$\ARNBST\vdash A\rhd B$ implies $A\prtsn B$.
\end{theorem}
\begin{proof}
Let $\ARNBST\vdash A\rhd B$. \Cref{gen-pres-sound} implies $A\prtspn B$, and then \Cref{prspn=prsn} implies $A\prtsn B$.
\end{proof}

In the following corollary, we consider $\ARNBST$ as the binary relation that it axiomatizes.
\begin{corollary}\label{summery-ikfourca-snnil}
\uparan{\typet}
The following equalities hold. Moreover, if $\sft$ is decidable, then all the above relations are decidable.
\emli $${\ARNBST}={\prtsn}={\prtsnv}={\prtspn}={\prtspnv}$$
\end{corollary}
\begin{proof}
\Cref{Theorem-NNIL-Pres,prspn=prsn,vee-pres} imply the equalities. For the decidability of $\prtsn$, we have the following argument. \Cref{SN-approx-prop} implies that $\ap\snnil\sft A$ exists and is computable. Then, by \Cref{Cor-Gamma-approx-preserv}, it suffices to decide $\tvdash \ap\snnil\sft A\to B$, which is provided by the decidability of $\sft$.
\end{proof}

\subsection{$\dnb$-Preservativity and $\nnilb$-Admissibility}
\label{Pres-logic-ARNBT1}

In this section, we show that $\ARNBTM$ axiomatizes $\prtdnb$ and $\artnb$ whenever $\sft$ is \typea. Moreover, we show that $(\dnbv,\sft)$ is recursively strong downward compact whenever $\sft$ is \typea.

\begin{lemma}\label{ARN-ARNBT}
	Let $\alpha$ be a general substitution such that for every $p\in\parr$, we have $\alpha(p)\in\parb$. Then for every $A,B\in\lcalz$:
\emli	$$\ARD\IPC\parr \vdash A\rhd B \quad \text{implies}\quad \ARNBIM\vdash \alpha(A)\rhd\alpha(B).$$
\end{lemma}
\begin{proof}
	Easy induction on the complexity of a proof $\ARD\IPC\parr\vdash A\rhd B$, left to the reader.
\end{proof}
\begin{definition}\label{def-glb-cdsnbv}\em
	For every $A\in\lcalb$, we define $\trz A,\tro A\in\lcalb$ and $\Trz A, \Tro A\subseteq\lcalb$ as follows. Let $\pvec:=p_1,\ldots,p_n$ include all parameters occurring in $A$, and let $\subb{A}=\{\Box B_1,\Box B_2\ldots,\Box B_m\}$. Also, let $\qbold:=q_1,q_2,\ldots,q_m$ be a list of fresh parameters; i.e., $q_i\nin\pvec$ for every $i$, and they are pairwise distinct. Let the substitution $\alpha$ be such that $\alpha(q_i)=\Box B_i$ and $\alpha(a)=a$ for every other atomic $a$. Given $B\in\lcalz(\subb{A},\pvec ,\varr)$, there exists a unique $B_0\in\lcalz(\qbold,\pvec ,\varr)$ such that $\alpha(B_0)=B$. In the remainder of this definition, we use $B_0$ to denote this unique proposition in $\lcalz(\qbold,\pvec ,\varr)$ for every $B\in \lcalz(\subb{A},\pvec ,\varr)$. By \Cref{PLHA0}, there exists a finite set $\Trz A\subseteq\dnnil(\pvec,\qvec)$ with the following properties:
	\begin{itemize}
	\item[P1:] $\IPC\vdash D_0\to A_0$ for every $D_0\in\Trz A$.
	\item[P2:] $\ARD\IPC\parr\vdash A_0\rhd \bigvee\Trz A$.
	\item[P3:] $\Trz A$ is a computable function of $A$. Moreover, for every $D_0\in\Trz A$, the substitution $\theta_0$ with $D_0\xrat{\theta_0}{\ipc}D^\dagger_0\in\nnilpq$ can be effectively computed.
	\item[P4:] $\cto(D_0)\leq \ctob(A)$ for every $D_0\in\Trz A$.
	\item[P5:] For every $D_0\in\Trz A$, we have $\suba{D_0}\subseteq \suba{A_0}=\pvec\cup\suboa{A}$.
	\end{itemize}
	Hence, for every $D_0\in\Trz A$, there exists some substitution $\theta_0$ such that $D_0\xrat{\theta_0}{}D^\dagger_0\in\nnilpq$. Thus, for $\theta:=\alpha\circ\theta_0$, we have $\alpha(D_0)=D\xrat{\theta}{\IPC}D^\dagger=\alpha(D_0^\dagger)\in\nnilb$.
	Then define
\emli $$
\trz A:=\bigvee\Trz A
 \quad \quad\text{and} \quad\quad
 \Tro A:=\{\alpha(B): B\in\Trz A\}
 \quad \quad\text{and} \quad\quad
\tro A:=\bigvee\Tro A.
$$
Note that by P3, one can effectively compute $\Trz A$, $\Tro A$, and $\tro A$.
\end{definition}
\begin{lemma}\label{lem-dnbv-glb}
	Given $A\in\lcalb$, we have:
		\begin{enumerate}
		\item $\Tro A\subseteq\dnb$, and hence $\tro A\in\dnbv$.
		\item $\ARNBIM\vdash A\rhd  \tro A$.
	\end{enumerate}
\end{lemma}
\begin{proof}
		\begin{enumerate}
		\item We show that for every $D\in \Tro A$, we have $D\in\dnb$. Let $D=\alpha(D_0)$ with $D_0\in\Trz A$ and $\alpha$ as in \Cref{def-glb-cdsnbv}. Hence $D_0\xrat{\theta_0}{\IPC}D_0^\dagger\in\nnilpq$. Then, if we let $\theta:=\alpha\circ\theta_0$, we have $D\xrat{\theta}{\IPC}D^\dagger\in\nnilb$, and thus $D\in\dnb$.
		\item P2 in \Cref{def-glb-cdsnbv} implies $\ARD\IPC\parr\vdash A_0\rhd \trz{A}$, and \Cref{ARN-ARNBT} implies $\ARNBIM\vdash A\rhd  \tro A$. \qedhere
	\end{enumerate}
\end{proof}
\begin{lemma}\label{lem-cdsnbv-glb-complexity}
	Given $A\in\lcalb$, for every $\Box E\in\subo{\tro A}$, we have $\Box E\in\subo{A}$.
\end{lemma}
\begin{proof}
Item P5 in \Cref{def-glb-cdsnbv}.
\end{proof}

\begin{lemma}\label{lem-cdlsnbv-glb-cto}
Given $A\in\lcalb$, for every $D\in \Tro A$, we have $\cto(D)\leq\cftob(A)$ (see \Cref{sec-cto}).
\end{lemma}
\begin{proof}
If $D=\alpha(D_0)$ with $D_0\in\Trz A$, then by P4 in \Cref{def-glb-cdsnbv}, we have $\cto(D_0)\leq\cftob(A)$. Also, we have $\cto(D_0)=\cto(\alpha(D_0))$, and thus $\cto(D)\leq\cftob(A)$.
\end{proof}

\begin{theorem}\label{Theorem-T-nnilb-Finitary0}
\uparan{\typea}
$(\dnbv,\sft)$ is recursively strong downward compact, and $\tro A=\ap\dnbv\sft A$.
\end{theorem}
\begin{proof}
Recursive strong downward compactness is derived from the definition of $\tro A$ once we show $\ap\dnbv\sft A=\tro A$. Hence, by \Cref{Gamma-approx-preserv}, it suffices to show the following items:
	\begin{itemize}
		\item $\tro A\in\dnbv$: First item of \Cref{lem-dnbv-glb}.
		\item $\tvdash \tro A\to A$: By P1 in \Cref{def-glb-cdsnbv}, for every $D_0\in\Trz A$, we have $\IPC\vdash D_0\to A_0$. Hence $\IPC\vdash D\to A$ for every $D\in\alpha(\Trz A)$.
		\item $A\prtdnbv \tro A$: By the second item of \Cref{lem-dnbv-glb,ASC}, we have $\ARNBTM\vdash A\rhd \tro A$. Then \Cref{Lem-con12-dlnb} implies $A\prtdnb \tro A$. Thus, by \Cref{vee-pres}, we are done. \qedhere
	\end{itemize}
\end{proof}

\begin{corollary}\label{Coro-Pres-HNF3}
\uparan{\typea}
	$\sft$ is closed under $\tro{(.)}$; i.e., if $\tvdash A$, then $\tvdash \tro A$.
\end{corollary}
\begin{proof}
	Let $\tvdash A$, and hence $\tvdash \top\to A$. \Cref{Theorem-T-nnilb-Finitary0} implies $A\prtdnb \tro A$, and since $\top\in\dnb$, we have $\tvdash \top\to\tro A$.
\end{proof}

\begin{theorem}\label{Theorem-NNILB-Pres3}
\uparan{\typea}
	For every $A,B\in\lcalb$:
\emli	$$ \ARNBTM\vdash A\rhd B\quad \text{ iff }\quad A\artnb B \quad \text{ iff }\quad A\prtdnb B.$$
\end{theorem}
\begin{proof}
	 $\ARNBTM\vdash A\rhd B$ implies $A\artnb B$: \Cref{Lem-con12-dlnb}.\\
	 $A\artnb B$ implies $A\prtdnb B$: Since $\sft$ is closed under outer substitutions, \Cref{adsm-pres} implies the desired result.
	 \\
	 $A\prtdnb B$ implies $\ARNBTM\vdash A\rhd B$: Let $A\prtdnb B$. \Cref{Cor-Gamma-approx-preserv,Theorem-T-nnilb-Finitary0} implies $\tvdash A_1^\sfp\to B$, and hence $\ARNBTM\vdash A_1^\sfp\rhd B$. On the other hand, \Cref{lem-dnbv-glb} implies $\ARNBTM\vdash A\rhd A_1^\sfp$. Thus $\ARNBTM\vdash A\rhd B$, as desired.
\end{proof}

\begin{theorem}[\textbf{Soundness}]\label{Lem-con12-dlnb}
\uparan{\typea}
$\ARNBTM\vdash A\rhd B$ implies $A\artnb B$.
\end{theorem}
\begin{proof}
Let $\ARNBTM\vdash A\rhd B$. \Cref{gen-admis-sound} implies $A\artpnb B$, and then \Cref{prspn=prsn} implies $A\artnb B$.
\end{proof}

\begin{lemma}\label{cdsnb=cdspnbv}
\uparan{\typez}
	Up to $\sft$-provable equivalence, we have $\dnbv=\dpnbv$.
\end{lemma}
\begin{proof}
	Since $\dpnb\subseteq\dnb$, we have $\dpnbv\subseteq\dnbv$. For the other direction, let $B\in\dnb$. Hence there exists some substitution $\theta$ such that $B\xrat{\theta}{\sft} B^\dagger\in \nnilb$, and by \Cref{prime-fact}, there exist propositions $B^\dagger_i\in\pnnilb$ such that $\ikfourcp\vdash B^\dagger\lr \bigvee_iB^\dagger_i$. Then let $B_i:=B\wedge B^\dagger_i$. Since $\tvdash B\to B^\dagger$, we have $\tvdash B\lr \bigvee_iB_i$. Then one may easily observe that $B_i\xrat{\theta}{\sft} B^\dagger_i\in\pnnilb$, and hence $B_i\in\dpnb$, as desired.
\end{proof}

In the following corollary, we consider $\ARNBTM$ as the binary relation that it axiomatizes.
\begin{corollary}\label{summery-ikfour-cdsnb3}
\uparan{\typea}
The following equalities hold. Moreover, if $\sft$ is decidable, then all the above relations are decidable.
\emli $${\ARNBTM}={\prtdnb}={\prtdnbv}={\prtdpnb}= {\prtdpnbv}= {\artnb}={\artnbv}={\artpnb}={\artpnbv}$$
\end{corollary}
\begin{proof}
\Cref{Theorem-NNILB-Pres3} implies ${\ARNBTM}={\prtdnb}$. On the other hand, \Cref{vee-pres} implies ${\prtdnb}={\prtdnbv}$, ${\prtdpnb}={\prtdpnbv}$, ${\artnb}={\artnbv}$, and ${\artpnb}={\artpnbv}$. \Cref{cdsnb=cdspnbv} implies ${\prtdnbv}={\prtdpnbv}$, and also \Cref{prime-fact} implies ${\artnb}={\artpnbv}$.

Next, we show the decidability of $A\prtdnbv B$. \Cref{Theorem-T-nnilb-Finitary0} implies that $\ap\dnbv\sft A$ exists and is computable. Then, by \Cref{Cor-Gamma-approx-preserv}, it suffices to decide $\tvdash \ap\dnbv\sft A\to B$, which is provided by the decidability of $\sft$.
\end{proof}

\subsection{$\dsnb$-Preservativity and $\snnilb$-Admissibility}
\label{Pres-logic-ARNBT3}

In this section, we show that $\AARNBT$ axiomatizes $\prtdsnb$ and $\artsnb$ whenever $\sft$ is \typea. Moreover, we show that $(\dsnbv,\sft)$ is recursively strong downward compact whenever $\sft$ is \typea.

\begin{lemma}\label{Lem-Box-translation-brtb}
	For every $A\in\NNILb$, we have $\AARNBK\vdash A\rhd \sggt A$.
\end{lemma}
\begin{proof}
	Use induction on the complexity of $A$. All cases are similar to the proof of \Cref{Lem-Box-translation-brts} and are left to the reader.
\end{proof}

\begin{definition}\label{def-glb-cdsnbv3}\em
	For every $A\in\lcalb$, define $\trth A\in\lcalb$ and $\Trth A \subseteq\lcalb$ as follows.
\emli
$$\Trth A:=\{D\wedge\sggt{(D^\dagger)}: D\in \Tro A \} \quad\quad\text{and} \quad\quad \trth A:=\bigvee\Trth A.
$$
Note that by \Cref{def-glb-cdsnbv}, one can effectively compute $\Trth A$ and $\trth A$.
\end{definition}

\begin{lemma}\label{cdlsnb-dlnb3}
If $\sft\supseteq\ikfourcp$ and $B\in\dnb$, then $B\wedge \sggt{(B^\dagger)}\in\dsnb$ and $\AARNBT\vdash B\rhd B\wedge \sggt{(B^\dagger)}$.
\end{lemma}
\begin{proof}
First, observe that $B\xrat{\theta}{\sft} B^\dagger$ implies $B\wedge\sggt{(B^\dagger)}\xrat\theta{\sft}\sggt{(B^\dagger)}$. Hence, if $B^\dagger\in\nnilb$, then $\sggt{(B^\dagger)} \in\snnilb$. Then \Cref{Remark-modal-proj} implies $\tvdash B\to B^\dagger$, and hence $\ARNBTM\vdash B\rhd B^\dagger$. Thus, by \Cref{Lem-Box-translation-brtb,ASC}, we have $\AARNBT\vdash B\rhd \sggt{(B^\dagger)}$.
\end{proof}
\begin{lemma}\label{lem-dsnbv-glb}
	Let $\sft\supseteq\ikfourcp$ and $A\in\lcalb$. Then
		\begin{enumerate}
		\item $\Trth A\subseteq\dsnb$, and hence $\trth A \in\dsnbv$.
		\item $\AARNBT\vdash A\rhd  \trth A$.
	\end{enumerate}
\end{lemma}
\begin{proof}
\Cref{lem-dnbv-glb,cdlsnb-dlnb3}.
\end{proof}
\begin{lemma}\label{glb-cdlsnb-dlnb3}
Let $\sft\supseteq\ikfourcp$, and let $B$ be a glb for $A$ with respect to $(\dnbv,\sft)$. Also assume that $B=\bigvee\Pi$ with $\Pi\subseteq\dnb$. Then $\bigvee\check{\Pi}$ is a glb for $A$ with respect to $(\dsnbv,\sft)$, where $\check{\Pi}:=\{D\wedge\sggt{(D^\dagger)}: D\in\Pi\}$.
\end{lemma}
\begin{proof}
	By \Cref{Gamma-approx-preserv}, it suffices to show the following items for arbitrary $D\in\Pi$:
	\begin{itemize}
		\item $D\wedge\sggt{(D^\dagger)}\in\dsnb$: \Cref{cdlsnb-dlnb3}.
		\item $\tvdash [D\wedge\sggt{(D^\dagger)}]\to A$: By assumption, we have $\tvdash D\to A$. Then, since $\vdash[D\wedge\sggt{(D^\dagger)}]\to D$, we get $\tvdash[D\wedge\sggt{(D^\dagger)}]\to A$.
		\item $A\prtdsnbv \bigvee\check\Pi$: By assumption, we have $A\prtdnbv \bigvee\Pi$. Also, \Cref{cdlsnb-dlnb3} implies $\AARNBT\vdash \bigvee\Pi\rhd \bigvee\check\Pi$. Thus, \Cref{Lem-con12-AR} implies $\bigvee\Pi\prtdsnb \bigvee\check\Pi$, and hence $A\prtdsnb \bigvee\check\Pi$. Then, by \Cref{vee-pres}, we have $A\prtdsnbv \bigvee\check\Pi$. \qedhere
	\end{itemize}
\end{proof}
\begin{theorem}\label{Theorem-T-nnilb-Finitary-AR}
\uparan{\typea}
$(\dsnbv,\sft)$ is recursively strong downward compact, and $\trth A=\ap\dsnbv\sft A$.
\end{theorem}
\begin{proof}
\Cref{Theorem-T-nnilb-Finitary0,glb-cdlsnb-dlnb3}.
\end{proof}
\begin{corollary}\label{Coro-Pres-HNF2-AR}
\uparan{\typea}
	$\sft$ is closed under $\trth{(.)}$; i.e., if $\tvdash A$, then $\tvdash \trth A$.
\end{corollary}
\begin{proof}
	Let $\tvdash A$, and hence $\tvdash \top\to A$. \Cref{Theorem-T-nnilb-Finitary-AR} implies $A\prtdsnb \trth A$, and since $\top\in\dsnb$, we have $\tvdash \top\to \trth A$.
\end{proof}

\begin{theorem}\label{Theorem-NNILB-Pres2-AR}
\uparan{\typea}
For every $A,B\in\lcalb$:
	\emli $$ \AARNBT\vdash A\rhd B\quad \text{ iff }\quad A\artsnb B \quad \text{ iff }\quad A\prtdsnb B.$$
\end{theorem}
\begin{proof}
	 $\AARNBT\vdash A\rhd B$ implies $A\artsnb B$: \Cref{Lem-con12-AR}.\\
	 $A\artsnb B$ implies $A\prtdsnb B$: Since $\sft$ is closed under outer substitutions, \Cref{adsm-pres} implies the desired result.
	 \\
	 $A\prtdsnb B$ implies $\AARNBT\vdash A\rhd B$: Let $A\prtdsnb B$. \Cref{Cor-Gamma-approx-preserv,Theorem-T-nnilb-Finitary-AR} implies $\tvdash \trth A\to B$, and hence $\AARNBT\vdash \trth A\rhd B$. On the other hand, \Cref{lem-dsnbv-glb} implies $\AARNBT\vdash A\rhd \trth A$. Thus $\AARNBT\vdash A\rhd B$.
\end{proof}

\begin{theorem}[\textbf{Soundness}]\label{Lem-con12-AR}
\uparan{\typea}
$\AARNBT\vdash A\rhd B$ implies $A\artsnb B$.
\end{theorem}
\begin{proof}
Let $\AARNBT\vdash A\rhd B$. \Cref{gen-admis-sound} implies $A\artspnb B$, and \Cref{prspn=prsn} implies $A\artsnb B$.
\end{proof}
\begin{lemma}\label{cdsnb=cdspnbv-2}
\uparan{\typea}
Up to $\sft$-provable equivalence, we have $\dsnbv=\dspnbv$ and $\cdsnbv=\cdspnbv$.
\end{lemma}
\begin{proof}
We only reason for the first equality and leave the similar argument for the second to the reader. Since $\dspnb\subseteq\dsnb$, we have $\dspnbv\subseteq\dsnbv$. For the other direction, let $B\in\dsnb$. Hence there exists some substitution $\theta$ such that $B\xrat{\theta}{\sft} B^\dagger\in \snnilb$, and by \Cref{prime-fact}, there exist propositions $B^\dagger_i\in\spnnilb$ such that $\ikfourcp\vdash B^\dagger\lr \bigvee_iB^\dagger_i$. Then let $B_i:=B\wedge B^\dagger_i$. Since $\tvdash B\to B^\dagger$, we have $\tvdash B\lr \bigvee_iB_i$. Then one may easily observe that $B_i\xrat{\theta}{\sft} B^\dagger_i\in\spnnilb$, and hence $B_i\in\dspnb$.
\end{proof}
\noindent
In the following corollary, we consider $\AARNBT$ as the binary relation that it axiomatizes.
\begin{corollary}\label{summery-ikfour-cdsnb3}
\uparan\typea
The following equalities hold. Moreover, if $\sft$ is decidable, then all the above relations are decidable.
\emli
$$
{\AARNBT}=
{\prtdsnb}=
{\prtdsnbv}=
{\prtdspnb}=
{\prtdspnbv}=
{\artsnb}={\artsnbv}={\artspnb}={\artspnbv}$$
\end{corollary}
\begin{proof}
\Cref{Theorem-NNILB-Pres2-AR} implies ${\AARNBT}={\prtdsnb}={\artsnb}$. On the other hand, \Cref{vee-pres} implies ${\prtdsnb}={\prtdsnbv}$, ${\prtdspnb}={\prtdspnbv}$, ${\artsnb}={\artsnbv}$, and ${\artspnb}={\artspnbv}$. Moreover, \Cref{prime-fact} implies ${\artsnb}={\artspnbv}$. Finally, \Cref{cdsnb=cdspnbv-2} implies ${\prtdsnbv}= {\prtdspnbv}$.

Next, we show the decidability of $A\prtdsnbv B$. \Cref{Theorem-T-nnilb-Finitary-AR} implies that $\ap\dsnbv\sft A$ exists and is computable. Then, by \Cref{Cor-Gamma-approx-preserv}, it suffices to decide $\tvdash \ap\dsnbv\sft A\to B$, which is provided by the decidability of $\sft$.
\end{proof}

\subsection{$\cdsnb$-Preservativity}
\label{Pres-logic-ARNBT2}
In this section, we show that $\ARNBT$ axiomatizes $\prtcdsnb$ whenever $\sft$ is \typea. Moreover, we show that $(\cdsnbv,\sft)$ is recursively strong downward compact whenever $\sft$ is \typea.
\begin{definition}\label{def-glb-cdsnbv2}\em
	For every $A\in\lcalb$, we define $\trt A\in\lcalb$ and $\Trt A \subseteq\lcalb$ as follows.
\emli $$
\Trt A:=\{\Boxdot D\wedge \sggt{(D^\dagger)}: D\in \Tro A \}
\quad\quad\text{and} \quad\quad
\trt A:=\bigvee\Trt A.
$$
Note that by \Cref{def-glb-cdsnbv}, one can effectively compute $\Trt A$ and $\trt A$.
\end{definition}

\begin{lemma}\label{cdlsnb-dlnb2}
If $\sft\supseteq\ikfourcp$ and $B\in\dnb$, then $\Boxdot B\wedge \sggt{(B^\dagger)}\in\cdsnb$ and $\ARNBT\vdash B\rhd \Boxdot B\wedge \sggt{(B^\dagger)}$.
\end{lemma}
\begin{proof}
First, observe that $B\xrat{\theta}{\sft} B^\dagger\in\nnilb$ implies $\Boxdot B\wedge \sggt{(B^\dagger)} \xrat\theta{\sft}\Box B\wedge \sggt{(B^\dagger)}$. Hence, if $B^\dagger\in\nnilb$, then $\Box B\wedge \sggt{(B^\dagger)}\in\snnilb$, and thus $\Boxdot B\wedge \sggt{(B^\dagger)}\in\cdsnb$. \Cref{cdlsnb-dlnb3} implies $\ARNBT\vdash B\rhd   \sggt{(B^\dagger)}$. Also, by Leivant's principle, we have $\ARNBT\vdash B\rhd \Boxdot B$, and thus $\ARNBT\vdash B\rhd \Boxdot B\wedge \sggt{(B^\dagger)}$.
\end{proof}
\begin{lemma}\label{lem-cdsnbv-glb}
	Let $\sft\supseteq\ikfourcp$ and $A\in\lcalb$. Then we have:
		\begin{enumerate}
		\item $\Trt A\subseteq\cdsnb$, and hence $\trt A\in\cdsnbv$.
		\item $\ARNBT\vdash A\rhd  \trt A$.
	\end{enumerate}
\end{lemma}
\begin{proof}
\Cref{lem-dnbv-glb,cdlsnb-dlnb2}.
\end{proof}

\begin{lemma}\label{glb-cdlsnb-dlnb2}
Let $\sft\supseteq\ikfourcp$, and let $B$ be a glb for $A$ with respect to $(\dnbv,\sft)$. Also assume that $B=\bigvee\Pi$ with $\Pi\subseteq\dnb$. Then $\bigvee\hat\Pi$ is a glb for $A$ with respect to $(\cdsnbv,\sft)$, where $\hat\Pi:=\{\Boxdot D\wedge \sggt{(D^\dagger)}: D\in\Pi\}$.
\end{lemma}
\begin{proof}
By \Cref{Gamma-approx-preserv}, it suffices to show the following items for arbitrary $D\in\Pi$:
	\begin{itemize}
		\item $\Boxdot D\wedge \sggt{(D^\dagger)}\in\cdsnb$: \Cref{cdlsnb-dlnb2}.
		\item $\tvdash [\Boxdot D\wedge \sggt{(D^\dagger)}]\to A$: By assumption, we have $\tvdash D\to A$. Then, since $ \vdash[\Boxdot D\wedge \sggt{(D^\dagger)}]\to D$, we get $\tvdash[\Boxdot D\wedge \sggt{(D^\dagger)}]\to A$.
		\item $A\prtcdsnbv \bigvee\hat\Pi$: By assumption, we have $A\prtdnbv \bigvee\Pi$. Also, \Cref{cdlsnb-dlnb2} implies $\ARNBT\vdash \bigvee\Pi\rhd \bigvee\hat\Pi$. Thus, \Cref{Lem-con12} implies $\bigvee\Pi\prtcdsnb  \bigvee\hat\Pi$, and hence $A\prtcdsnb \bigvee\hat\Pi$. Then, by \Cref{vee-pres}, we have $A\prtcdsnbv \bigvee\hat\Pi$. \qedhere
	\end{itemize}
\end{proof}
\begin{theorem}\label{Theorem-T-nnilb-Finitary}
\uparan\typea
$(\cdsnbv,\sft)$ is recursively strong downward compact, and $ \trt A=\ap\cdsnbv\sft A$.
\end{theorem}
\begin{proof}
\Cref{Theorem-T-nnilb-Finitary0,glb-cdlsnb-dlnb2}.
\end{proof}

\begin{corollary}\label{Coro-Pres-HNF2}
\uparan\typea
	$\sft$ is closed under $\trt{(.)}$; i.e., if $\tvdash A$, then $\tvdash \trt A$.
\end{corollary}
\begin{proof}
	Let $\tvdash A$, and hence $\tvdash \top\to A$. \Cref{Theorem-T-nnilb-Finitary} implies $A\prtcdsnb \trt A$, and since $\top\in\cdsnb$, we have $\tvdash \top\to \trt A$.
\end{proof}

\begin{theorem}\label{Theorem-NNILB-Pres2}
\uparan\typea
	For every $A,B\in\lcalb$:
\emli	$$ \ARNBT\vdash A\rhd B\quad \text{ iff }\quad A\prtcdsnb B.$$
\end{theorem}
\begin{proof}
	 $\ARNBT\vdash A\rhd B$ implies $A\prtcdsnb B$: \Cref{Lem-con12}.\\
	 $A\prtcdsnb B$ implies $\ARNBT\vdash A\rhd B$: Let $A\prtcdsnb B$. \Cref{Cor-Gamma-approx-preserv,Theorem-T-nnilb-Finitary} imply $\tvdash \trt A\to B$, and hence $\ARNBT\vdash \trt A\rhd B$. On the other hand, \Cref{lem-cdsnbv-glb} implies $\ARNBT\vdash A\rhd \trt A$. Thus $\ARNBT\vdash A\rhd B$.
\end{proof}

\begin{theorem}[\textbf{Soundness}]\label{Lem-con12}
\uparan\typea
	$\ARNBT\vdash A\rhd B$ implies $A\prtcdsnb B$.
\end{theorem}
\begin{proof}
Let $\ARNBT\vdash A\rhd B$. \Cref{gen-pres-sound} implies $A\prtcdspnb B$, and thus \Cref{cdsnb=cdspnbv-2} implies $A\prtcdsnb B$.
\end{proof}
In the following corollary, we consider $\ARNBT$ as the binary relation that it axiomatizes.
\begin{corollary}\label{summery-ikfour-cdsnb2}
\uparan\typea
${\ARNBT}={\prtcdsnb}={\prtcdsnbv}={\prtcdspnb}= {\prtcdspnbv}$. Moreover, if $\sft$ is decidable, then all mentioned relations are decidable.
\end{corollary}
\begin{proof}
\Cref{Theorem-NNILB-Pres2} implies ${\ARNBT}={\prtcdsnb}$. On the other hand, \Cref{vee-pres} implies ${\prtcdsnb}={\prtcdsnbv}$ and ${\prtcdspnb}= {\prtcdspnbv}$. Finally, \Cref{cdsnb=cdspnbv-2} implies ${\prtcdsnbv}={\prtcdspnbv}$.

Next, we show the decidability of $A\prtcdsnbv B$. \Cref{Theorem-T-nnilb-Finitary} implies that $\ap\cdsnbv\sft A$ exists and is computable. Then, by \Cref{Cor-Gamma-approx-preserv}, it suffices to decide $\tvdash \ap\cdsnbv\sft A\to B$, which is provided by the decidability of $\sft$.
\end{proof}

\begin{lemma}\label{iglh-in-iph}
$\iglcph\vdash A$ implies $\iph\vdash A$.
\end{lemma}
\begin{proof}
We prove by induction on the proof complexity of $\iglcph\vdash A$ that $\iph\vdash A$. All cases are trivial except for when $A$ is an axiom instance of $\vis\cdsnb\iglcp$; i.e., $A=\Box B\to \Box C$ with $B\priglcdsnb C$. \Cref{summery-ikfour-cdsnb2} implies that $\ARDP\iglcp\parb\vdash B\rhd C$. Also, from the definition of $\ARDP\iglcp\parb$ (see \Cref{pres-admis}), it is clear that $\iph\vdash  B\rhd C$. Then Cut implies $\iph\vdash \top\rhd B\to\top\rhd C$. Thus $\iph\vdash \Box B\to \Box C$.
\end{proof}

\subsection{$\snnilb$-Preservativity}
\label{Pres-logic-ARNBTP}
In this section, we show that $\ARNBTP$ axiomatizes $\prtsnb$ whenever $\sft$ is \typea. Moreover, we show that $(\snnilbv,\sft)$ is recursively strong downward compact whenever $\sft$ is \typea.
\begin{theorem}\label{PLHA0-2}
	For every $A\in\lcalz$, one can effectively compute $A^\star\in\NNILpar$ such that:
	\begin{enumerate}
		\item $\IPC\vdash A^\star\to A$,
		\item $\ARNN\vdash A\rhd A^\star$,
		\item $\suba{A^\star}\subseteq\suba{A}$.
	\end{enumerate}
\end{theorem}
\begin{proof}
See Lemma 4.23 in \citep{PLHA0}.
\end{proof}
Recall that $\parb$ is the set of parameters or boxed propositions.
\begin{lemma}\label{ARNN-ARNBIP}
Let $\alpha$ be a substitution such that for every $a\in\atom$, if $\alpha(a)\neq a$, then $a\in\parr$ and $\alpha(a)\in \parb$. Then $\ARNN\vdash A\rhd B$ implies $\ARD\IPC\parb\VAB\vdash \alpha(A)\rhd\alpha(B)$, for every $A,B\in\lcalz$.
\end{lemma}
\begin{proof}
Straightforward induction on the proof $\ARNN\vdash A\rhd B$, left to the reader.
\end{proof}
\begin{lemma}\label{APlus1}
	For every $A\in\lcalb$, one can effectively compute some $A^\sstar\in\NNILb$ such that:
	\begin{enumerate}
		\item $\IPC\vdash A^\sstar\to A$,
		\item $\ARNBIP\vdash A\rhd A^\sstar$,
		\item $\suba{A^\sstar}\subseteq\sub A$.
	\end{enumerate}
\end{lemma}
\begin{proof}
Let $\Box B_1,\ldots,\Box B_n$ be the list of all outer occurrences of boxed formulas in $A$. Moreover, let $\pvec:=p_1,\ldots,p_n$ be a list of fresh atomic parameters; i.e., $p_i\nin\sub A$ and they are pairwise distinct. Also, assume that $\alpha$ is the substitution such that
\emli $$\alpha(a):=\begin{cases}
\Box B_i\quad &:\text{ if } a=p_i,\\
a &:\text{ otherwise.}
\end{cases}$$
Then there exists a unique $A_0\in\lcalz$ such that $\alpha(A_0)=A$. \Cref{PLHA0-2} gives us some $A_0^\star\in\NNILpar$ such that
	\begin{enumerate}
		\item $\IPC\vdash A_0^\star\to A_0$,
		\item $\ARNN\vdash A_0\rhd A_0^\star$,
		\item $\suba{A_0^\star}\subseteq\suba{A_0}$.
	\end{enumerate}
Define $A^\sstar:=\alpha(A_0^\star)$. Then we have:
\begin{enumerate}
	\item $\IPC\vdash A^\sstar\to A$,
	\item By \Cref{ARNN-ARNBIP}, we have $\ARNBIP\vdash A\rhd A^\sstar$,
	\item $\suba{A^\sstar} \subseteq\sub A$. \qedhere
\end{enumerate}
\end{proof}
\begin{theorem}\label{APlus}
	For every $A$, there exists some $A^\sfss\in\snnilb$ such that:
	\begin{enumerate}
		\item $\ikfourca\vdash A^\sfss\to A$,
		\item $\ARNBKP\vdash A\rhd A^\sfss$,
		\item $\Box B\in \sub{A^\sfss}$ implies either $\Box B\in  \sub{A}$ or $\Boxdot B\in \snnilb$.
		\item $\suba{A^\sfss}\subseteq\suba{A}$.
	\end{enumerate}
\end{theorem}
\begin{proof}
	By \Cref{APlus1}, there exists some $A^\sstar$ with the mentioned properties. Since $A^\sstar\in\nnilb$, we have $\ikfourcp\vdash  \sggt{(A^\sstar)}\to A^\sstar$. Also, \Cref{Lem-Box-translation-brtb} implies $\ARNBK\vdash A^\sstar\rhd \sggt{(A^\sstar)}$. Hence, if we let $A^\sfss:= \sggt{(A^\sstar)}$, by \Cref{APlus1} we have all required properties.
\end{proof}

\begin{theorem}\label{SNB-approx-prop}
\uparan\typea
$(\snnilb,\sft)$ is recursively strong downward compact, and $A^\sfss:=\sggt{(A^\sstar)}=\ap\snnil\sft A$.
\end{theorem}
\begin{proof}
	We show that $A^\sfss$ is the $(\snnilb,\sft)$-glb for $A$, where $A^\sfss$ is as provided by \Cref{APlus}. By \Cref{APlus}, $A^\sfss\in\snnilb$, $\ikfourcp\vdash A^\sfss\to A$, and $\ARNBTP\vdash A\rhd A^\sfss$. Then \Cref{Lem-con1-Box} implies $A\prtsnb A^\sfss$. Hence, \Cref{Gamma-approx-preserv} implies the desired result.
\end{proof}

\begin{theorem}\label{Theorem-NNIL-Pres2}
\uparan\typea
	$\ARNBTP$ is sound and complete for $\snnilb$-preservativity in $\sft$; i.e., for every $A,B\in\lcalb$
\emli	$$ \ARNBTP\vdash A\rhd B \quad \text{ iff }\quad  A\prtsnb B.$$
\end{theorem}
\begin{proof}
	The left-to-right direction (soundness) holds by \Cref{Lem-con1-Box}.
	\\
	For the other direction (completeness), let $A\prtnb B$. Then \Cref{SNB-approx-prop,Cor-Gamma-approx-preserv} implies $\tvdash A^\sfss\to B$, and hence $\ARNBT\vdash A^\sfss\rhd B$. Also, by \Cref{APlus}, we have $\ARNBT\vdash A\rhd A^\sfss$, and Cut implies the desired result.
\end{proof}

\begin{theorem}[\textbf{Soundness}]\label{Lem-con1-Box}
\uparan\typea
$\ARNBTP\vdash A\rhd B$ implies $A\prtsnb B$.
\end{theorem}
\begin{proof}
Let $\ARNBTP\vdash A\rhd B$. \Cref{gen-pres-sound} implies $A\prtspnb B$, and thus \Cref{prspn=prsn} implies $A\prtsnb B$.
\end{proof}

For uniformity of notation, in the following corollary we consider $\ARNBTP$ as the binary relation that it axiomatizes.
\begin{corollary}\label{summery-ikfour-snnilb}
\uparan\typea
${\ARNBTP}={\prtsnb}= {\prtspnb}={\prtspnbv}$. Moreover, if $\sft$ is decidable, then all mentioned relations are decidable.
\end{corollary}
\begin{proof}
All equalities are derived from \Cref{Theorem-NNIL-Pres2,prspn=prsn,vee-pres}. For the decidability of $\prtsnb$, we have the following argument. \Cref{SNB-approx-prop} implies that $\ap\snnilb\sft A$ exists and is computable. Then, by \Cref{Cor-Gamma-approx-preserv}, it suffices to decide $\tvdash \ap\snnilb\sft A\to B$, which is provided by the decidability of $\sft$.
\end{proof}

\section{Provability Models}\label{Prov-semant}
\citep{IemhoffT,Iemhoff,Iemhoff.Preservativity} consider Kripke semantics for several intuitionistic modal logics and prove soundness/completeness theorems for them. Although such Kripke semantics are useful tools, a major obstacle to their application is that they are typically infinite. Here we introduce an alternative semantics for provability logics. This variant, as we will see, enjoys the finite model property, which is a key point for proving the arithmetical completeness and decidability of the provability logic of $\HA$.

The idea behind provability models is that each possible world is assigned a theory, and the validity of $\Box A$ is defined as follows: for every accessible node, $A$ must be \textit{provable} in the assigned theory. 

In \citep{PModels}, we considered the provability models for classical modal logics, like 
${\sf GL}$, the interpretability logic ${\sf ILM}$ and the poly-modal provability logic ${\sf GLP}$.

\begin{definition}\label{Mixed-sem}
A \textit{provability model} is a tuple $\pcal=(W,\pce,\R,\{\lgc w\}_{w\in W},\V)$ with the following properties:
\begin{itemize}
\item $\kcal_\pcal:=(W,\pce,\R,\V)$ is a Kripke model for intuitionistic modal logic (as defined in \cref{sec-Kripke}).
\item For any $\R$-accessible node $w\in W$,
$\lgc w$ is a theory\footnote{Recall that a theory is a set of formulas that is closed under modus ponens and includes all axioms of $\IPC$.}.
\end{itemize}
\textit{Notation:} We use $W^\R$ to denote the set of all $\R$-accessible nodes in $W$.
We define $\pcal,w\pmodels A$ by induction on the complexity of $A$, with $\pcal,w\npmodels \bot$, and commuting with $\vee$ and $\wedge$, and
\emli
$$\pcal,w\pmodels A\to B  \quad\text{iff}\quad \forall u\sce w \ (\pcal,u\pmodels A\Rightarrow \pcal,u\pmodels B).$$
\emli
$$\pcal,w\pmodels \Box A \quad\text{iff}\quad \forall u\sqsupset w \ \lgc u\vdash A.$$
We define $\pcal\pmodels A$ if for every $w\in W$, we have $\pcal,w\pmodels A$. Also, for a class $\pfrak$ of provability models, we define $\pfrak \pmodels A$ if for every $\pcal\in\pfrak$, we have $\pcal\pmodels A$.
We define $\pcal,w\pmodelsp A$ if there exists some $w_0\sqsubset w$ such that for every $u$ accessible from $w_0$ via the transitive closure of $\R$, we have $\pcal,u\pmodels A$. Also, define $\pcal,w\pmodelss A$ if for every $u$ accessible from $w$ via the transitive reflexive closure of $\R$, we have $\pcal,u\pmodels A$.
Given a set $\Gamma$ of formulas, $\pcal$ is called \textit{$\Gamma$-full} if
\begin{itemize}
\item For every $A\in\Gamma$ and every $\R$-accessible $w\in W$, if $\pcal,w\pmodelsp A$, then $\lgc w\vdash A$.
\end{itemize}
Also, $\pcal$ is called \textit{strongly $\Gamma$-full} if
\begin{itemize}
\item For every $A\in\Gamma$ and every $\R$-accessible $w\in W$, if $\pcal,w\pmodelss A$, then $\lgc w\vdash A$.
\end{itemize}
Moreover, $\pcal$ is called (strongly) \textit{$A$-full} if it is (strongly) $\Gamma$-full for $\Gamma:=\{\Box B: \Box B\in \sub A\}$.
Note that strong $\Gamma$-fullness implies $\Gamma$-fullness.
For later reference, define
\emli
	$$\Gamma^\pcal_w:=\{A\in\Gamma: \pcal,w\pmodelsp A\}.$$
	If no confusion is likely, we may omit $\pcal$ and simply write $\Gamma_w$.
	We also say that $\lgc w$ is \textit{locally sound} if $\lgc w\vdash A$ implies $\pcal,w\pmodels A$. Furthermore, we say $\pcal$ is locally sound if for every $\R$-accessible $w\in W$, $\lgc w$ is locally sound.

A \textit{frame property} for a provability model is a property that applies to $(W,\R)$. For example, frame properties such as reflexivity, transitivity, or converse well-foundedness mean that $(W,\R)$ has that property.

On the other hand, we also have \textit{logical} properties of the provability model, which are properties that correspond to all $\lgc w$ for $w\in W^\R$. For example, a provability model with necessitation means that all theories $\lgc w$ are closed under necessitation.
Also, we say that a provability model is \textit{intuitionistic} if all $\lgc w$ are intuitionistic\footnote{Recall that this means $\lgc w$ includes modus ponens as its rule of inference and includes all Hilbert axioms of intuitionistic logic among its theorems.}.

We say that a theory $\sft$ is \textit{sound} for $\pcal$ if $\pcal,w\pmodels A$ for every $w\in W$ and every $A$ with $\sft\vdash A$. We say that $\sft$ is \textit{strongly sound} for $\pcal$ if it is sound for $\sft$ and also $\lgc w$ includes $\sft$ for every $w\in W$. Then $\sft$ is called (strongly) sound for a class $\pfrak$ of provability models if it is (strongly) sound for every $\pcal\in\pfrak$.

We say that $\sft$ is \textit{complete} for a class $\pfrak$ of provability models if $\pfrak\pmodels A$ implies $\sft\vdash A$.
\\[3mm]
\textbf{Convention:} \textit{Throughout this paper, we assume that provability models are intuitionistic; i.e., the assigned theories include all axioms of intuitionistic logic and the modus ponens rule of inference. Moreover, all provability models are parameter-persistent; i.e., $\pcal,w\pmodels p$ and $u\sqsupset w$ implies $\pcal,u\pmodels p$, for any parameter $p$.}
\end{definition}

\begin{lemma}\label{prov-kripke-normal-forms}
Let $\pcal$ be an $A$-full, transitive, conversely well-founded, and locally sound provability model. Then for any world $w$ in $\pcal$, we have $\kcal_\pcal,w\Vdash A$ iff $\pcal,w\pmodels A$.
\end{lemma}
\begin{proof}
We prove this by induction on $A$. Observe that if $\pcal$ is $A$-full, then it is also $B$-full for every $B\in \sub A$.
All cases are obvious except for $A=\Box B$.
If $\kcal_\pcal,w\Vdash \Box B$, then for any $u\sqsupset w$, we have $\kcal_\pcal,u\Vdash B$. Thus, by the induction hypothesis, for every $u\sqsupset w$, we have $\pcal,u\pmodelsp B$. Then, by $A$-fullness, we get $\lgc u\vdash B$, and hence $\pcal,w\pmodels \Box B$, as desired.

For the other direction, let $\pcal,w\pmodels \Box B$. Then for every $u\sqsupset w$, we have $\lgc u\vdash B$. By local soundness, we get $\pcal,u\pmodels B$ for every $u\sqsupset w$. Thus, the induction hypothesis implies $\kcal_\pcal,w\Vdash\Box B$, as desired.
\end{proof}

\begin{lemma}\label{prov-kripke-normal-forms2}
Let $\pcal$ be a strongly $\Gamma$-full, transitive, conversely well-founded, and locally sound provability model. Then for any world $w$ in $\pcal$ and any $A\in\nf\Gamma$, we have $\kcal_\pcal,w\Vdash A$ iff $\pcal,w\pmodels A$.
\end{lemma}
\begin{proof}
We prove this by induction on $A$.
All cases are obvious except for $A=\Box B$.
Given that $\Box B\in\nf\Gamma$, either we have $B\in\Gamma$ or $\Boxdot B\in\Gamma$.
Let us assume that $\Boxdot B\in\Gamma$. The other case is even easier and left to the reader.
If $\kcal_\pcal,w\Vdash \Box B$, then for any $u\sqsupset w$, we have $\kcal_\pcal,u\Vdash B$. Thus, by the induction hypothesis, for every $u\sqsupset w$, we have $\pcal,u\pmodelss B$. Now we use a second induction on $u$, ordered by $\sqsupset$, and show that $\pcal,u\pmodelss \Boxdot B$. This means that, as the second induction hypothesis, we may assume that for every $v\sqsupset u$, we have $\pcal,v\pmodelss \Boxdot B$.
To show that $\pcal,u\pmodelss \Boxdot B$, first note that we already have $\pcal,u\pmodelss B$. Also, by the second induction hypothesis and strong $\Gamma$-fullness, we get $\lgc v\vdash \Boxdot B$ for every $v\sqsupset u$. Thus, by transitivity, $\pcal,u\pmodelss \Box B$, which implies $\pcal,u\pmodelss \Boxdot B$, as desired.

For the other direction, let $\pcal,w\pmodels \Box B$. Then for every $u\sqsupset w$, we have $\lgc u\vdash B$. By local soundness, we get $\pcal,u\pmodels B$ for every $u\sqsupset w$. Thus, the induction hypothesis implies $\kcal_\pcal,w\Vdash\Box B$, as desired.
\end{proof}

\subsection{Soundness theorems}
In this subsection, we show that provability models with specific closure conditions are sound for the intuitionistic modal logics of interest (see \Cref{provability-models-gen-soundness-3}).

A set of formulas $\Delta$ is called \textit{$\iglcp$-adequate} if it includes all instances of the following schemes:
\begin{itemize}[leftmargin=1.5cm]
\item[\ul{\sf K}:] $\Box(A\to B)\to (\Box A \to\Box B)$.
\item[\ul{\sf 4}:] $\Box A\to\Box\Box A$.
\item[\ul{\sf L}:] $\Box(\Box A\to A)\to\Box A$ (the L\"ob axiom).
\item[$\underline{\cpp}$:] $p\to\Box p$ for every $p\in\parr$.
\end{itemize}
Also, $\Delta$ is called \textit{$\iglca$-adequate} if it is $\iglcp$-adequate and also includes all instances of the axiom scheme
\begin{itemize}
\item[$\underline{\cpa}$:] $a\to\Box a$ for every $a\in\atom$.
\end{itemize}
Moreover, $\Delta$ is called \textit{$\Box$-adequate} if $\Box E\in\Delta$ for every $E\in\lcalb$. Also, $\Delta$ is called \textit{rule-adequate} if it is $\Box$-adequate and $\Box E\to\Box F\in \Delta$ for every $E,F\in\lcalb$.

\begin{remark}
$\snnilb$ and $\snnil$ are both rule-adequate. Also, $\snnilb$ is $\iglcp$-adequate and $\snnil$ is $\iglca$-adequate. We shall use this fact without further explanation.
\end{remark}

\begin{theorem}\label{provability-models-gen-soundness}
We have the following soundness results:
\begin{itemize}
\item If $\Delta$ is $\iglcp$-adequate, then it is strongly sound for strong $\Delta$-full provability models with necessitation and converse well-foundedness.
\item If $\Delta$ is $\iglca$-adequate, then it is strongly sound for atomic ascending\footnote{Recall that a provability model is atomic ascending if it satisfies: $w\V a$ and $w\R u$ imply $u\V a$ for every atomic $a$.} strong $\Delta$-full provability models with necessitation and converse well-foundedness.
\end{itemize}
\end{theorem}
\begin{proof}
Easy induction on the complexity of a proof.
\end{proof}

\begin{lemma}\label{provability-models-gen-soundness-1}
$\underline{\visgt}$ is sound for provability models that are closed under $\prtg$.
\end{lemma}
\begin{proof}
Obvious.
\end{proof}

\begin{corollary}\label{provability-models-gen-soundness-3}
Let $\Delta$ be $\iglcp$-adequate and rule-adequate. Then we have:
\begin{itemize}
\item $\iglcp\Ha$ is strongly sound for strong $\Delta$-full provability models with necessitation, converse well-foundedness, and closure under $\priglcdsnb$.
\item $\iglcp\Hb$ is strongly sound for strong $\Delta$-full provability models with necessitation, converse well-foundedness, and closure under $\priglsnb$.
\item If $\Delta$ is also $\iglca$-adequate, then $\lles$ is strongly sound for strong $\Delta$-full provability models with necessitation, converse well-foundedness, and closure under $\priglcasn$.
\end{itemize}
\end{corollary}
\begin{proof}
Immediate consequence of \Cref{provability-models-gen-soundness,provability-models-gen-soundness-1}.
\end{proof}

\subsection{Construction of provability models}\label{constrction-of-pro-models}
In this subsection, we address how to construct a provability model from given basic information.

Let $\kcal=(W,\pce,\R,\V)$ be a conversely well-founded Kripke model for intuitionistic modal logic, $\Delta$ a set of formulas, and $\Phi=\{\varphi_w\}_{w\in W}$ an indexed family of formulas. Then define $\pcal:=\cpmp\kcal\Delta\Phi:=(W,\pce,\R,\{\lgc w\}_{w\in W},\V)$ with
\emli
$$
\Delta^\pcal_w:=\{A\in \Delta: \cpmp\kcal\Delta\Phi ,w\pmodelss A \}
\quad \text{ and } \quad
\lgc w:=\IPC+\Delta^\pcal_w+\varphi_w.
$$
More precisely, the set of axioms of $\lgc w$ includes all axioms of $\IPC$, $\varphi_w$, and $\Delta^\pcal_w$. Also, $\lgc w$ includes modus ponens as its only inference rule.

We say that $\Phi$ is \textit{$\kcal$-persistent} if for every $w\R u$ in $\kcal$, we have $\vdash \varphi_u\to\varphi_w$.

A provability model $\pcal$ is called \textit{$(\Delta,\Gamma)$-based} if there exists a conversely well-founded Kripke model $\kcal$ and a $\kcal$-persistent $\Phi:=\{\varphi_w\}_{w\in W}$ such that $\pcal=\cpmp\kcal\Delta\Phi$ and, modulo $\IPC+\Delta^\pcal_w$-provable equivalence, $\varphi_w\in \Gamma$\footnote{This means that there exists some $E\in \Gamma$ such that $\Delta^\pcal_w\vdash \varphi_w\lr E$.}. It is called \textit{$\Delta$-based} if $\pcal=\cpm\kcal\Delta$.
Observe that $\cpmp\kcal\Delta\Phi$ is locally sound iff for every $w$, we have $\cpmp\kcal\Delta\Phi,w\pmodels \varphi_w$. In particular, any $\Delta$-based provability model is locally sound.

Note that in the above definition of $\lgc w$, we are using validity in the same provability model, leading to a circular definition. However, by converse well-foundedness, this definition can be presented via a legitimate recursion and is thus well-defined.
From the definition, it is obvious that $\cpmp\kcal\Delta\Phi$ is a $\Delta$-full provability model.

Whenever all formulas $\varphi_w$ in $\cpmp\kcal\Delta\Phi$ are equal to $\top$, we omit them in the notation and simply write $\cpm{\kcal}{\Delta}$ instead.

\begin{lemma}\label{sound-const}
Let $\Delta$ be $\Box$-adequate. Then any $(\Delta,\Gamma)$-based provability model is strongly $\Delta$-full and closed under necessitation.
\end{lemma}
\begin{proof}
Strong $\Delta$-fullness is obvious by definition. So it only remains to show closure under necessitation.
Let $\kcal=(W,\pce,\R,\V)$ and $\pcal=\cpmp\kcal\Delta\Phi$.
Assume that $\lgc w\vdash A$. Then there exists a finite set $X\subseteq \Delta$ such that $\pcal,w\pmodelss \bigwedge X$ and $\vdash (\bigwedge X\wedge\varphi_w)\to A$. Therefore, $\pcal,u\pmodelss \bigwedge X$ for every $u\sqsupset w$. By the definition of $\pcal$, this implies that $\lgc u\vdash \bigwedge X$ for every $u\sqsupset w$. Hence $\lgc u \vdash \varphi_w\to A$ for every $u\sqsupset w$. Thus, by $\kcal$-persistence of $\Phi$, we get $\lgc u\vdash A$. Therefore, $\pcal,w\pmodels \Box A$. Then, by strong $\Delta$-fullness and $\boxed\subseteq\Delta$, we have $\lgc w\vdash \Box A$, as desired.
\end{proof}


\begin{corollary}\label{sound-const1}
Let $\Delta$ be $\Box$-adequate. Then we have the following soundness results:
\begin{itemize}
\item If $\Delta$ is $\iglcp$-adequate, then $\iglcp$ is strongly sound for $(\Delta,\Gamma)$-based provability models.
\item If $\Delta$ is $\iglca$-adequate, then $\iglca$ is strongly sound for $(\Delta,\Gamma)$-based and atomic ascending provability models.
\end{itemize}
\end{corollary}
\begin{proof}
Direct consequence of \Cref{sound-const,provability-models-gen-soundness}.
\end{proof}

\begin{corollary}\label{cor-pro-models}
For every $(\snnilb,\cdsnb)$-based provability model, there exists some $(\snnilb,\sfc\dar\IPC\snnil)$-based provability model that is equal to it.
\end{corollary}
\begin{proof}
Use \Cref{Lem-red-proj-rela,sound-const1}.
\end{proof}

\begin{theorem}\label{sound-const2}
Let $\Delta$ be rule-adequate. Then we have the following soundness results:
\begin{itemize}
\item If $\Delta\subseteq \snnilb$ is $\iglcp$-adequate, then $\iglcp\Hb$ is strongly sound for ${\Delta}$-based provability models.
\item If $\Delta\subseteq \snnil$ is $\iglca$-adequate, then $\iglca\Hs$ is strongly sound for ${\Delta}$-based and atomic ascending provability models.
\end{itemize}
\end{theorem}
\begin{proof}
Using \Cref{sound-const1}, observe that $\cpm{\kcal}{\Delta}$ is closed under $\priglsnb$ and $\priglcasn$, respectively, in the first and second cases. Then, by \Cref{sound-const,provability-models-gen-soundness-3}, we obtain the desired result.
\end{proof}

\begin{theorem}\label{sound-const3}
Let $\Delta\subseteq\snnilb$ be rule-adequate and $\iglcp$-adequate. Then $\iglcph$ is strongly sound for $(\Delta,\cdsnb)$-based provability models.
\end{theorem}
\begin{proof}
First, note that by \Cref{sound-const1}, $\iglcp$ is strongly sound for $\cpmp\kcal\Delta\Phi$. This means that for every $w\in W$, the theory $\lgc w$ includes $\iglcp$. Then, by \Cref{provability-models-gen-soundness-3,sound-const}, it suffices to show that $\cpmp\kcal\Delta\Phi$ is closed under $\priglcdsnb$. So assume that $\lgc w\vdash A$ and $A\priglcdsnb B$. Hence there exists a formula $E\in\snnilb$ such that $\vdash (E\wedge\varphi_w)\to A$ and $\cpmp\kcal\Delta\Phi,w\pmodelss E$. Given that $\varphi_w\in\cdsnb$ and $\cdsnb$ is closed under $\snnilb$-conjunctions, we still have $(E\wedge\varphi_w)\in\cdsnb$. Thus, by $A\priglcdsnb B$, we get $\iglcp\vdash (E\wedge\varphi_w)\to  B$. Since $\lgc w$ includes $\iglcp$, this implies $\lgc w\vdash B$, as desired.
\end{proof}

\begin{theorem}\label{Decod-Pres-Sem}
The forcing relation for finite conversely well-founded $(\Delta,\Gamma)$-based provability models 
is decidable\footnote{In other words, one can decide $\pcal,w\pmodels A$ for every $(\Delta,\Gamma)$-based provability model $\pcal$ and any world $w$ in $\pcal$ whenever $(\Delta,\sft)$ is recursively downward compact and $\sft$ is strongly sound for $\pcal$ and $\sft\supseteq\IPC$.} whenever 
$(\Delta,\sft)$ is recursively downward compact, $\sft$ is strongly sound for $\pcal$, and $\sft\supseteq\IPC$.
\end{theorem}
\begin{proof}
Let $\kcal=(W,\pce,\R,\V)$ and $\cpmp\kcal\Delta\Phi$ be a $(\Delta,\Gamma)$-based provability model.
We show the decidability of $\pcal,w\pmodels A$ by double induction: first on $w$, ordered by $\sqsupset$, and second on $A$.
As the first induction hypothesis, assume that for every $u\sqsupset w$ and every $B\in\lcalb$, we have decidability of $\pcal,u\pmodels B$.
As the second induction hypothesis, assume that for every $u\sce w$, we have decidability of $\pcal,u\pmodels B$ for every $B$ that is a strict subformula of $A$.
In the following cases for $A$, we show decidability of $\pcal,w\pmodels A$:
\begin{itemize}
\item $A$ is atomic. Obvious.
\item $A$ is a conjunction, disjunction, or implication. Use the second induction hypothesis.
\item $A=\Box B$. It suffices to decide $\Delta_u\vdash \varphi_u\to B$ for every $u\sqsupset w$. Since $(\Delta,\sft)$ is recursively downward compact, one can effectively compute $\ap\Delta\sft{\varphi_u\to B}$. By the definition of $\ap\Gamma\sft{.}$, it is enough to decide $\Delta_u\vdash \ap\Delta\sft{\varphi_u\to B}$, which is equivalent to $\pcal,u\pmodels \ap\Delta\sft{\varphi_u\to B}$. Now the first induction hypothesis implies decidability of $\pcal,u\pmodels \ap\Delta\sft{\varphi_u\to B}$. \qedhere
\end{itemize}
\end{proof}

\begin{corollary}\label{Decod-Pres-Sem2}
Let $\pcal$ be a finite and conversely well-founded $(\Delta,\Gamma)$-based provability model. Then the forcing relation $\pcal,w\pmodels A$ is decidable for every $A\in\lcalb$ and every world $w$ in $\pcal$ in either of the following cases:
\begin{itemize}
\item $\Delta=\snnilb$.
\item $\Delta=\snnil$.
\end{itemize}
\end{corollary}
\begin{proof}
Direct consequence of \Cref{Decod-Pres-Sem,SNB-approx-prop,SN-approx-prop,sound-const1}.
\end{proof}

\subsection{Completeness theorems for $\lles$ and $\iglcphb$}
\begin{theorem}\label{provability-models-gen-completeness}
$\lles$ is complete for $\snnil$-based atomic ascending good provability models.
\end{theorem}
\begin{proof}
Let $\lles\nvdash A$. Then, by \Cref{sdc,SN-approx-prop}, we also have $\lles\nvdash \iap\snnil\iglca{A}$, and hence $\iglca\nvdash \iap\snnil\iglca{A}$. \Cref{Kripke-completeness-igl} implies that there exists a good atomic ascending Kripke model $\kcal:=(W,\pce,\R,\V)$ such that $\kcal,w\nVdash \iap\snnil\iglca{A}$. Let $\pcal:=\cpm\kcal\snnil$ be the $\snnil$-based provability model, as defined in \cref{constrction-of-pro-models}. Also, by definition, it is clear that $\pcal$ is an atomic ascending good provability model. Then \Cref{prov-kripke-normal-forms2} implies that $\pcal\npmodels \iap\snnil\iglca{A}$. On the other hand, \Cref{sound-const2} implies that $\pcal\pmodels\lles$. By \Cref{sdc,SN-approx-prop}, we have $\lles\vdash \iap\snnil\iglca{A}\lr A$. Thus $\pcal\npmodels A$, as desired.
\end{proof}

\begin{theorem}\label{provability-models-gen-completeness2}
 $\iglcphb$ is complete for $\snnilb$-based good provability models.
\end{theorem}
\begin{proof}
Let $\iglcphb\nvdash A$. Then, by \Cref{sdc,SNB-approx-prop}, we also have $\iglcphb\nvdash \iap\snnilb\iglcp{A}$, and hence $\iglcp\nvdash \iap\snnilb\iglcp{A}$. \Cref{Kripke-completeness-igl} implies that there exists a good Kripke model $\kcal:=(W,\pce,\R,\V)$ such that $\kcal,w\nVdash \iap\snnilb\iglcp{A}$. Let $\pcal:=\cpm\kcal\snnilb$ be the $\snnilb$-based provability model, as defined in \cref{constrction-of-pro-models}. Also, by definition, it is clear that $\pcal$ is a good provability model. Then \Cref{prov-kripke-normal-forms2} implies that $\pcal\npmodels \iap\snnilb\iglcp{A}$. On the other hand, \Cref{sound-const2} implies that $\pcal\pmodels\iglcphb$. By \Cref{sdc,SNB-approx-prop}, we have $\iglcphb\vdash \iap\snnilb\iglcp{A}\lr A$. Thus $\pcal\npmodels A$, as desired.
\end{proof}

\subsection{Completeness theorem for $\iglcph$}\label{prov-models-compl-iglh}

In this section, we prove the completeness of $\iglcph$ for good $(\snnilb,\cdsnb)$-based provability models. This completeness result will be helpful in proving ``$\iglcph\nvdash A$ implies $\lles\nvdash \gamma(A)$" (\Cref{PLHA-reduction}), which itself implies the arithmetical completeness of $\iglcph$. We will also use the completeness of $\iglcph$ for preservativity semantics in \Cref{decidability-iglh} to establish the decidability of $\iglcph$.
\\
We say that $Y\subseteq\lcalb$ is \textit{$(\Gamma,\sft)$-adequate} if
\begin{itemize}
	\item $\bot\in Y$ and $Y$ is closed under subformulas.
	\item If $A,B\in Y$, then $\ap{\sft}{\Gamma} {A\wedge B}\in Y$. More precisely, it means that for every $A,B\in Y$, there exists some $C\in Y\cap\Gamma$ such that $\tvdash C\to (A\wedge B)$ and for every $E\in\Gamma$ with $\tvdash E\to (A\wedge B)$, we have $\tvdash E\to C$.
\end{itemize}
Also, a set $\Delta$ is called \textit{$Y$-saturated w.r.t.~$\sft$} iff
\begin{itemize}
\item $\Delta\subseteq Y$,
\item if $\Delta\vdashsub\sft B$ and $B\in Y$, then $B\in \Delta$,
\item $\Delta\nvdasht\bot$,
\item $B\vee C\in \Delta$ implies either $B\in \Delta$ or $C\in \Delta$.
\end{itemize}
In the remainder of this subsection, we use the simplified notation $\ap{}{}A$ for $\ap{\iglcp}{\cdsnbv} A$.

Let us define the following notions, which we need for the completeness of $\iglcph$ with respect to provability models.

Let us fix a frame $\fcal=(W,\pce,\R)$ for intuitionistic modal logics.
We say that $u$ is an \textit{immediate predecessor} of $w$, denoted $u\ipred w$, if
\begin{itemize}
\item $u\R w$.
\item $v\R w$ implies $v\Rvar u$.\footnote{Recall that $\Rvar$ is the transitive closure of ${\sqsubseteq}\cup{\pce}$. Note that in a transitive frame for intuitionistic modal logic, $\Rvar$ is equal to ${\sqsubseteq}\circ{\pce}$.}
\end{itemize}
Note that in an irreflexive and transitive frame, an immediate predecessor is unique: let $u_1,u_2\ipred w$. Then, by the definition of $\ipred$, we have both $u_1\Rvar u_2$ and $u_2\Rvar u_1$. Since ${\Rvar}={{\sqsubseteq}\circ{\pce}}$, we have $u_1({{\sqsubseteq}\circ{\pce}})u_2$ and $u_2({{\sqsubseteq}\circ{\pce}})u_1$. Since $\pce$ is a partial order, at least one of the $\sqsubseteq$'s must be strict, i.e., $\R$. Without loss of generality, assume that $u_1({{\R}\circ{\pce}})u_2$ and $u_2({{\sqsubseteq}\circ{\pce}})u_1$. This implies that $u_1\R v\pce u_1$ for some $v$. This means that $v\R v$, contradicting irreflexivity.

On the other hand, if $\fcal$ is an $\R$-tree, transitive, and conversely well-founded, then any $\R$-accessible $w$ has an immediate predecessor: First, observe that in such a case, $\Rvar$ is also conversely well-founded. Then, since $\fcal$ is an $\R$-tree, the set $W_w:=\{u\in W: u\R w\}$ is linearly ordered by $\Rvar$. Thus, by the converse well-foundedness of $\Rvar$, there exists some $\Rvar$-maximum element $v\in W_w$, which, by definition, is also an immediate predecessor of $w$.

Then define $w\sim u$ iff either $w=u$ or there exists a joint immediate predecessor of $u$ and $w$; i.e., there exists some $v$ such that $v\ipred w$ and $v\ipred u$.
Finally, define $w\Rv u$ iff there exists some $v$ such that $w\sim v$ and $v\R u$.

\begin{lemma}\label{well-found-Rv}
For every conversely well-founded transitive $\R$-tree frame $\fcal=(W,\pce,\R)$, the relation $\Rv$ on $W$ is also conversely well-founded.
\end{lemma}
\begin{proof}
Let $w_0\Rv w_1\Rv w_2\Rv \ldots$ be an infinite sequence. It suffices to define an infinite sequence $v_1\Rvarr v_2\Rvarr\ldots$, where $\Rvarr$ is defined as ${\R}\circ{\pce}$.
We define $v_n$ as the immediate predecessor of $w_n$. Note that for every $n>0$, $w_n$ is $\R$-accessible and hence has a unique immediate predecessor $v_n\ipred w_n$.
We need to show $v_n\Rvarr v_{n+1}$ for every $n>0$. Since $w_n\Rv w_{n+1}$, there exists some $u$ such that $v_n\R u\R w_{n+1}$. Hence, by $v_{n+1}\ipred w_{n+1}$, we get $u\Rvar v_{n+1}$. Thus, by $v_n\R u\Rvar v_{n+1}$, we get $v_n\Rvarr v_{n+1}$, as desired.
%
\end{proof}

\begin{theorem}[\textbf{Completeness}]	
	\label{completeness-llea-pres-semant}
	$\iglcph$ is complete for  
	good $(\snnilb,\cdsnb)$-based locally sound provability models.
\end{theorem}
\begin{proof}
	Assume $\iglcph\nvdash A$.  
	By \Cref{Lem-cdsnb-adequate}, there exists a finite 
	set $\zed\ni A$ which is $(\cdsnbv,\iglcp)$-adequate.  
	Note that this implies $A \in \nf\zed$.
	Define $\vay:=\zed\cup\Box\zed$ and 
\emli	$$\chi:= \bigwedge_{E,F\in \zed} \Box (E\wedge F)\to \Box \ap{}{} {E\wedge F}.$$
	Observe that $\iglcph\vdash \chi$. Since $\iglcph\nvdash A$, it follows that 
	$\iglcp\nvdash\chi\to A$. Then \Cref{Kripke-completeness-igl} implies the existence of a good 
	Kripke model $\kcal=(W,\pce,\R,\V)$ such that $\kcal,w_0\Vdash \chi$ and $\kcal,w_0\nVdash A$, where $w_0$ is the root of $\kcal$.
	
	Define $\pcal:=\cpmp\kcal\snnilb\Phi$\footnote{%
		Recall from \cref{constrction-of-pro-models} that 
		$\cpmp\kcal\snnilb\Phi:=(W,\pce,\R,\{\lgc w\}_{w\in W},\V)$ with
		$\lgc w:=\IPC+\varphi_w+\snnilb_w$, where $\snnilb_w:=\{E\in\snnilb: \pcal,w\pmodelsp E\}$. 
	} where $\Phi:=\{\varphi_w\}_{w\in W}$ is defined by recursion on $w$ ordered by $\vR$ (the inverse of $\Rv$, defined earlier in this subsection). 
	By \Cref{well-found-Rv}, $(W,\Rv)$ is converse well-founded, ensuring the recursion is well-defined.
	Assume that for any $u\vR w$, $\varphi_u$ has already been defined. 
	Consequently, $\pcal,w\pmodelsp B$ and $\pcal,w\pmodels B$ are already defined for every formula $B$.
	Define
	\emli
	$$\varphi'_w:=\bigwedge (\cdsnbv\cap \zeds w)
	\quad\text{where}\quad 
	\zeds w:=\{B\in \zed: \pcal,w\pmodelsp B\}.
	$$
	We shall later verify that $\varphi'_w$, as defined, belongs to $\cdsnbv$. 
	Hence (modulo $\iglcp$-provable equivalence) $\varphi'_w=\bigvee \Gamma_w$ for some finite $\Gamma_w\subseteq\cdsnb$. Choose  
	$\varphi_w\in\Gamma_w$ such that $\pcal,w\pmodels\varphi_w$. From this definition, it is evident that $\pcal$ is locally sound.
	
	Clearly, $\pcal$ as defined is converse well-founded, finite, irreflexive, and transitive (inherited from $\kcal$), and constitutes a locally sound $\snnilb$-full provability model. Furthermore, by \Cref{sound-const1}, $\iglcp$ is strongly sound for $\pcal$, meaning that for every $w\in W$, the theory $\lgc w$ includes $\iglcp$.
	
	We will show that $\pcal$ is $\zed$-full; then by \Cref{prov-kripke-normal-forms} we obtain $\pcal\npmodels A$.  
	Thus it remains to prove by induction on $w$ (ordered by $\vR$) that: 
	(1) $\varphi'_w\in\cdsnbv$, and 
	(2) $\lgc w\vdash \bigwedge\zeds w$.
	
	Assume as the induction hypothesis that for every $v\in W$ with $w\Rv v$ (and hence also for every $v\sqsupset w$), we have 
	(1) $\varphi'_v\in\cdsnbv$, and 
	(2) $\lgc v\vdash \bigwedge\zeds v$.
	Consider the Kripke model $\pcal_{v}$, defined as the restriction of $\pcal$ to nodes accessible (via the transitive closure of ${\pce }\cup{\R }$) from $v$. 
	Clearly, $\pcal, u\Vdash B$ iff $\pcal_{v}, u\Vdash B$ for every node $u$ accessible from $v$ and every formula $B$. 
	It is straightforward to observe that for every $v\in W$ with $w\Rv v$, the induction hypothesis implies that $\pcal_{v}$ is a $\zed$-full $(\snnilb,\cdsnb)$-based model. Hence, by \Cref{prov-kripke-normal-forms}, for every $B\in \vay $ and every $u$ in $\pcal_v$, we have $\kkcal , u\Vdash B$ iff $\pcal, u\Vdash B$. Thus,
	\begin{center}
		(*) \quad 
		For every $B\in \vay $ and every $v$ with $w\Rv v$, we have $\kkcal , v\Vdash B$ iff $\pcal, v\Vdash B$.
	\end{center}
	We now prove the two statements:
	\begin{itemize}[leftmargin=*]
		\item $\varphi'_w\in\cdsnbv$. We first demonstrate that if $E,F\in \zeds w\cap\cdsnbv$, then there exists some $G\in \zeds w\cap\cdsnbv$ such that $\iglcp\vdash G\to (E\wedge F)$. Since $\zed$ is finite, there must be a single $G\in \zeds w\cap\cdsnbv$ such that $G\vdashsub\iglcp \zeds w\cap\cdsnbv$; consequently, $G$ is $\iglcp$-equivalent to $\varphi'_w$. Thus $\varphi'_w\in\cdsnbv$, as required.  
		
		Assume $E,F\in \zeds w\cap\cdsnbv$. 
		Then (*) implies $\kkcal,w\vdashp E\wedge F$.
		By the $(\cdsnbv,\iglcp)$-adequacy of $\zed$, we have $\ape{E\wedge F}\in \zed$. Let $w'\in W$ be the unique immediate predecessor of $w$. Hence $\kkcal,w'\Vdash \Box E\wedge \Box F$. 
		Since ${\iglcph}\vdash \Box(E\wedge F)\to \Box\ape{E\wedge F}$, we obtain $\kkcal,w\vdashp \ape{E\wedge F}$. 
		Thus (*) implies $\pcal,w\pmodelsp \ape{E\wedge F}$, whence $\ape{E\wedge F}\in \zeds w\cap\cdsnbv$ and $\iglcp\vdash \ape{E\wedge F}\to (E\wedge F)$.
		
		\item $\zed$-fullness.  
		Let $B\in \zed $ such that $\pcal, w\pmodelsp B$; we aim to show $ \lgc w\vdash B$. Since $\vdash \varphi_w\to\varphi'_w$, it suffices to show $\varphi'_w\vdashsub\iglcp B$. 
		Given $\pcal,w\pmodelsp B$, statement (*) implies $\kkcal,w\vdashp B$. 
		By definition of $\zed$, we have $\ape B\in \zed$. 
		Since $\iglcph\vdash \Box B\to\Box\ape B$, we obtain $\kkcal,w\vdashp\ape B$. 
		Since $\ape B\in\cdsnbv$, it follows that $\vdash \varphi'_w\to \ape B$ and hence $\iglcp\vdash \varphi'_w\to B$.
		\qedhere
	\end{itemize}
\end{proof}

\begin{lemma}\label{Lem-cdsnb-adequate}
	For every formula $A\in\lcalb$, there exists a finite set $Y \ni A$ (modulo $\ikfourcp$-provable equivalence) which is $(\cdsnbv ,\iglcp)$-adequate.  
	Moreover, such a finite set can be effectively computed.
\end{lemma}
\begin{proof}
	We only establish existence; effectiveness is evident from the construction.
	First define $Y_0$ and $Y_1$, then set
	\emli
	$$\displaystyle Y':=Y_0\cup Y_1\cup \{B\wedge C: B\in Y_0, C\in Y_1\}
	\quad \text{ and } \quad 
	Y:=\{\bigvee X: X\fsubeq Y'\}.$$
	\begin{itemize}
		\item[$Y_0$:] Let $X:=\subab A$ and define
		\emli
		$$Y_0:=\{B\in\lcalz(X): \cto(B)\leq\ctob(A)\}.$$ 
		Recall that $\cto$ counts the number of nested implications outside boxes, and $\ctob(A)$ is defined as the maximum of $|\subab A|$ and $\max\{\cto(B):\Box B\in\sub{A}\}$.
		\Cref{Remark-Finiteness-c(A)} implies that $Y_0$ is finite up to $\IPC$-provable equivalence.
		
		\item[$Y_1$:] Let $Z:=\subpb{A}\cup \{\Box B: B\in Y_0\}$. 
		By \Cref{Remark-NNIL-finiteness}, the set $Y'_1:= {\NNIL(Z)}$ is finite up to $\ikfourcp$-provable equivalence. 
		Define $Y_1:=\{\sggt B: B\in Y'_1\}$.
	\end{itemize}
	We show that $Y$ is $(\cdsnbv,\iglcp)$-adequate. 
	Since $Y_0$ and $Y_1$ are closed under subformulas, $Y$ is also closed under subformulas. Given $B,C\in Y'$, we demonstrate that a greatest lower bound for $B\wedge C$ with respect to $(\cdsnbv,\Gamma)$ belongs to $Y$. Assume $B=E\wedge E'$ and $C=F\wedge F'$ with $E,F\in Y_0$ and $E',F'\in Y_1$. 
	Since $E',F'\in \snnilb$, we have $\ape{B\wedge C}=\ape{E\wedge F}\wedge (E'\wedge F')$. 
	It suffices to show $\ape{E\wedge F}\in Y$. 
	Since $E\wedge F\in Y_0$, \Cref{tech-lem} yields the desired result. 
\end{proof}

\begin{lemma}\label{tech-lem}
	If $B\in Y_0$, then $\ap{}{}{B}\in Y$.
\end{lemma}
\begin{proof}
	Let $\Tro B$ be defined as in \Cref{def-glb-cdsnbv}. 
	\Cref{lem-cdlsnbv-glb-cto,lem-cdsnbv-glb-complexity} imply $\Tro B\subseteq Y_0$.
	Moreover, \Cref{Theorem-T-nnilb-Finitary0} implies that $\bigvee\Tro B$ is a greatest lower bound for $B$ with respect to $(\dnbv,\iglcp)$.
	\Cref{glb-cdlsnb-dlnb2} implies that $\bigvee\hTro B$ is a greatest lower bound for $B$ with respect to $(\cdsnbv,\iglcp)$, i.e., $\ap{\cdsnbv}{\iglcp}{B}=\bigvee\hTro B$.
	As defined in \Cref{glb-cdlsnb-dlnb2}, we have 
	$\hTro B:=\{D\wedge\Box D\wedge \sggt{(D^\dagger)}: D\in \Tro B\}$; thus $\ap{}{}{B}\in Y$.
\end{proof}

\subsection{Decidability of $\iglcph$}\label{decidability-iglh}
This section establishes the decidability of $\iglcph$. 
The proof employs the finite model property obtained from \Cref{completeness-llea-pres-semant}.

Given a formula $A\in\lcalb$, we must decide whether $\iglcph\vdash A$.  
First compute $n_A$ and $\Gamma_A$ as provided by \Cref{rem-completeness-llea-pres-semant}.  
Then check the validity of $A$ in the root of every good $(\snnilb,\cdsnb)$-based provability model $\pcal= \cpmp\kcal\snnilb\Phi$ with $\Phi=\{\varphi_w\}_{w\in W}$ and $\kcal=(W,\pce,\R,\V)$, where $|W|\leq n_A$ and $\varphi_w\in \Gamma_A$ for every $\R$-accessible $w\in W$. 
\Cref{Decod-Pres-Sem2} implies that validity in $\pcal$ is decidable. If all such models satisfy $A$, return \textit{yes}; otherwise, return \textit{no}.
Thus we obtain the following decidability result:

\begin{theorem}\label{iglh-decidable}
	$\iglcph$ is decidable.
\end{theorem}

\begin{lemma}\label{rem-completeness-llea-pres-semant}
	Given $A\in\lcalb$ with $\iglcph\nvdash A$, there exist $n_A\in\nat$ and a finite set $\Gamma_A\subseteq\cdsnb$ such that:
	\begin{itemize}
		\item $n_A$ and $\Gamma_A$ are effectively computable,
		\item there exists a $(\snnilb,\cdsnb)$-based provability model $\pcal= \cpmp\kcal\snnilb\Phi$ with $\Phi=\{\varphi_w\}_{w\in W}$, $\kcal=(W,\pce,\R,\V)$, $\pcal\npmodels A$, $|W|\leq n_A$, and $\varphi_w\in \Gamma_A$ for every $\R$-accessible $w\in W$.
	\end{itemize} 
\end{lemma}
\begin{proof}
	Let $\Gamma_A:=Y_A\cap\cdsnb$, where $Y_A$ is as provided by \Cref{Lem-cdsnb-adequate}. The number $n_A$ and the model $\kcal$ can be directly inferred from the proof of \Cref{completeness-llea-pres-semant}.
\end{proof}

\section{Provability logic of $\HA$: arithmetical completeness}
\label{sec-completeness-reduction}
This section proves that $\iglcph$ (see \Cref{propo-logics}) is the provability logic of $\HA$. 
Soundness has already been established by \Cref{iglh-in-iph,iph-sound}; it remains to prove arithmetical completeness. This is achieved by propositional reduction to the completeness of $\lles$ for $\Sigma_1$-substitutions.

Historically, the provability logic of Peano Arithmetic, $\PA$, was discovered \citep{Solovay} before the $\Sigma_1$-provability logic of $\PA$ \citep{VisserThes}. The method in \citep{VisserThes} essentially employs Solovay's technique. Subsequently, \citep{reduction} demonstrated that, in a sense, the $\Sigma_1$-provability logic of $\PA$ is more complex than the standard provability logic of $\PA$. Later, \citep{mojtahedi2021hard} studied reductions between provability logics and characterized several others. Most notably, \citep{mojtahedi2021hard} showed that the $\Sigma_1$-provability logic of $\HA$ relative to the standard model is the most complex known provability logic. 
Here, in \Cref{PLHA-reduction}, we show that the $\Sigma_1$-provability logic of $\HA$ \citep{Sigma.Prov.HA,Jetze-Visser} is more complex than the standard provability logic of $\HA$; in other words, we reduce ``completeness of $\iglcph$ for arithmetical interpretations in $\HA$'' to ``completeness of $\lles$ for arithmetical $\Sigma_1$ interpretations in $\HA$''.
\begin{theorem}[\textbf{Reduction}]
	\label{PLHA-reduction}
	If $\iglcph\nvdash A$, then $\lles\nvdash\theta(A)$ for some substitution $\theta$.
\end{theorem}
\begin{proof}
	Assume $\iglcph\nvdash A$. 
	\Cref{Reduction-llea-lleb} yields a substitution $\gamma$ such that $\iglcphb\nvdash \gamma(A)$. Then by \Cref{Reduction-lleb-lles}, $\lles\nvdash\beta(\gamma(A))$ for some substitution $\beta$. Setting $\theta:=\beta\circ\gamma$ completes the proof.
\end{proof}

\begin{corollary}[\textbf{Arithmetical Completeness}]
\label{arith-comp}
$\iglcph$ is complete for arithmetical interpretations in $\HA$; that is, if $\HA\vdash\aha(A)$ for every arithmetical substitution $\alpha$, then $\iglcph\vdash A$.
\end{corollary}
\begin{proof}
	We argue contrapositively. Assume $\iglcph\nvdash A$ and seek an arithmetical substitution $\alpha$ such that $\HA\nvdash\aha(A)$.
	By $\iglcph\nvdash A$ and \Cref{PLHA-reduction}, there exists a propositional substitution $\theta$ such that $\lles\nvdash \theta(A)$. Then \Cref{Theorem-Sigma-Provability-HA} yields an arithmetical substitution $\sigma$ such that $\HA\nvdash\sigma_{_{\sf HA}}(\theta(A))$. Thus $\alpha:=\sigma\circ \theta$ completes the proof.
\end{proof}

The above argument imposed an extra restriction on arithmetical substitutions, namely that atomic parameters $p_i$ are replaced by $\Sigma_1$-sentences. However, by setting $\parr=\emptyset$ in the following theorem, we restate the previous result in a more familiar setting. With $\parr=\emptyset$, the additional requirement on arithmetical interpretations of parameters disappears.
\begin{theorem}\label{jhgd}
$\iglcph$ is the provability logic of $\HA$.  
\end{theorem}
\begin{proof}
	Soundness again follows from \Cref{iglh-in-iph,iph-sound}.
	For completeness, assume $\iglcph\nvdash A$ for some $A\in\lcalb$. Then  
	\Cref{arith-comp} yields the required arithmetical substitution.  
\end{proof}

Recall the definitions of $\iph$ and $\iphp$ from \Cref{sec-preser}, 
and define $\iphb$ ($\iphpb$) as the fragment of $\iph$ ($\iphp$) in  the language 
$\lcalb$. As a corollary to arithmetical completeness, we have:
\begin{corollary}\label{iph-iglh}
$\iglcph=\iphb=\iphpb$.
\end{corollary}
\begin{proof}
	We show $\iglcph\subseteq\iphb\subseteq\iphpb\subseteq\iglcph$.
	\begin{itemize}[leftmargin=*]
		\item $\iglcph\subseteq\iphb$: \Cref{iglh-in-iph}.
		\item $\iphb\subseteq\iphpb$: Trivial.
		\item $\iphpb\subseteq\iglcph$: 
		We argue contrapositively. Assume $\iglcph\nvdash A$. Then \Cref{jhgd} implies $\HA\nvdash \aha(A)$. Hence \Cref{iph-sound} yields $\iphp\nvdash A$. \qedhere
	\end{itemize}
\end{proof}
The proof of the above corollary employs arithmetical soundness and completeness theorems. However, the equality $\iglcph=\iphb=\iphpb$ invites a propositional proof without recourse to arithmetical interpretations. Thus we pose:
\begin{question}\label{q-3}
	Does there exist a translation $(.)^{\sf t}:\lcalp\to\lcalb$ with the following properties?
	\begin{itemize}
		\item If $\iphp\vdash A$, then $\iglcph\vdash A^{\sf t}$ for every $A\in\lcalp$.
		\item $\iglcph\vdash A\lr A^{\sf t}$ for every $A\in\lcalb$.
	\end{itemize}
\end{question}
If such a translation exists, \Cref{iph-iglh} can be proved without invoking arithmetical soundness-completeness theorems:
\begin{proof}[Proof of \Cref{iph-iglh}]
	The inclusions $\iglcph\subseteq\iphb\subseteq\iphpb$ hold as before.
	For $\iphpb\subseteq\iglcph$, assume $\iphp\vdash A$ for $A\in\lcalb$. Then $\iglcph\vdash A^{\sf t}$, and since $\iglcph\vdash A\lr A^{\sf t}$, we obtain $\iglcph\vdash A$.
\end{proof}

\subsection{First reduction step: $\iglcph\nvdash A$ implies $\iglcphb\nvdash\gamma(A)$}\label{first-red}
This subsection proves \Cref{Reduction-llea-lleb}: if $\iglcph\nvdash A$, then $\iglcphb\nvdash\gamma(A)$ for some substitution $\gamma$. All subsequent technical lemmas are used solely in the proof of \Cref{Reduction-llea-lleb}. For notational simplicity, throughout this section we denote derivability in $\iglcp$ by $\vdash$, unless otherwise stated.

This reduction transforms $(\snnilb,\cdsnb)$-based provability models into $\snnilb$-based models. Before delving into the detailed construction, we outline the underlying ideas. Suppose $\pcal\npmodels A$ for some $(\snnilb,\cdsnb)$-based provability model $\pcal=\cpmp\kcal{\snnilb}{\Phi}$ with $\kcal=(W,\pce,\R,\V)$ and $\Phi=\{\varphi_w\}_{w\in W}$. Since each $\varphi_w\in\cdsnb$, there exists a projective substitution $\theta_w$ such that $\varphi_w\xrat{\theta_w}{\iglcp} \snnilb$. We use these substitutions systematically to replace each $\varphi_w$ by its $\snnilb$-projection. One obstacle is that, by the definition of projectivity, we work with the outer substitution $\th_w$, not $\theta_w$ itself, meaning $\th_w$ is the identity on boxed formulas. We employ the simultaneous fixed-point theorem in $\iglcp$ (\Cref{sec-fix}) to simulate this condition.

\textbf{Important notation:} Unless stated otherwise, throughout \Cref{first-red,second-red}, $\vdash$ denotes derivability in $\iglcp$ (which lacks the necessitation rule $A\nvdash \Box A$).

\begin{theorem}\label{Reduction-llea-lleb}
	If $\iglcph\nvdash A$, then there exists a substitution $\gamma$ such that $\iglcphb\nvdash \gamma(A)$.
\end{theorem}
\begin{proof}
	We first define several notions. Given a provability model $\pcal=\cpmp{\kcal}{\Delta}{\Phi}$ with $\Phi=\{\varphi_w\}_{w\in W}$ and $\kcal=(W,\pce,\R,\V)$, we say $w\in W$ is an $\Ha$-node if $\varphi_u\neq\top$ for some $\R$-accessible $u$ with $u\sqsubseteq w$. Define a complexity measure $\dfrak A$ for any $A$ with $\iglcph\nvdash A$ as follows: $\dfrak A$ is the minimum number $d$ such that there exists a good $(\snnilb,\cdsnb)$-based locally sound provability model $\pcal$ with exactly $d$ $\Ha$-nodes. By \Cref{completeness-llea-pres-semant}, such a provability model indeed exists.
	
	We proceed by induction on $\dfrak A$. If $\dfrak A=0$, then the provability model $\pcal$ is also $\snnilb$-based, and \Cref{sound-const2} implies $\iglcphb\nvdash A$. Thus setting $\gamma$ as the identity substitution suffices.
	
	As induction hypothesis, assume that for every $A'$ with $\dfrak{A'} < \dfrak A$, there exists a substitution $\gamma$ such that $\iglcphb\nvdash\gamma(A')$. Then it suffices to find a substitution $\gamma$ with $\dfrak{\gamma(A)} < \dfrak A$.
	
	Let $\pcal^0=\cpmp{\kcal^0}{\snnilb}{\Phi^0}$ with $\kcal^0=(W,\pce,\R,\V^0)$ and $\Phi^0:=\{\varphi^0_w\}_{w\in W}$ be a good $(\snnilb,\cdsnb)$-based provability model with root $w_0$ such that $\pcal^0,w_0\npmodels A$ and the number of $\Ha$-nodes in $\pcal^0$ is $\dfrak A$. Note that whenever $w\R u$, we have $\vdash \varphi^0_u\to\varphi^0_w$.
	
	Fix an $\R$-minimal $\Ha$-node $\wfix\in W$; that is, every $w'\R \wfix$ is not an $\Ha$-node. Observe that if $w$ is not an $\Ha$-node, then $\varphi^0_w=\top$.\footnote{If $w$ is not $\R$-accessible, it is not an $\Ha$-node even if $\varphi_w\neq \top$. However, since $\lgc w$ is irrelevant for non-$\R$-accessible $w$, we may assume from the outset that $\varphi_w=\top$ for every such $w$, without affecting formula evaluation.}
	
	Define $W_0:=\{w_0\}$ and $W_{i+1}:=\{w\in W: \text{the immediate predecessor of } w \text{ is in } W_i\}$.\footnote{For the definition of immediate predecessor, see \Cref{prov-models-compl-iglh}.} Since $\kcal$ is conversely well-founded, there exists a maximum $n\in\nat$ such that $W_n\neq\emptyset$. Set $W':=\bigcup_{i=0}^n W_i$ and $W'':=\bigcup_{i=1}^n W_i$. Let $\chi:=\varphi^0_\wfix$. By \Cref{cor-pro-models}, we may assume $\chi\in\sfc\dar\IPC\snnil$; thus there exists a $\chi$-projective (in $\IPC$) substitution $\theta$ such that $\IPC\vdash \th (\chi)\lr \chi^\dagger \in\snnilb$, $\IPC\vdash\chi \to(x\lr \theta(x))$ for every variable $x$, and $\theta(x)=x$ for every $x\nin\sub A$ (by \Cref{local-proj}).  
	
	Let  
	$$\BA:=\Box B_1,\ldots,\Box B_m$$ 
	enumerate all boxed subformulas of $\chi$. Further, let  
	$$\pvec:= p_1,\ldots,p_m$$ 
	and $q$ be fresh atomic parameters (pairwise distinct and not appearing in $A$, $\theta(x)$, or $\{\varphi^0_w,\varphi^\dagger_w\}_{w\in W''}$ for any $x\in\varr$).
	
	Let $\eta$ be a substitution with $\eta(p_i):=\Box B_i$ for $1\leq i\leq m$ and identity elsewhere. Let $\alpha$ be the parametric dual of $\theta$ in the language $\lcalz(\varr,\pvec)$; that is, $\alpha$ is a substitution\footnote{Recall that substitutions are identity on parameters: $\alpha(p)=p$ for any parameter $p$.} such that $\eta(\alpha(B))=\th(\eta(B))$ for every $B\in\lcalz(\varr,\pvec)$. Define $\chi'\in\lcalz(\varr,\pvec)$ and $\chi^\ddagger\in\lcalz(\pvec)$ such that $\eta(\chi^\ddagger)=\chi^\dagger$ and $\eta(\chi')=\chi$. Then clearly $\IPC\vdash \alpha(\chi')\lr \chi^\ddagger$ and $\IPC\vdash \chi'\to(x\lr \alpha(x))$ for every $x\in\varr$.
	
	Define the substitution:
	\emli
	$$
	\beta(a):=\begin{cases}
	(q \to \alpha (a))\wedge(\neg q \to a) & : a\in\varr\\
	a & :a\in\parr\cup\{\bot\}
	\end{cases}
	$$
	Let $\tau$ be the simultaneous fixed point of $\beta(\Box B_1),\ldots,  \beta(\Box B_m)$ with respect to $\pvec$, as provided by \Cref{Theorem-simultan-fixed-point}. That is, $\iglcp\vdash \tau(p_i)\lr \tau\beta(\Box B_i)$. Finally, set $\gamma:=\tau\circ\beta$.
	
	Define $\pcal^1:=\cpmp{\kcal^1}\snnilb{\Phi^1}$ with $\kcal^1:=(W,\pce,\R,\V^1)$, where $W,\pce,\R$ and $\Phi^1:=\Phi^0$ are as in $\pcal^0$, and $\V^1$ is defined by:
	\emli
	\begin{align*}
		w\V^1 p_i &\quad\text{iff}\quad \pcal^1,w\pmodels \Box B_i,\\
		w\V^1 q &\quad\text{iff}\quad \wfix\Rvar w \quad\text{(i.e., } \exists v\; \wfix\sqsubseteq v\pce w\text{)},
	\end{align*}
	and $w\V^1 a$ iff $w\V^0 a$ for every other atomic $a$. Note that this definition is recursive but valid because $(W,\R)$ is conversely well-founded. By induction on $w$ (ordered by $\sqsupset$), one easily observes that $\pcal^1,w\pmodels B$ iff $\pcal^0,w\pmodels B$ for every $B$ not containing $p_i$ or $q$. Hence $\pcal^1$ is locally sound and $\pcal^1,w_0\npmodels A$. Moreover, $\pcal^1$ is a good $(\snnilb,\cdsnb)$-based locally sound provability model sharing the same set of $\Ha$-nodes as $\pcal^0$.
	
	Finally, define $\pcal^2:=\cpmp{\kcal^2}\snnilb{\Phi^2}$ with $\kcal^2:=(W,\pce,\R,\V^2)$ and $\V^2$ and $\Phi^2:=\{\varphi^2_w\}_{w\in W}$ as follows:
	\emli
	\begin{align*}
		\varphi^2_w &:=
		\begin{cases}
			\varphi^0_w & : w\neq \wfix \\
			\top & : w=\wfix
		\end{cases}\\
		w\V^2 p_i &\text{ iff } \pcal^2,w\pmodels \Box \gamma(B_i),\\
		w\V^2 a &\text{ iff } w\V^1 a \text{ for every other atomic } a.
	\end{align*}
	Define $\snnilb^i_w:=\{E\in\snnilb: \pcal^i,w\pmodels E\}$ for $i\in\{0,1,2\}$. For later use, also define:
	\emli
	\begin{align*}
		\lgci i w &:=\IPC+\snnilb^i_w+\varphi^i_w \quad\text{for } i=0,1,2,\\
		\psi &:=\bigwedge\{\Boxdot(p_i\lr \Box B_i): 1\leq i\leq m\},\\
		\psi' &:=\bigwedge\{\Boxdot(p_i\lr \Box \gamma(B_i)): 1\leq i\leq m\}.
	\end{align*}
	The valuations in $\pcal^1$ and $\pcal^2$ are defined so that $\pcal^1\pmodels \psi$ and $\pcal^2\pmodels \psi'$.
	\Cref{local-kcal-Kripke} implies that $\pcal^2$ is also a good $(\snnilb,\cdsnb)$-based provability model, and \Cref{local-lem-4} yields $\pcal^2,w_0\npmodels \gamma(A)$. Consequently, $\dfrak{\gamma(A)} < \dfrak A$.
\end{proof}

The remainder of this subsection proves the technical lemmas required for the above theorem, namely \Cref{Reduction-llea-lleb}.  

Let $\lcalbp$ denote the set of formulas $B\in\lcalb$ such that $p_i\notin\sub B$ for every $p_i\in\pvec$. Define $\Delta_1\equiv_\Gamma\Delta_2$ by:
$$\forall\,E\in\Gamma\ (\Delta_1\vdash E \Leftrightarrow \Delta_2\vdash E).$$

Before proceeding, we note several observations that will be used implicitly in the proofs:
\begin{itemize}
	\item $\lgci 1 w\vdash p_i\lr \Box B_i$ and $\lgci 1 w\vdash p_i\lr \Box \gamma(B_i)$ for every $1\leq i\leq m$.
	\item $\wfix$ is not $\pce$-accessible because $\wfix$ is $\R$-accessible and $(W,\pce,\R)$ is a good frame.
	\item Every $\snnilb^i_w$ (and $\lgci i w$) includes $\iglcp$ for $i=0,1,2$ and $w\in W$, by \Cref{sound-const1}.
	\item $\pcal^1$ is a good $(\snnilb,\cdsnb)$-based provability model. Local soundness follows by showing that $\pcal^1$ and $\pcal^0$ are equivalent for the restricted language $\lcalbp$.
	\item $\lgci i w\vdash \chi$ for $i=0,1,2$ and $\wfix\sqsubseteq w$.
\end{itemize}

\begin{lemma}\label{local-lem-1.49}
	For every $w\not\sqsubset \wfix$ and $E\in\lcalbp$, we have $\pcal^2,w\pmodels E$ iff $\pcal^1,w\pmodels E$. 
	Moreover, for every $\R$-accessible $w$ with $w\not\sqsubseteq \wfix$, we have $\lgci 2 w \equiv_{\lcalb'}  \lgci 1 w$, and $\lgci 2 \wfix\cup\{\chi\} \equiv_{\lcalb'} \lgci 1 \wfix$.
\end{lemma}
\begin{proof}
	We first show that the first statement implies the others. Let $w\not\sqsubseteq \wfix$ and $\lgci 1 w\vdash E$ for some $E\in\lcalbp$. Then $\snnilb^1_w\vdash E$, so there exists $F\in\snnilb^1_w$ with $\vdash F\to E$. Since $E$ contains no parameters $p_i$, we also have $\vdash F'\to E$, where $F':=F[p_1:\Box B_1,\ldots,p_m:\Box B_m]$. Clearly $F'\in\lcalbp$. Because $\lgci 1 w\vdash p_i\lr \Box B_i$, we have $F'\in\snnilb^1_w$; by the first statement, $F'\in\snnilb^2_w$. Hence $\lgci 2 w\vdash E$.
	
	Conversely, let $w\not\sqsubseteq \wfix$ and $\lgci 2 w\vdash E$ for some $E\in\lcalbp$. Then $\snnilb^2_w\vdash E$, so there exists $F\in\snnilb^2_w$ with $\vdash F\to E$. Since $E$ contains no $p_i$, we also have $\vdash F'\to E$, where $F':=F[p_1:\Box \gamma(B_1),\ldots,p_m:\Box \gamma(B_m)]$. Clearly $F'\in\lcalbp$. Because $\lgci 2 w\vdash p_i\lr \Box \gamma(B_i)$, we have $F'\in\snnilb^2_w$; by the first statement, $F'\in\snnilb^1_w$. Thus $\lgci 1 w\vdash E$.
	
	The proof of $\lgci 2 \wfix\cup\{\chi\} \equiv_{\lcalb'} \lgci 1 \wfix$ is similar and left to the reader.
	
	We now prove the first statement by a double induction: first on $W$ ordered by $\sqsupset$, second on the complexity of $E\in\lcalb'$. As the first induction hypothesis, assume that for every $u\sqsupset w$ and every $E\in\lcalb'$, we have $\pcal^1,u\pmodels E$ iff $\pcal^2,u\pmodels E$ (hence $\lgci 2 u \equiv_{\lcalb'} \lgci 1 u$). As the second induction hypothesis, assume that for every strict subformula $F$ of $E$ and every $u\not\sqsubset w_1$, we have $\pcal^1,u\pmodels F$ iff $\pcal^2,u\pmodels F$.
	We consider cases for $E$:
	\begin{itemize}[leftmargin=*]
		\item $E\in\atom\setminus\pvec$ or $E=\bot$: Trivial.
		\item $E$ is a conjunction, disjunction, or implication: Follows directly from the second induction hypothesis.
		\item $E=\Box F$: By the first induction hypothesis and the definitions of $\lgci 1 u$ and $\lgci 2 u$, we have $\lgci 1 u\equiv_{\lcalb'}\lgci 2 u$ for every $u\sqsupset w$. The definition of validity for $\Box F$ in provability models then yields the desired result. \qedhere
	\end{itemize}
\end{proof}

\begin{lemma}\label{local-kcal-Kripke}
$\pcal^2$ is a good $(\snnilb,\cdsnb)$-based provability model.
\end{lemma}
\begin{proof}
	\Cref{local-lem-1.49} implies that $\pcal^2$ is locally sound. All other required properties are inherited from $\pcal^1$.
\end{proof}

\begin{lemma}\label{local-lem-3}
	For every $B\in\lcalb$, $\psi'\vdash B\lr \tau(B)$.
\end{lemma}
\begin{proof}
	Straightforward induction on the complexity of $B$, left to the reader.
\end{proof}

\begin{lemma}\label{handy}
	For every $B\in\lcalb$, we have $\psi,\psi',\chi,q \vdash \gamma(B)\lr B$ and $\psi',\neg q\vdash \gamma(B)\lr B$.
\end{lemma}
\begin{proof}
	We only treat the first statement for $B=a\in\atom$; the remaining cases are similar. By definition of $\beta$, $q \vdash \beta(a)\lr\alpha (a)$. Since $\chi' \vdash \alpha(a)\lr a$, we have $q ,\chi' \vdash\beta(a)\lr a$. Then \Cref{local-lem-3} gives $\psi',q ,\chi' \vdash\tau\beta(a)\lr a$. Because $\psi$ implies that $\eta$ is the identity, we obtain $\psi,\psi',q ,\chi \vdash\gamma(a)\lr a$.
\end{proof}

\begin{lemma}\label{2-1}
	For every $\R$-accessible $w$ with $w\not\R \wfix$, we have $\lgci 2 w\vdash \psi\wedge \psi'$. Furthermore, for every $\R$-accessible $w\not\sqsubseteq \wfix$ and every $B\in\lcalb$, $\lgci 2 w\vdash \gamma(B)\lr B$.
\end{lemma}
\begin{proof}
	Note that \Cref{handy} and $\lgci 2 w\vdash \psi\wedge \psi'$ imply $\lgci 2 w\vdash \gamma(B)\lr B$ whenever $w\not\sqsubseteq \wfix$. Thus we only prove the first statement by induction on $w$ ordered by $\sqsupset$. Assume inductively that for every $u\sqsupset w$, we have $\lgci 2 u\vdash \psi\wedge\psi'$ (hence $\lgci 2 u\vdash \gamma(B)\lr B$ for every $B\in\lcalb$). We show $\lgci 2 w\vdash \psi\wedge \psi'$. Since $\psi\wedge\psi'\in\snnilb$, it suffices to show $\pcal^2,w \pmodels \psi\wedge\psi'$. That is, we must show $\pcal^2,w\pmodels \Boxdot (p_i\lr \Box B_i)\wedge \Boxdot (p_i\lr \Box \gamma(B_i))$. By the induction hypothesis, $\pcal^2,w\pmodels \Box (p_i\lr \Box B_i)\wedge \Box (p_i\lr \Box \gamma(B_i))$. By definition of $\V'$, we have $\pcal^2,w\pmodels p_i\lr \Box\gamma(B_i)$. The induction hypothesis also yields $\pcal^2,w\pmodels \Box(B_i\lr\gamma(B_i))$, hence $\pcal^2,w\pmodels \Box B_i\lr \Box\gamma(B_i)$. Thus $\pcal^2,w\pmodels p_i\lr\Box B_i$.
\end{proof}

\begin{lemma}\label{2-1-2}
	For every $\R$-accessible $w\not\R \wfix$, we have $\lgci 1 w\vdash \psi\wedge \psi'$. Moreover, for every $\R$-accessible $w\not\sqsubset \wfix$ and every $B\in\lcalb$, $\lgci 1 w\vdash \gamma(B)\lr B$.
\end{lemma}
\begin{proof}
	Similar to the proof of \Cref{2-1} and left to the reader.
\end{proof}

\begin{lemma}\label{local-lem-1.5}
	For every $B\in\lcalb$, $\vdash\gamma(B)\lr \gamma(\eta(B))$.
\end{lemma}
\begin{proof}
	Induction on the complexity of $B$. We only treat $B=p_i$; other cases are trivial. Since $\gamma:=\tau\circ\beta$ and $\beta$ is the identity on $p_i$, we have $\gamma(p_i)=\tau(p_i)$. Because $\tau$ is the simultaneous fixed point of $\pvec$, $\vdash \tau(p_i)\lr \tau\beta(\Box B_i)$, i.e., $\vdash \gamma(p_i)\lr \gamma(\eta(p_i))$.
\end{proof}

\begin{lemma}\label{local-lem-2}
	For every $\R$-accessible $w\in W$, $\lgci 2 w\vdash \gamma(\varphi^0_w)$.
\end{lemma}
\begin{proof}
	We consider cases:
	\begin{itemize}[leftmargin=*]
		\item $w\R \wfix$: Since $\wfix$ is $\R$-minimal among $\Ha$-nodes, we have $\varphi^0_w=\top$; thus the claim holds trivially.
		\item $\wfix\R w$: Then $\lgci 2 w\vdash \varphi^0_w$, and \Cref{2-1} yields $\lgci 2 w\vdash \gamma(\varphi^0_w)$.
		\item $w=\wfix$: By definition of $\beta$, $q \vdash \beta(x)\lr\alpha (x)$. Since $\chi^\ddagger \vdash \alpha(\chi')$, we have $q ,\chi^\ddagger \vdash \beta(\chi')$. Then \Cref{local-lem-3} gives $\psi', q ,\chi^\ddagger \vdash \tau\beta(\chi')$. Because $\psi$ implies $\eta$ is the identity, we obtain $\psi,\psi', q ,\chi^\dagger \vdash\gamma(\chi')$. \Cref{local-lem-1.5} then yields $\psi,\psi', q ,\chi^\dagger \vdash\gamma(\chi)$. Since $\lgci 2 \wfix\vdash q$, we have $\psi,\psi',\chi^\dagger \vdash \gamma(\chi)$. By \Cref{2-1}, $\lgci 2 \wfix\vdash \psi,\psi'$. Moreover, $\pcal^1,\wfix\pmodels\chi^\dagger$ (because $\pcal^1,\wfix\pmodels\chi$ and $\iglcp\vdash \chi \to\chi^\dagger$), so \Cref{local-lem-1.49} gives $\pcal^2,\wfix\pmodels \chi^\dagger$. Since $\chi^\dagger \in\snnilb$, it follows that $\chi^\dagger \in \lgci 2 \wfix$. Hence $\lgci 2 \wfix \vdash \gamma(\chi)$.
		\item Otherwise: By \Cref{2-1}, $\lgci 2 w\vdash \psi'$. Since $\Boxdot \neg q \in \lgci 2 w$ and $\lgci 2 w\vdash \varphi^0_w$, \Cref{handy} yields the result. \qedhere
	\end{itemize}
\end{proof}

Define $\lbqw$ as the set of formulas $A\in\lcalb$ such that for every $\Box B\in\sub A$, either $B\in\lcalz(\parb)$ or $q \notin\subo B$. In other words, $\lbqw$ consists of formulas in which no variable besides $q$ occurs inside a box, unless it also occurs outside all boxes.

A set $Z$ of formulas is \textit{$\snnilb$-closed} if $Z\vdash E$ and $E\in\snnilb$ imply $E\in Z$.
\begin{lemma}\label{4-1}
	Let $Z\subseteq\snnilb$ be $\snnilb$-closed. If $C\in\lbqw$ and $Z\vdash C$, then $\vdash E\to C$ for some $E\in Z\cap\lbqw$.   
	Also, if $C\in\lbqw$ and $Z\vdash \gamma(C)$, then $\vdash \gamma(E)\to \gamma(C)$ for some $E\in \snnilb\cap \lbqw$ with $\gamma(E)\in Z$.
\end{lemma}
\begin{proof}
	First assume $Z\vdash C$. By $\snnilb$-closure of $Z$, there exists $F\in Z$ with $\vdash F\to C$. Since $F\in\snnilb$, the definition of $\ap\iglcp\snnilb C$ (denoted $\ape{C}$) gives $\ape {C}\in Z$. \Cref{APlus}.3 implies that for every $\Box D\in\sub {\ape C}$, either $\Box D\in\sub C$ or $\Boxdot D\in\snnilb$. Hence $C\in\lbqw$ implies $\ape C\in \lbqw$. Thus $E:=\ape C$ satisfies the requirements.
	
	Next assume $Z\vdash\gamma(C)$. Then there exists $F\in Z$ such that $\vdash F\to \gamma(C)$. Since $F\in\snnilb$, \Cref{APlus1,Theorem-NNIL-Pres2} imply $\vdash F\to \gamma(C)^\star$, $\vdash \gamma(C)^\star\to \gamma(C)$, and $\subob{\gamma(C)^\star}\subseteq  \subob{\gamma(C)}$. Consider $\Box D_0\in\subob{\gamma(C)^\star}$. Either $\Box D_0=\gamma(\Box D)$ for some $\Box D\in\subob{C}$, or $\Box D_0\in\subob{\gamma(x)}$ for some $x\in \subo{C}$. In the latter case, there exists $p_i\in\pvec$ such that $\Box D_0=\tau(p_i)$. Since $\tau$ is the simultaneous fixed point, $\vdash\Box D_0\lr \gamma(\Box B_i)$. Consequently, there exists $E_0\in \nnilb\cap\lbqw$ such that $\vdash \gamma(C)^\star \lr \gamma(E_0)$ and $\vdash \gamma(C)^\star \to \gamma(C)$, hence $\vdash \gamma(E_0)\to \gamma(C)$. Moreover, from $\vdash F\to \gamma(E_0)$ and $F\in\snnilb$, \Cref{igl-closure-box} yields $\vdash F\to \sggt{\gamma(E_0)}$. Let $E:=\sggt E_0$. Then $\gamma(E)\in Z$ and $E\in\snnilb\cap\lbqw$. Since $\vdash \gamma(E_0)\to\gamma(C)$ and $\vdash \sggt{\gamma(E_0)}\to\gamma(E_0)$, we obtain $\vdash \sggt{\gamma(E_0)}\to \gamma(C)$. Because $\gamma(E)=\sggt{\gamma(E_0)}$, we have $\vdash \gamma(E)\to\gamma(C)$.
\end{proof}

\begin{lemma}\label{4-2-2}
	Let $B\in \lcalz(\parb)$ such that for every $\Box E\in\subo B\cup\{\Box B_i: i\leq m\}$ and every $u\sqsupset w$, we have $\lgci 1 u\vdash E$ iff $\lgci 1 u\vdash \gamma(E)$. Then $\lgci 1 w\vdash B\lr\gamma(B)$.
\end{lemma}
\begin{proof}
	Induction on the complexity of $B$. All cases are straightforward except:
	\begin{itemize}[leftmargin=*]
		\item $B=p_i$ for some $i\leq m$: By assumption, for every $u\sqsupseteq w$, $\pcal^1,u\pmodels \Boxdot(\Box B_i\lr \Box\gamma(B_i))$. By definition of $\pcal^1$, $\pcal^1\pmodels \Boxdot(p_i\lr \Box B_i)$. Hence $\pcal^1,w\pmodels \Boxdot(p_i\lr \Box\gamma(B_i))$. Since $\Boxdot(p_i\lr \Box\gamma(B_i))\in\snnilb$, it belongs to $\snnilb^1_w$. Also, by definition of $\gamma$, $\gamma(p_i)=\tau(p_i)$, and the fixed-point property gives $\vdash\tau(p_i)\lr \Box\gamma(B_i)$, i.e., $\vdash \gamma(p_i)\lr \Box\gamma(B_i)$. Thus $\lgci 1 w\vdash p_i\lr\gamma(p_i)$.
		\item $B$ is a parameter other than $p_i$ or $B=\bot$: Then $\gamma(B)=B$, trivial.
		\item $B$ is a conjunction, disjunction, or implication: Follows from the induction hypothesis.
		\item $B=\Box C$: By assumption, for every $u\sqsupseteq w$, $\pcal^1,u\pmodels\Boxdot(\Box C\lr\Box \gamma(C))$. Since $\Boxdot(\Box C\lr\Box \gamma(C))\in\snnilb$, it belongs to $\snnilb^1_w$, so $\lgci 1 w\vdash \Box C\lr\Box \gamma(C)$. \qedhere
	\end{itemize}
\end{proof}

\Cref{2-1-2} states that for every $\R$-accessible $w\not\sqsubset\wfix$ and $B\in\lcalb$, we have $\lgci 1 w\vdash \gamma(B)\lr B$. The next lemma extends a weaker form of this result to $w\sqsubset \wfix$.
\begin{lemma}\label{4-2}
	For every $\R$-accessible $w\in W$ and $\Box B\in\lbqw$, we have $\lgci 1 w\vdash B $ iff $\lgci 1 w\vdash \gamma(B)$. Moreover, for every $B\in \lbqw\cap\lcalz(\parb)$, we have $\lgci 1 w\vdash B\lr \gamma(B)$.
\end{lemma}
\begin{proof}
	The second statement follows from the first and \Cref{4-2-2}. We prove the first by induction on $w$ ordered by $\sqsupset$. Assume inductively that for every $\Box B\in \lbqw$ and $u\sqsupset w$, we have $\lgci 1 u\vdash B$ iff $\lgci 1 u\vdash \gamma(B)$. The induction hypothesis together with \Cref{4-2-2} implies $\lgci 1 w\vdash E\lr\gamma(E)$ for every $E\in\lbqw\cap\lcalz(\parb)$.
	
	If $w\not\sqsubset \wfix$, then \Cref{2-1-2} yields the result. Since $\Box B\in\lbqw$, either $B\in\lcalz(\parb)$ or $q \notin\subo B$. If $B\in\lcalz(\parb)$, the induction hypothesis gives $\lgci 1 w\vdash B\lr\gamma(B)$, which suffices. Thus assume $w\R \wfix$ and $q \notin\subo B$.
	\begin{itemize}[leftmargin=*]
		\item Assume $\lgci 1 w\vdash B$. By \Cref{4-1}, there exists $E\in \lgci 1 w\cap\lbqw$ with $\vdash E\to B$. Since $\lgci 1 w\subseteq\snnilb\subseteq\lcalz(\parb)$, the induction hypothesis yields $\lgci 1 w\vdash \gamma(E)\lr E$. From $\vdash E\to B$, we get $\vdash \gamma(E)\to\gamma(B)$, hence $\vdash E\to \gamma(B)$. Because $E\in \lgci 1 w$, we conclude $\lgci 1 w\vdash\gamma(B)$.
		\item Assume $\lgci 1 w\vdash\gamma(B)$. The induction hypothesis gives $\lgci 1 w\vdash \gamma(B)\lr \ov{\gamma}(B)$, so $\lgci 1 w\vdash\ov{\gamma}(B)$. Define $\lambda$ by $\lambda(q ):=\bot$ and identity elsewhere. Since $\iglcp$ is closed under outer substitutions, $\ov{\lambda} (\lgci 1 w)\vdash \ov{\lambda} (\ov\gamma(B))$. Observe that by definition of $\gamma$, $\vdash \ov{\lambda} (\ov\gamma(B))\lr \ov{\lambda} (B)$, and because $q \notin\subo B$, $\vdash \ov{\lambda} (\ov\gamma(B))\lr B$. Hence $\ov{\lambda}(\lgci 1 w)\vdash B$. For every $E\in\snnilb$, we have $\ov{\lambda} (E)\in\snnilb$. Since $\pcal^1,w\pmodels \neg q$ and $\neg q \vdash E\lr \ov{\lambda}(E)$, we get $\lgci 1 w\vdash \ov{\lambda}(\lgci 1 w)$ and thus $\lgci 1 w\vdash B$. \qedhere
	\end{itemize}
\end{proof}

\begin{lemma}\label{local-lem-4-0}
	For every $D\in\lcalz(\parb)\cap\lbqw$ and $w\in W$, $\pcal^1,w\pmodels  D$ iff $\pcal^2,w\pmodels  \gamma(D)$.
\end{lemma}
\begin{proof}
	Since $\pcal^1\pmodels \psi$, we have $\pcal^1\pmodels \eta(D)\lr D$. By \Cref{local-lem-1.5}, $\pcal^2\pmodels \gamma(\eta(D))\lr \gamma(D)$. Because $p_i\notin\sub{\eta(D)}$ for every $1\leq i\leq m$ and $D\in\lbqw$ implies $\eta(D)\in\lbqw$, it suffices to prove the statement with the additional assumption $D\in\lcalb'$.
	
	We proceed by a double induction: first on $w\in W'$ ordered by $\sqsupset$, second on the complexity of $D\in \lcalz(\parb)\cap\lcalb'\cap\lbqw$. As the first induction hypothesis, assume that for every $C\in \lcalz(\parb)\cap\lcalb'\cap\lbqw$ and every $u\sqsupset w$, we have $\pcal^1,u\pmodels  C$ iff $\pcal^2,u\pmodels \gamma(C)$. As the second induction hypothesis, assume that for every strict subformula $C$ of $D$ and every $w'\in W$, we have $\pcal^1,w'\pmodels  C$ iff $\pcal^2,w'\pmodels  \gamma(C)$. We consider cases for $D$:
	\begin{itemize}[leftmargin=*]
		\item $D=\bot$ or $D$ is a parameter (other than $p_i$): Then $\gamma(D)=D$ by definition. The definitions of $\pcal^1$ and $\pcal^2$ yield the desired equivalence.
		\item $D$ is a conjunction, disjunction, or implication: Straightforward from the second induction hypothesis.
		\item $D=\Box C$:
		\begin{itemize}[leftmargin=*]
			\item Suppose $\pcal^1,w\pmodels  \Box C$ and let $u\sqsupset w$. We must show $\lgci 2 u\vdash \gamma(C)$. From $\pcal^1,w\pmodels  \Box C$, we have $\lgci 1 u\vdash C$. We consider three subcases:
			\begin{enumerate}
				\item $u\not\sqsubseteq \wfix$: Then \Cref{4-2} gives $\lgci 1 u\vdash \gamma(C)$, and \Cref{local-lem-1.49} implies $\lgci 2 u\vdash\gamma(C)$.
				\item $u=\wfix$: Then $\snnilb_u\vdash \varphi^0_u\to C$. Since $q \notin\sub{\varphi^0_u}$, we have $\varphi^0_u\in\lbqw$. Hence $\varphi^0_u\to C\in\lbqw$, and \Cref{4-1} yields $E\in\snnilb_u\cap\lbqw$ with $\vdash E\to (\varphi^0_u\to C)$. Then $\gamma(E),\gamma(\varphi^0_u)\vdash \gamma(C)$. Because $E\in\snnilb\cap\lbqw$, we have $\pcal^1,u\pmodels E$, so the first induction hypothesis gives $\pcal^2,u\pmodels \gamma(E)$, whence $\gamma(E)\in \lgci 2 u$. By \Cref{local-lem-2}, $\lgci 2 u\vdash \gamma(\varphi^0_u)$. Thus $\lgci 2 u\vdash \gamma(C)$.
				\item $u\R \wfix$: Then $\lgci 1 u\subseteq\snnilb$ and $\lgci 1 u\vdash C$. By \Cref{4-1}, there exists $E\in \lgci 1 u\cap\lbqw$ with $\vdash E\to C$. Hence $\vdash \gamma(E)\to \gamma(C)$. Since $E\in \lgci 1 u$, we have $\pcal^1,u\pmodels E$; the first induction hypothesis yields $\pcal^2,u\pmodels \gamma(E)$. Because $\gamma(E)\in\snnilb$, we get $\gamma(E)\in \lgci 2 u$, so $\lgci 2 u\vdash \gamma(C)$.
			\end{enumerate}
			\item Conversely, suppose $\pcal^2,w\pmodels \gamma(\Box C)$ and let $u\sqsupset w$. We must show $\lgci 1 u\vdash C$. From $\pcal^2,w\pmodels \Box\gamma(C)$, we have $\lgci 2 u\vdash \gamma(C)$. Consider two subcases:
			\begin{enumerate}
				\item $u\not\sqsubset \wfix$: Then $\lgci 2 u,\varphi^0_u\vdash \gamma(C)$. By \Cref{local-lem-1.49}, $\lgci 1 u\vdash \gamma(C)$. Then \Cref{4-2} gives $\lgci 1 u\vdash C$.
				\item $u\R \wfix$: From $\lgci 2 u\vdash \gamma(C)$ and \Cref{4-1}, there exists $E\in \snnilb\cap\lbqw$ such that $\gamma(E)\in \lgci 2 u$ and $\vdash \gamma(E\to C)$. Hence $\pcal^2,u\pmodels \gamma(E)$, and the first induction hypothesis yields $\pcal^1,u\pmodels E$. Since $E\in\snnilb$, we have $E\in \lgci 1 u$. Moreover, $\vdash \gamma(E\to C)$ implies $\lgci 1 u\vdash \gamma(E\to C)$, so $\lgci 1 u\vdash \gamma(E)\to\gamma(C)$. Because $E\in\lcalz(\parb)\cap\lbqw$, \Cref{4-2} gives $\lgci 1 u\vdash E\lr\gamma(E)$. Hence $\lgci 1 u\vdash\gamma(C)$, and \Cref{4-2} yields $\lgci 1 u\vdash C$. \qedhere
			\end{enumerate}
		\end{itemize}
	\end{itemize}
\end{proof}

\begin{corollary}\label{local-lem-4}
	For every $D\in\lbqw$ and $w\sce w_0$, $\pcal^1,w\pmodels  D$ iff $\pcal^2,w\pmodels  \gamma(D)$.
\end{corollary}
\begin{proof}
	Induction on the complexity of $D$:
	\begin{itemize}[leftmargin=*]
		\item $D\in\parb$: \Cref{local-lem-4-0}.
		\item $D$ is an atomic variable: Since $w\sce w_0$, we have $w\neq \wfix$, so $\pcal^2,w\pmodels\neg q$. By \Cref{handy}, $\pcal^2,w\pmodels \gamma(x)\lr x$. Because $\pcal^2,w\pmodels x$ iff $\pcal^1,w\pmodels x$, the result follows.
		\item $D$ is a conjunction, disjunction, or implication: Follows from the induction hypothesis. \qedhere
	\end{itemize}
\end{proof}

\subsection{Second reduction step: $\iglcphb\nvdash A$ implies $\lles\nvdash\beta(A)$}\label{second-red}
This section shows that if $\iglcphb\nvdash A$, then $\lles\nvdash\beta(A)$ for some substitution $\beta$.

\begin{theorem}\label{Reduction-lleb-lles}
	If $\iglcphb\nvdash A$, then there exists a substitution $\beta$ such that $\lles\nvdash \beta(A)$.
\end{theorem}
\begin{proof}
	Assume $\iglcphb\nvdash A$. By \Cref{provability-models-gen-completeness2}, there exists a good $\snnilb$-based provability model $\pcal_0=\cpm{\kcal_0}{\snnilb}$ with $\kcal_0=(W,\pce,\R,\V_0)$ and $w_0\in W$ such that $\pcal_0,w_0\npmodels A$. Moreover, we may assume that for every atomic $a\notin\sub A$ and every $w\in W$, we have $\pcal_0,w\npmodels a$. For each $w\in W$, let $p_w$ be a fresh parameter (not appearing in $A$). Define the substitution $\beta$ on atomic variables $x$ by (and $\beta(p)=p$ for parameters):
	\emli
	$$
	\beta(x):=\bigvee_{ \kcal_0,w\Vdash x}\varphi_w
	\quad \text{where}\quad 
	\varphi_w:=p_w\wedge \bigwedge_{w\R u}\neg p_u.
	$$
	Conventions: disjunction over an empty set is $\bot$, conjunction over an empty set is $\top$.
	
	Define $\kcal_1:=(W,\pce,\R,\V_1)$ and $\pcal_1$ as follows:
	\emli
	\begin{align*}
		w\V_1 a &\quad\text{iff}\quad \exists u\Rvar w\ (a=p_u) \text{ or } (a\in\parr\ \&\ w\V_0 a);\\
		X &:=\snnilb\cup\cpv \quad \text{ with } \quad \cpv:=\{ x\to\Box x : x\in\varr\};\\
		\pcal_1 &:=\cpm{\kcal_1}X.
	\end{align*}
	Recall that $u\Rvar w$ iff $w$ is in the model generated by $u$ (see \Cref{sec-Kripke}). By construction, $\pcal_1$ is atomic ascending.
	
	Then \Cref{bet-lem} implies $\pcal_1,w_0\npmodels \beta(A)$. Hence \Cref{sound-const2} yields $\lles\nvdash \beta(A)$.
\end{proof}

Let $\lcalbp$ denote the set of formulas in $\lcalb$ that do not contain parameters $\{p_w: w\in W\}$.
\begin{lemma}\label{bet-lem}
	For every $A\in\lcalbp$ and $w\in W$, $\pcal_0,w\pmodels A$ iff $\pcal_1,w\pmodels \beta(A)$.
\end{lemma}
\begin{proof}
	We introduce some notation and observations. Let $\lgci i w$ denote the theory assigned to world $w$ in $\pcal_i$. Up to $\IPC$-provable equivalence:
	\emli
	\begin{align*}
		\lgci 0 w &:=\snnilb_w:=\{E\in\snnilb: \pcal_0,w\pmodelss E\},\\
		\lgci 1 w &:= X_w:=\{E\in X: \pcal_1,w\pmodelss E\}.
	\end{align*}
	We write $\ape{ \ }$ for $\ap\snnilb\iglcp{\ }$. By \Cref{APlus1,SNB-approx-prop}, $\ape{ A }=\sggt E$ where $B$ contains no atomic or boxed formulas beyond those appearing in $A$. Also, because $\beta(x)$ is a Boolean combination of parameters (hence contains no boxed subformulas), for any $B\in\lcalbp$ we have $\ape{\beta(B)}=\beta(B')$ for some $B'\in\snnilb\cap\lcalbp$ with $\vdash B'\to B$.
	
	We proceed by double induction: first on $w\in W$ ordered by $\sqsupset$, second on the complexity of $A\in\lcalb$. Assume inductively that for every $u\sqsupset w$ and every $B\in\lcalb$ without parameters $p_v$, we have $\pcal_0,u\pmodelss B$ iff $\pcal_1,u\pmodelss \beta(B)$. Also assume that for every strict subformula $B$ of $A$ and every $u\sce w$, we have $\pcal_0,u\pmodels B$ iff $\pcal_1,u\pmodels \beta(B)$. Cases for $A$:
	\begin{itemize}
		\item $A$ is a parameter: Then $A\neq p_w$ for all $w$. By definition of $V_1$, $w\V_0 A$ iff $w\V_1 A$. Hence $\pcal_0,w\pmodels A$ iff $\pcal_1,w\pmodels A$. Since $\beta(A)=A$, the result follows.
		\item $A$ is a variable: Direct from \Cref{bet-lem0}.
		\item $A$ is a conjunction, disjunction, or implication: Follows from the second induction hypothesis.
		\item $A=\Box B$: 
		\begin{itemize}
			\item Suppose $\pcal_0,w\pmodels \Box B$ and let $u\sqsupset w$. We must show $\lgci 1 u\vdash \beta(B)$. From $\pcal_0,w\pmodels \Box B$, we have $\lgci 0 u\vdash B$. Hence there exists $E\in \snnilb$ with $\pcal_0,u\pmodelss E$ and $\vdash E\to B$. Consequently, $\vdash E\to \ape B$, so $\pcal_0,u\pmodelss \ape B$ and $\vdash \ape B\to B$. Therefore, $\vdash \beta(\ape B) \to \beta(B)$. The first induction hypothesis gives $\pcal_1,u\pmodelss \beta(\ape B)$. Thus $\lgci 1 u\vdash \beta(\ape B)$. Since $\vdash \beta(\ape B) \to \beta(B)$, we obtain $\lgci 1 u\vdash \beta(B)$.
			\item Conversely, suppose $\pcal_1,w\pmodels \Box \beta(B)$ and let $u\sqsupset w$. Then $\bigwedge_{i=1}^n (x_i\to\Box x_i)\vdash E\to \beta(B)$ for some $E\in \snnilb$ and atomic variables $x_i$ such that $\pcal_1,u\pmodelss E$. Let $\tau$ substitute $\bot$ for each $x_i$ and be identity elsewhere. Then $\ov\tau(\bigwedge_{i=1}^n x_i\to\Box x_i)\vdash \ov\tau(E\to \beta(B))$. Since $\ov\tau(E)=E$ and $\ov\tau\beta(B)=\beta(B)$, we get $\vdash E\to\beta(B)$. Hence $\vdash E\to \ape{\beta(B)}$ and $\pcal_1,w\pmodelss\ape{\beta(B)}$. As noted earlier, $\ape{\beta(B)}=\beta(B')$ for some $B'\in\snnilb\cap\lcalbp$ with $\vdash B'\to B$. The first induction hypothesis yields $\pcal_0,u\pmodelss B'$. Thus $\lgci 0 u\vdash B$. \qedhere
		\end{itemize}
	\end{itemize}
\end{proof}

\begin{lemma}\label{bet-lem00}
	For every $w\in W$, $\kcal_1,w\Vdash  \varphi_u$ iff $u\pce w$.
\end{lemma}
\begin{proof}
	First assume $u\pce w$. Then $u\Rvar w$, so by definition $\kcal_1\Vdash p_u$. Now let $v\sqsupset u$ and $w'\sce w$. If $v\Rvar w'$, then both $u\Rvarr w'$ and $u\pce w'$, which by transcendence (see \cref{sec-Kripke}) imply $w'=w=u$. Hence $u\Rvarr u$, meaning there exists $u'$ with $u\R u'\pce u$. Then $u'\R u'$, contradicting converse well-foundedness. This contradiction shows $v\not\Rvar w'$ for every $w'\sce w$. Thus for every $v\sqsupset u$, $\kcal_1,w\Vdash \neg p_v$. Consequently, $\kcal_1,w\Vdash \varphi_u$.
	
	Conversely, assume $\kcal_1,w\Vdash \varphi_u$. 
	Then $w\V_1 p_u$ and not $w\V_1 p_v$ for every $v\sqsupset u$. 
	By definition of $\V_1$, $u\Rvar w$ and $v\not\Rvar w$ for every $v\sqsupset u$. 
	From $u\Rvar w$, there exists $v$ such that $u\sqsubseteq v\pce w$. 
	If $v\neq u$, then $u\R v$. Then $v\not\Rvar w$, contradicting $v\pce w$. 
	Hence $v=u$, so $u\pce w$.
\end{proof}

\begin{lemma}\label{bet-lem0}
	For every $x\in\varr$ and $w\in W$, $\kcal_0,w\Vdash x$ iff $\kcal_1,w\Vdash \beta(x)$.
\end{lemma}
\begin{proof}
	If $\kcal_0,w\Vdash x$, then $\varphi_w$ is a disjunct of $\beta(x)$. By \Cref{bet-lem00}, $\kcal_1,w\Vdash \beta(x)$.
	
	If $\kcal_1,w\Vdash \beta(x)$, then $\kcal_1,w\Vdash \varphi_u$ for some $u$ with $\kcal_0,u\Vdash x$. By \Cref{bet-lem00}, $u\pce w$. Since $\kcal_0,u\Vdash x$ and $u\pce w$, we have $\kcal_0,w\Vdash x$.
\end{proof}

For subsequent lemmas, let $\pcal_i=(W,\pce,\R,\{\lgci i w\}_{w\in W},\V_i)$ for $i=0,1,2$. Recall that modulo $\IPC$-provable equivalence:
\emli
\begin{align*}
	\lgci 0 w &= \snnilb^0_w:=\{E\in\snnilb: \pcal_0,w\pmodelss E\},\\
	\lgci 2 w &= X_w:=\{E\in X: \pcal_2,w\pmodelss E\},\\
	\lgci 1 w &= \snnil^1_w:=\{E\in\snnil: \pcal_1,w\pmodelss E\}.
\end{align*}
By \Cref{sound-const,sound-const1}, for every $\R$-accessible $w$ and $i=0,1,2$, $\lgci i w$ includes $\iglcp$ and is closed under necessitation. These facts will be used implicitly. Also, by definition, $\pcal_0\pmodels \Boxdot\neg p_u$ for every $u\in W$; i.e., $\lgci 0 w \vdash \neg p_u$ for every $\R$-accessible $w$. Similarly, $\pcal_1\pmodels \Boxdot(\neg x)$ for every atomic variable $x$, so $\lgci 1 w\vdash \neg x$ for every $x\in\varr$ and $\R$-accessible $w$.

\section*{Acknowledgements}
We warmly acknowledge Amirhossein Akbar Tabatabai, Mohammad Ardeshir, Lev Beklemishev, Rosalie Iemhoff, Fedor Pakhomov, Borja Sierra Miranda, and Albert Visser for stimulating discussions and correspondence on provability logic. We are especially grateful to Albert Visser, Rosalie Iemhoff, and Lev Beklemishev for their invaluable support and guidance; their comments, suggestions, and remarks were crucial to this work. In particular, we benefited greatly from fruitful discussions with Borja Sierra Miranda on the classical case of provability models.

\section*{Funding}
This work was partially funded by FWO grant G0F8421N and BOF grant BOF.STG.2022.0042.01.

\mojref

\newpage
\begin{appendices}

\section{Table of Symbols: Formulas and Translations}
\renewcommand{\figurename}{Table}
\begin{table}[H]
\bgroup
\def\arraystretch{1.6}
\begin{center}
\begin{tabular}{|c|c|c|}
    \hline
    \multicolumn{3}{|c|}{$X$, $Y$, and $\Gamma$ denote sets of propositions (with optional subscripts)} 
    \\ \hline\hline
    \textbf{Symbol} 
    & \textbf{Definition}
    & \textbf{Section} 
    \\ \hline
    $X_1X_2\ldots X_n$ & $X_1\cap X_2\cap\ldots\cap X_n$ & \ref{notation-set}
    \\ \hline
    $\varr$ & Countably infinite set of atomic variables & \ref{sec-languages}
    \\ \hline
    $\parr$ & Countably infinite set of atomic parameters & \ref{sec-languages}
    \\ \hline
    $\lcalz$ & Boolean combinations of $\parr$, $\varr$, and $\bot$ & \ref{sec-languages}
    \\ \hline 
    $\lcalb$ & The language $\lcalz$ extended with the unary modal operator $\Box$ & \ref{sec-languages}
    \\ \hline
    $\lcalp$ & The language $\lcalz$ extended with the binary modal operator $\rhd$ & \ref{sec-languages}
    \\ \hline
    $\boxed$ & The set of all boxed formulas in $\lcalb$ & \ref{sec-languages}
    \\ \hline
    $\parb$ & $\parr\cup\boxed\cup\{\bot\}$ & \ref{sec-languages}
    \\ \hline
    $\atomb$ & $\parb\cup\varr$ & \ref{sec-languages}
    \\ \hline
    $\NNIL$ or $\nnil$ &  
    \def\arraystretch{1}
    \begin{tabular}{c}
    All propositions with no nested implications on the left, \\
    except those nested within $\Box$ operators
    \end{tabular}
    & \ref{sec-NNIL-def}
    \\ \hline
    $\nnilb$ or $\NNILb$ & $\NNIL$-propositions that are Boolean combinations of $\parb$ & \ref{notation-set}
    \\ \hline
    $\dar\sft\Gamma$ or $\darl\Gamma$ &  
    \def\arraystretch{1}
    \begin{tabular}{c}
    Operator with arguments $\Gamma$ and $\sft$: \\
    The set of all $\Gamma$-projective propositions within the logic $\sft$
    \end{tabular}
    & \ref{Gam-projec-modal}
    \\ \hline
    $\sub A$ & All subformulas of $A$, including $A$ & \ref{notation-set}
    \\ \hline
    $\subo A$ & All subformulas of $A$ appearing outside $\Box$ operators & \ref{notation-set}
    \\ \hline
    $\subob A$ & $\subo A\cap \boxed$ & \ref{notation-set}
    \\ \hline 
    $\Gamma(X)$ & $\Gamma\cap\lcalz(X)$ & \ref{notation-set}
    \\ \hline
    $\sfs$ & $\{\sggt B: B\in\lcalb\}$, where $\sggt{(.)}$ denotes G\"odel's translation & \ref{notation-set}
    \\ \hline
    $\tcom$ or $\sfc$ & $\{A\in\lcalb: \tvdash A\to\Box A\}$ & \ref{notation-set}
    \\ \hline
    $\prim\sft$ or $\sfP$ &  
    \def\arraystretch{1}
    \begin{tabular}{c}
    Set of all propositions $A$ such that $\sft+A$ \\
    has the disjunction property.
    \end{tabular}
    & \ref{notation-set}
    \\ \hline
    $\nfz\Gamma$ & All propositions $A$ such that either $A\in\Gamma$ or $\boxdot A\in\Gamma$ & \ref{notation-set}
    \\ \hline
    $\nf \Gamma$ & $\{A\in\lcalb: \forall\, \Box B\in \sub A (B\in\nfz\Gamma)\}$ & \ref{notation-set}
    \\ \hline
    $ \Gamma^\vee$ & Disjunctive closure of $\Gamma$ (excluding the empty disjunction) & \ref{notation-set}
    \\ \hline
    \multicolumn{3}{|c|}{$\dar{\,}(.)$ has the lowest precedence after $(.)^\vee$: $\dsnbv:=(\dar{\,}(\snnilb))^\vee$ and $\cdsnbv:=(\sfc(\dar{\,}(\snnilb)))^\vee.$} 
    \\ \hline
\end{tabular}
\caption[Symbols: sets of propositions]{Sets of propositions\label{Table-sets}}
\end{center}
\egroup
\end{table}

\newpage

\begin{table}[H]
\bgroup
\def\arraystretch{1.5}
\begin{center}
\begin{tabular}{|c|c|c|}
\hline
\multicolumn{3}{|c|}{$\sft$ denotes a logic; ${\sf X}$ with possible subscripts denotes an axiom scheme} 
\\ \hline
\multicolumn{3}{|c|}{$\sft{\sf X}_1{\sf X}_2\ldots{\sf X}_n$ denotes $\sft$ extended with axiom schemas ${\sf X}_1, \ldots, {\sf X}_n$} 
\\ \hline
\multicolumn{3}{|c|}{The \textit{only} inference rule in all logics is \textit{modus ponens}. \textit{Necessitation} is admissible.} 
\\ \hline
\multicolumn{3}{|c|}{ 
\def\arraystretch{1}
\begin{tabular}{c}
\textbf{Important convention:} All axiom schemas are considered with their $\Box$ prefixes, \\
but for notational simplicity the $\Box$ symbols are omitted.
\end{tabular}
} 
\\ \hline    
\hline
\textbf{Symbol} 
& \textbf{\hspace{4.5cm}Definition\hspace{4.5cm}}
& \textbf{Section} 
\\ \hline 
${\sf K}$ & $\Box (A\to B)\to (\Box A\to \Box B)$ & \ref{propo-logics}
\\ \hline 
{\sf 4} & $\Box A\to\Box\Box A$ & \ref{propo-logics}
\\ \hline
{\sf L} & The {\sf L\"ob} axiom: $\Box(\Box A\to A)\to\Box A$ & \ref{propo-logics}
\\ \hline
${\cpp}$ (${\cpa}$) & $a\to\Box a$ for every $a\in\parr$ ($a\in\atom$) & \ref{propo-logics}
\\ \hline
${\visgt}$ & $\Box A\to \Box B$ for every $A\prtg B$ & \ref{propo-logics}
\\ \hline
$\Ha$ & $\visp{\cdsnb}{\iglcp}$ & \ref{propo-logics}
\\ \hline
$\Hs$ & $\vis{\snnil}{\iglca}$ & \ref{propo-logics}
\\ \hline
$\Hb$ & $\vis{\snnilb}{\iglcp}$ & \ref{propo-logics}
\\ \hline
$\IPC$ & Intuitionistic propositional logic & \ref{propo-logics}
\\ \hline  
{\sf i} & $\IPC$ plus $\cpp$ & \ref{propo-logics}
\\ \hline 
$\ikfourcp$ & ${\sf i}+ {\sf K}+ {\sf 4}$ & \ref{propo-logics}
\\ \hline 
$\iglcp$ & $\ikfourcp{\sf L}$ & \ref{propo-logics}
\\ \hline 
\end{tabular}
\caption[Symbols: axioms, rules and logics for $\Box$]{Axioms, rules and logics for $\Box$\label{Table-logics}}
\end{center}
\egroup
\end{table}

\newpage 

\begin{table}[H]
\bgroup
\def\arraystretch{1.5}
\begin{center}
\begin{tabular}{|c|c|c|}
\hline
\textbf{Symbol} 
& \textbf{Definition}
& \textbf{Section} 
\\ \hline \hline 
Conj & $A\rhd B, A\rhd C / A\rhd(B\wedge C)$ & \ref{pres-admis}
\\ \hline  
Cut & $A\rhd B, B\rhd C / A\rhd C$ & \ref{pres-admis}
\\ \hline 
$\BART$ & 
\def\arraystretch{1}
\begin{tabular}{c}
A system proving statements $A\rhd B$ for $A,B$ in the language of $\sft$, \\
including $\{A\rhd B : \tvdash A\to B\}$ and closed under Conj and Cut rules
\end{tabular}
& \ref{pres-admis}
\\ \hline 
$\Les$ ($\Lem$) & $A\rhd \Box A$ for every $A\in\lcalb$ ($A\in\lcalz(\parb)$) & \ref{pres-admis}
\\ \hline 
$\VAB$ & $A\rhd \th (A)$ for every $A\in\lcalb$ and substitution $\theta$ & \ref{pres-admis}
\\ \hline 
$\viss\Delta$ &
\def\arraystretch{1}
\begin{tabular}{c}
$B\to C \rhd \bigvee_{i=1}^{m+n} \itv B {E_i}\Delta$ for $B=\bigwedge_{i=1}^n (E_i\to F_i)$ and $C=\bigvee_{i=n+1}^{n+m} E_i$, \\
where $\itv E F \Delta$ is defined as $F$ if $F\in\Delta$, and $E\to F$ otherwise
\end{tabular}
& \ref{pres-admis}
\\ \hline 
Disj & $B\rhd A, C\rhd A / (B\vee C)\rhd A$ & \ref{pres-admis}
\\ \hline  
$\montd$ & $A\rhd B / (E\to A)\rhd(E\to B)$ for every $E\in\Delta$ & \ref{pres-admis}
\\ \hline  
$\ARD\sft\Delta$ & $\BART +\text{Disj} + \montd + \viss\Delta$ & \ref{pres-admis}
\\ \hline  
$\amontd$ & $(A\rhd B)\to ((C\to A)\rhd (C\to B))$ & \ref{sec-preser}
\\ \hline  
ADisj & $(B\rhd A \wedge C\rhd A)\to ((B\vee C)\rhd A)$ & \ref{sec-preser}
\\ \hline  
AConj & $((A\rhd B)\wedge(A\rhd C))\to(A\rhd (B\wedge C))$ & \ref{sec-preser}
\\ \hline  
ACut & $[(A\rhd B)\wedge(B\rhd C)]\to(A\rhd C)$ & \ref{sec-preser}
\\ \hline  
$\LARD\sft\Delta$ &  
\def\arraystretch{1}
\begin{tabular}{c}
$\sft+\viss\Delta+\amontd+\Le+\text{ADisj}+\text{AConj}+\text{ACut}$ \\
together with the inference rule $A\to B / A\rhd B$
\end{tabular}
& \ref{sec-preser}
\\ \hline  
${\sf 4p}$ & $(B\rhd C)\to\Box (B\rhd C)$ & \ref{sec-preser}
\\ \hline  
$\LARDP\sft\Delta$ & $\LARD\sft\Delta+{\sf 4p}$ & \ref{sec-preser}
\\ \hline  
$\iph$ ($\iphp$) & $\LARD\iglcp\parb$ ($\LARDP\iglcp\parb$) & \ref{sec-preser}
\\ \hline  
$\iphs$ & $\LARD\iglca\parb$ & \ref{sec-preser}
\\ \hline  
\end{tabular}
\caption[Symbols: axioms, rules and logics for $\rhd$]{Axioms, rules and logics for $\rhd$\label{Table-logics-rhd}}
\end{center}
\egroup
\end{table}

\newpage 

\begin{table}[H]
\bgroup
\def\arraystretch{1.5}
\begin{center}
\begin{tabular}{|c|c|c|}
\hline
\textbf{Name} 
& \textbf{Definition}
& \textbf{Section} 
\\ \hline \hline 
$\ov{\theta}$ & 
\def\arraystretch{1}
\begin{tabular}{c}
Function that commutes with Boolean connectives, \\
is the identity on boxed formulas, \\
and satisfies $\th=\theta$ on atomic variables
\end{tabular}
& \ref{sec-sub} 
\\ \hline  
$A\prtg B$ & $\forall\,E\in\Gamma(\tvdash E\to A \Rightarrow \tvdash E\to B)$ & \ref{pres-admis}
\\ \hline 
$A\argt B$ & $\forall\,\theta\ \forall\,E\in\Gamma(\tvdash \th(E\to A) \Rightarrow \tvdash \th(E\to B))$ & \ref{pres-admis}
\\ \hline   
$\ap\Gamma\sft A$ & Greatest lower bound for $A$ in $\Gamma$ within $\sft$ & \ref{glb}
\\ \hline 
$\aha$ & 
\def\arraystretch{1}
\begin{tabular}{c}
Given a function $\alpha$ from $\atom$ to first-order sentences, 
\\ $\aha$ extends $\alpha$ 
to commute with Boolean connectives, \\ interpret $\Box$ as 
provability in $\HA$, 
and interpret $\rhd$ as \\
$\Sigma_1$-preservativity in $\HA$
\end{tabular}
& \ref{sec-sub-1} 
\\ \hline  
$\sggt{A}$ & 
\def\arraystretch{1}
\begin{tabular}{c}
G\"odel's translation: \\
places $\boxdot$ before (almost) every subformula 
\end{tabular}
& \ref{sec-Box-trans}
\\ \hline 
$A\xrat{\theta}{\sft} \Gamma$ & $\theta$ is $A$-projective in $\sft$ and satisfies $\sft\vdash \th(A)\in\Gamma$ & \ref{Gam-projec-modal}
\\ \hline 
$\cto(A)$ & Number of nested $\to$ operators not in the scope of $\Box$ & \ref{sec-cto}
\\ \hline
$\ctob(A)$ & Maximum of $\cto(B)$ over all $\Box B\in\sub{A}$ & \ref{sec-cto}
\\ \hline  
\def\arraystretch{1}
\begin{tabular}{c}
Transcendental \\
Kripke model
\end{tabular}
 & $u\pce w$ and $v\R w$ implies $u=v$ & \ref{sec-Kripke}
\\ \hline
\def\arraystretch{1}
\begin{tabular}{c}
Good\\
Kripke model
\end{tabular}
& 
\def\arraystretch{1}
\begin{tabular}{c}
 Finite, rooted, transitive, 
\\ conversely well-founded tree forcing $\cpp$
\end{tabular} 
& \ref{sec-Kripke}
\\ \hline
\end{tabular}
\caption[Other symbols]{Other symbols\label{Other-symb}}
\end{center}
\egroup
\end{table}

\end{appendices}

\end{document}

%% file: pre.tex
\documentclass[10pt]{article}
\usepackage[pagewise]{lineno}
\def\emli{}
\usepackage{catchfilebetweentags,float}
\usepackage[toc,page]{appendix}
\usepackage{amsmath,amssymb,amsfonts,graphicx,bm,amsthm,geometry}
\usepackage{mathtools,stackengine}
\usepackage{mathrsfs}
\usepackage{enumitem}
\usepackage{latexsym}
\usepackage{rotating}
\usepackage{xspace}
\usepackage[usenames]{xcolor}
\usepackage{stmaryrd}
\usepackage[colorlinks=true]{hyperref}
\definecolor{ultramarine}{rgb}{0 , 0, 0.3}
\definecolor{bulgarianrose}{rgb}{0.3, 0, 0}
\hypersetup{citecolor={ultramarine},
urlcolor={ultramarine},linkcolor={bulgarianrose}}
\usepackage[capitalise]{cleveref}

\usepackage{soul}
\usepackage{bussproofs}
\usepackage{pstricks}
\usepackage[square]{natbib}
\usepackage{aliascnt}

\usepackage[a]{esvect}

\numberwithin{equation}{section}
\newcommand{\myref}{/home/mojtaba/texmf/tex/latex/commonstuff/bibliography.bib}
\newcommand{\mojref}{\bibliography{\myref}%
\bibliographystyle{apalike}%
}

\newcommand{\Boxdot}{\,
                    \setlength{\unitlength}{1ex}
                    \begin{picture}(2,2)
                    \put(0,0){$\Box$}
                    \put(.50,.65){.}
                    \end{picture}
                    }

\newcommand{\Boxh}{{
                    \setlength{\unitlength}{1ex}
                    \begin{picture}(1.7,1.7)
                    \put(0,0){$\Box$}
                    \put(-.2,.4){{\fontsize{5}{0} \selectfont \sfh}}
                    \end{picture}
                    \!}}

\def\tcal{\mathcal{T}}

\def\lcal{\mathcal{L}}

\def\lcalb{\lcal_{\Box}}

\def\lcalp{\lcal_{\rhd}}
\def\lcalz{\lcal_{0}}

\def\kcal{\mathcal{K}}

\def\kkcal{{\mathcal{K}}}

\def\kripke{\kcal=(W,\pce,\R,\V)}
\def\kripkep{\kcal'=(W',\pce',\R',\V')}

\def\IPC{{\sf IPC}}
\def\ipc{{\sf IPC}}

\def\HA{\hbox{\sf HA}{} }

\def\PA{\hbox{\sf PA}{} }

\def\K4{\hbox{\sf K4}{} }
 
\def\GL{\hbox{\sf GL}{} }

\newcommand{\ttbrace}[1]{[\![{#1}]\!]}
\newcommand{\tcbrace}[1]{\{\!\!\{{#1}\}\!\!\}}
\newcommand{\tcbracep}[1]{\{\!\!\{{#1}\}\!\!\}^+}

\newcommand{\Ax}{\AxiomC}
\newcommand{\UI}{\UnaryInfC}
\newcommand{\BI}{\BinaryInfC}

\newcommand{\LLa}{\LeftLabel}
\newcommand{\RLa}{\RightLabel}
\newcommand{\DP}{\DisplayProof}

\newcommand{\uparan}[1]{\textup{(}#1\textup{)}}

\newcommand{\bs}[1]{\vv{#1}}

\def\abold{{\bs{a}}}

\def\pbold{{\bs{p}}}
\def\pvec{{\vv{p}}}

\def\qbold{{\bs{q}}}
\def\qvec{{\vv{q}}}

\def\IPC{{\sf IPC}}
\def\HA{\hbox{\sf HA}{} }

\def\PA{\hbox{\sf PA}{} }

\def\CPp{\hbox{\sf CP}_{\sf p}}

\def\NNIL{{\sf NNIL}}
\def\NNILP{{\sf NNIL}^{\!+}}
\def\SNNILP{\sfs\NNILP}
\def\nnil{{\sf N}}

\def\bsfs{{\sfs}}
\def\snnil{\bsfs\nnil}
\def\wsnnil{\sfs\nnil}
\def\snnilv{\snnil^{\!\vee}}
\def\wsnnilv{\wsnnil^{\!\vee}}
\def\SNNIL{\sfs\NNIL}

\def\K4{\hbox{\sf K4}{} }
\def\GL{\hbox{\sf GL}{} }

\def\iglca{{\sf iGLC_a}}
\def\iglcp{{\sf iGL}}

\def\R{\sqsubset}

\newcommand{\pres}[2]{\mathrel{\setlength{\unitlength}{1ex}
                    \begin{picture}(2,2)
                    \put(0,0){$\pr$}
                    \put(-.1,1.3){\fontsize{4}{0} \selectfont ${#1}$}
                    \put(-.1,-.8){\fontsize{4}{0} \selectfont ${#2}$}
                    \end{picture}}
                    }

\newcommand{\adsm}[2]{\mathrel{\setlength{\unitlength}{1ex}
                    \begin{picture}(2,2)
                    \put(0,0){$\ar$}
                    \put(-.3,1.2){\fontsize{4}{0} \selectfont ${#1}$}
                    \put(-.3,-.5){\fontsize{4}{0} \selectfont ${#2}$}
                    \end{picture}}
                    }

\newcommand{\vdashsub}[1]{\mathrel{\setlength{\unitlength}{1ex}
                    \begin{picture}(2,2)
                    \put(0,0){$\vdash$}
                    \put(0.1,1.2){\fontsize{4}{0} \selectfont 
                    $ $}
                    \put(0.1,-.4){\fontsize{4}{0} \selectfont ${#1}$}
                    \end{picture}}
                    }                                                  
\newcommand{\nvdashsub}[1]{\mathrel{\setlength{\unitlength}{1ex}
                    \begin{picture}(2,2)
                    \put(0,0){$\nvdash$}
                    \put(0.1,1.2){\fontsize{4}{0} \selectfont 
                    $ $}
                    \put(0.1,-.4){\fontsize{4}{0} \selectfont ${#1}$}
                    \end{picture}}
                    }                                                                      
\def\pr{\mid\!\approx}

\def\priglcdsnb{\pres{\iglcp}{\cdsnb}\ \ \ }
\def\priglcasn{\pres{\iglca}{\wsnnil}\quad}

\def\priglsnb{\pres{\iglcp}{\snnilb}\quad}

\def\prtsn{\pres{\sft}{\wsnnil}\ }
\def\prtsnb{\pres{\sft}{\snnilb}\ \ \ }
\def\prtsnv{\pres{\sft}{\wsnnilv}\ \, }

\def\prtspnb{\pres{\sft}{\spnnilb}\ \ }

\def\prtgp{\mathrel{ ^{^*}\!\!\!\prtg}}

\def\prtgv{\pres{\sft}{\Gamma^\vee}}

\def\prtnb{\prtsnb}
\def\prtspn{\pres{\sft}{\spnnil}\ \, }
\def\prtspnb{\pres{\sft}{\spnnilb}\ \ \ \ }
\def\prtspnbv{\pres{\sft}{\spnnilbv}\quad\ \  }
\def\prtspnv{\pres{\sft}{\spnnilv}\ \ \, }

\def\prtcdsnb{\pres{\sft}{\cdsnb}\ \ \,}
\def\prtdsnb{\pres{\sft}{\dsnb}\quad \ }

\def\prtdnb{\pres{\sft}{\dnb}\ \ \ }
\def\prtdpnb{\pres{\sft}{\dpnb}\ \ \ \ }
\def\prtcdspnb{\pres{\sft}{\cdspn}\ \ }
\def\prtdspnb{\pres{\sft}{\dspn}\quad}
\def\prtdspnbv{\pres{\sft}{\dspnv}\quad}
\def\prtcdspnbv{\pres{\sft}{\cdspnv}\quad\ }
\def\prtcdsnbv{\pres{\sft}{\cdsnbv}\quad }
\def\prtdsnbv{\pres{\sft}{\dsnbv}\quad\ \  }

\def\prtdnbv{\pres{\sft}{\dnbv}}
\def\prtdpnbv{\pres{\sft}{\dpnbv}\ \ \ \ \ \ }

\def\prin{\pres{\IPC}{\NNIL}\ \, }

\def\lles{\iglca\Hs}

\def\iglcphb{\iglcp\Hb}

\def\Hb{{{\sf H}^{\Box}}}

\def\Ha{{\sf H}}
\def\ulHa{{\underline{\sf H}}}

\newcommand{\vis}[2]{{\sf H}({#1},{#2})}
\newcommand{\visp}[2]{{\sf H}({#1},{#2})}
\newcommand{\ulvis}[2]{\underline{\vis{#1}{#2}}}
\newcommand{\ulvisp}[2]{\underline{\visp{#1}{#2}}}
\def\visgt{\vis{\Gamma}{\sft}}

\def\ulvisgt{\underline{\visgt}}

\def\ulHb{\underline{{\sf H}^{\Box}}}
\def\Hs{{{\sf H}_{\sigma}}}
\def\ulHs{\underline{{\sf H}_{\sigma}}}

\def\iglcph{\iglcp\Ha}

\def\iph{{\sf{iPH}}}
\def\iphs{\iph_\sigma}

\def\iphp{\iph^+}
\def\iphpb{\iph_\Box^+}
\def\iphb{\iph_\Box}

\def\pce{\preccurlyeq}
\def\sce{\succcurlyeq}

\newcommand{\V}{\mathrel{V}}

\newcommand{\ra }{\rightarrow }

\newcommand{\ikfourca}{{\sf iK4C_a}}
\newcommand{\ikfourcp}{{\sf iK4}}
\newcommand{\lr}{\leftrightarrow}

 \theoremstyle{theorem}
\newtheorem{theorem}{Theorem}[section]
\newtheorem{question}{Question}

\newaliascnt{lemma}{theorem}  
\newtheorem{lemma}[lemma]{Lemma}  
\aliascntresetthe{lemma}  
  
\newaliascnt{definition}{theorem}  
\theoremstyle{definition}
\newtheorem{definition}[definition]{Definition}  
\aliascntresetthe{definition}  
 
\newaliascnt{corollary}{theorem}  
\newtheorem{corollary}[corollary]{Corollary}  
\aliascntresetthe{corollary}  
 
\newaliascnt{proposition}{theorem}  
  
\aliascntresetthe{proposition}  
 
\newaliascnt{remark}{theorem}  
\newtheorem{remark}[remark]{Remark}  
\aliascntresetthe{remark}  
 
\newaliascnt{notation}{theorem}  
  
\aliascntresetthe{notation}  
 
\newaliascnt{example}{theorem}  
  
\aliascntresetthe{example}  
 
\newaliascnt{conjecture}{theorem}  
  
\aliascntresetthe{conjecture}  
 
\newaliascnt{fact}{theorem}  
  
\aliascntresetthe{fact}  
 
\newaliascnt{claim}{theorem}  
  
\aliascntresetthe{claim}  
 
\newcommand{\sub}[1]{{\sf sub}(#1)}

\def\lbqw{\lcalb''}

\def\scrk{\mathscr{K}}

\def\scrm{\mathscr{M}}
\def\scrmt{\mathscr{M}(\sft)}

\newcommand{\Mod}[1]{\textup{Mod}(#1)}

\def\spnnil{\text{{\sf SPN}}}
\def\spnnilv{\spnnil^{\!\vee}}

\def\NNILpar{{\NNIL(\parr)}}

\def\dnnil{{\dar{\,}\nnil}}

\def\nnilpq{{\nnil(\pbold,\qbold)}}

\def\nnilb{{\nnil(\Box)}}

\def\nnilv{{\nnil^{\!\vee}}}
\def\NNILb{{\NNIL(\Box)}}
\def\pnnilb{{{\sf PN}(\Box)}}

\def\pnnilv{{\pnnil^{\!\vee}}}
\def\spnnilb{{\spnnil(\Box)}}
\def\spnnilbv{{\spnnilb^{\!\vee}}}
\def\snnilb{{\wsnnil(\Box)}}

\def\snnilbv{{\snnilb^{\!\vee}}}

\def\spnv{\spnnilv}

\def\cdspnbv{{\cdspnb^{\!\vee}}}
\def\cdspnv{{\cdspn^{\!\vee}}}
\def\dspnbv{{\dspnb^{\!\vee}}}
\def\dspnv{{\dspn^{\!\vee}}}

\def\mipc{{modulo $\IPC$-provable equivalence}{} }

\def\ar{            \mathrel{\setlength{\unitlength}{1ex}
                    \begin{picture}(1.65,1.65)           
                    \put(0,0){\line(0,1){1.65}}                  
                    \put(0,.3){\textup{\footnotesize{$\backsim$}}}
                    \end{picture}}\!
        }
\def\argt{\mathrel{\adsm{\sft}{\Gamma}}}
\def\prtg{\mathrel{\pres{\sft}{\Gamma}}}

\def\artg{\argt}
\def\artgv{\mathrel{\adsm{\sft}{\Gamma^\vee}}\ }

\def\ARNN{\ARD\IPC\parr\VAB}

\def\ARNBSK{\ARDP\ikfourca\atomb}

\def\ARNBST{\ARDP\sft\atomb}
\def\BART{\BAR\sft}
\newcommand{\BAR}[1]{[#1]}

\def\ARNBTP{\ARNBT\VAB}
\def\ARNBKP{\ARNBK\VAB}
\def\ARNBIP{\ARD\IPC\parb\Le\VAB}
\def\parb{{\sf parb}}

\def\atomb{{\sf atomb}}

\def\ARNBK{\ARD\ikfourcp\parb\Le}

\def\ARNBT{\ARNBTM\Le}
\def\AARNBT{\ARDM\sft\parb}
\def\AARNBK{\ARDM\ikfourcp\parb}

\def\ARNBIM{\ARD\ipc\parb}

\def\ARNBTM{\ARD\sft\parb}

\def\ARNBSI{\ARDP\IPC\atomb}

\def\AR{{\sf AR}}
\def\VAR{{\sf V}^\parr_{\!\text{\fontsize{6}{0}\selectfont\AR}}}

\newcommand{\ARD}[2]{\ttbrace{#1,#2}}
\newcommand{\ARDP}[2]{\ARD{#1}{#2}\Le}
\newcommand{\ARDM}[2]{\ARD{#1}{#2}\Lem}
\newcommand{\LARD}[2]{\tcbrace{#1,#2}}
\newcommand{\LARDP}[2]{\tcbracep{#1,#2}}

\newcommand{\itv}[3]{#1\xrightarrow{#3}{#2}}

\def\mont{\textup{Mont}}
\def\amont{\textup{AMont}}

\def\montd{\mont(\Delta)\xspace}
\def\amontd{\amont(\Delta)\xspace}

\def\pnnil{{\sf PN}}

\def\Le{{\sf Le}\xspace}
\def\Lem{{\sf Le}^{^-}\xspace}
\def\Les{{\sf Le}\xspace}

\newcommand{\vdashp}{\Vdash^{^{\!\!+}}}

\newcommand{\nin}{\not\in}

\newcommand\xtwoheadrightarrow[2]{\ensurestackMath{\mathrel{%
			\stackengine{1pt}{%
				\stackengine{0pt}{\twoheadrightarrow}{\scriptstyle#2}{O}{c}{F}{F}{S}%
			}{\scriptstyle#1}{U}{c}{F}{F}{S}%
}}}

\newcommand\xrightarrowtail[2]{\ensurestackMath{\mathrel{%
			\stackengine{1pt}{%
				\stackengine{0pt}{{\rightarrowtail}\!\!\!\!\!{\twoheadrightarrow}}{\scriptstyle#2}{O}{c}{F}{F}{S}%
			}{\scriptstyle#1}{U}{c}{F}{F}{S}%
}}}

\newcommand{\xrat}[2]{\xrightarrowtail{#2}{#1}}

\newcommand{\xrattht}{\xrightarrowtail{\sft}{\theta}}

\newcommand{\xtratht}{\xtwoheadrightarrow{\sft}{\theta}}

\def\sft{{\sf T}}
\def\sfn{{\sf N}}

\def\vdasht{\vdashsub\sft}

\def\nvdasht{\nvdashsub\sft}
\def\tvdash{\sft\vdash}

\newcommand{\fcal}{\mathcal{F}}

\def\VAB{{\sf A}}
\def\th{\ov{\theta}}

\def\cto{c_{\!_\to\!}}

\def\ctob{c^\Box_{\!_\to\!}}
\def\cftob{{c}^\Box_{\!_\to\!}}

\def\sfp{{\sf p}}
\def\sfh{{\sf h}}

\def\sfss{{\sf s}}
\def\sfa{{\sf a}}
\def\artsnb{\adsm{\sft}{\snnil}\ }
\def\artsnbv{\adsm{\sft}{\snnilv}\ \ }
\def\artspnb{\adsm{\sft}{\spnnil}\ }
\def\artspnbv{\adsm{\sft}{\spnv}\ \ \ }
\def\artnb{\adsm{\sft}{\nnil}}
\def\artnbv{\adsm{\sft}{\nnilv}}
\def\artpnb{\adsm{\sft}{\pnnil}\ }
\def\artpnbv{\adsm{\sft}{\pnnilv}\ \ }

\newcommand{\subo}[1]{{\sf sub_o}(#1)}

\newcommand{\subob}[1]{\boxed{\sf sub_{o}}(#1)}
\newcommand{\suboab}[1]{\atomb{\sf sub_{o}}(#1)}
\newcommand{\subab}[1]{\atomb{\sf sub}(#1)}
\newcommand{\subb}[1]{\boxed{\sf sub}(#1)}

\newcommand{\subp}[1]{\parr{\sf sub}(#1)}
\newcommand{\subpb}[1]{\parb{\sf sub}(#1)}

\newcommand{\suba}[1]{\atom{\sf sub}(#1)}
\newcommand{\suboa}[1]{\atom{\sf sub_o}(#1)}
\def\BA{{\vv{\Box B}}}

\newcommand{\nf}[1]{#1\text{-{\sf NF}}}
\newcommand{\nfz}[1]{\hat{#1}}
\def\Rvar{\mathrel{\widetilde{\sqsubseteq}}}
\def\Rvarr{\mathrel{\widetilde{\sqsubset}}}
\def\atom{{\sf atom}}
\newcommand{\varr}{{\sf var}}
\newcommand{\parr}{{\sf par}}
\def\cpp{{\sf C_p}}
\def\cpv{{\sf C_v}}
\def\cpa{{\sf C_a}}

\def\dnbv{\darl{\nnil}^{\!\vee}}
\def\dnb{\darl{\nnil}}
\def\dsnb{\darl{\wsnnil}}

\def\cdsnbv{{\cdsnb^{\!\vee}}}
\def\dsnbv{{\dsnb^{\!\vee}}}

\def\cdsnb{{\sfc}\darl{\wsnnil}}
\def\cdspnb{{\sfc}{\dspnb}}
\def\cdspn{{\sfc}{\dspn}}
\def\dspnb{\darl{\spnnilb}}
\def\dspn{\darl{\spnnil}}
\def\dpnb{{\downarrow}{\pnnilb}}
\def\dpnbv{{\downarrow}{\pnnilb}^{\!\vee}}

\newcommand{\ggt}[1]{{#1}^\Box}
\newcommand{\sggt}[1]{{#1}^{\boxdot}}
\def\NI{{\sf NI}}

\newcommand{\dfrak}[1]{\mathfrak{d}(#1)}

\def\nat{\mathbb{N}}
\def\tcom{\sfc(\sft)}

\def\sfc{{\sf C}}
\def\sfs{{\sf S}}
\def\sfsb{\bar{\sf S}}

\def\sfP{{\sf P}}
\newcommand{\ap}[3]{\lfloor{#3}\rfloor^{^{\!#2}}_{_{\!{#1}}}}
\newcommand{\ape}[1]{\ap{}{}{#1}}

\newcommand{\iap}[3]{\lfloor\!\!\lfloor{#3}\rfloor\!\!\rfloor^{^{\!#2}}_{_{#1}}}

\def\fsubeq{\subseteq_{\text{fin}}}
\newcommand{\dar}[2]{{\downarrow}^{^{\!\!{#1}}}{#2}}
\newcommand{\darl}[1]{{\downarrow}{#1}}

\def\wfix{{w_1}}

\def\aha{\alpha_{_{\sf HA}}}
\def\apa{\alpha_{_{\sf PA}}}
\def\sha{\sigma_{_{\sf HA}}}
\def\ahap{\aha^+}
\def\boxed{{\sf B}}
\def\typez{\text{{\sf TYPE-0}}}
\def\typeo{\text{{\sf TYPE-$\alpha$}}}
\def\types{\typet}
\def\typet{\text{{\sf TYPE-$\sigma$}}}
\def\typea{\typeo}
\newcommand{\trz}[1]{#1_0^\sfp}
\newcommand{\Trz}[1]{\Pi^0_{#1}}
\newcommand{\tro}[1]{#1^\sfp_1}
\newcommand{\Tro}[1]{\Pi^1_{#1}}
\newcommand{\hTro}[1]{\hat{\Pi}^1_{#1}}
\newcommand{\trt}[1]{#1^\sfp}
\newcommand{\Trt}[1]{\Pi_{#1}}
\newcommand{\trth}[1]{#1^\sfa}
\newcommand{\Trth}[1]{\Pi'_{#1}}
\newcommand{\viss}[1]{{\sf V}(#1)}
\def\sstar{{\star\!\star}}

\def\ZFC{{\sf ZFC}}
\newcommand{\prim}[1]{{\sf P}(#1)}
\newcommand{\Prim}[1]{{\sf Prime}(#1)}
\def\pcal{\mathcal{P}}
\newcommand{\lgc}[1]{{\sf T}_{#1}}
\newcommand{\lgci}[2]{{\sf T}^{#1}_{#2}}

\def\pfrak{\mathfrak{P}}
\newcommand{\pmodelssymbol}{\setlength{\unitlength}{1ex}
                    \begin{picture}(2,2)
                    \put(-0.1,0){$|$}
                    \put(0.2,0){$\models$}
                    \end{picture}}             
\newcommand{\pmodels}{\mathrel{\pmodelssymbol}}
\newcommand{\npmodels}{\not\pmodels}   
\newcommand{\pmodelsp}{\pmodels^+}   
\newcommand{\pmodelss}{\pmodels^*}

\def\ik{{\sf iK}}

\newcommand{\cpm}[2]{\pcal[{#1},{#2}]}
\newcommand{\cpmp}[3]{\pcal[{#1},{#2},{#3}]}
\def\vay{Y}
\def\zed{Z}
\newcommand{\zeds}[1]{Z_{#1}}
\def\ipred{\mathrel{\R^1}}
\def\Rv{\mathrel{\R'}}
\def\vR{\mathrel{\sqsupset'}}

\newcommand{\ov}[1]{\check{#1}}

\def\lcalbp{\lcal'_\Box}

%% file: PLHA.bbl
\begin{thebibliography}{}

\bibitem[Ardeshir and Mojtahedi, 2015]{reduction}
Ardeshir, M. and Mojtahedi, M. (2015).
\newblock {R}eduction of provability logics to {$\Sigma_1$}-provability logics.
\newblock {\em Logic Journal of IGPL}, 23(5):842--847.

\bibitem[Ardeshir and Mojtahedi, 2018]{Sigma.Prov.HA}
Ardeshir, M. and Mojtahedi, M. (2018).
\newblock The {$\Sigma_1$}-{P}rovability {L}ogic of {${\sf HA}$}.
\newblock {\em Annals of Pure and Applied Logic}, 169(10):997--1043.

\bibitem[Ardeshir and Mojtahedi, 2019]{Sigma.Prov.HA*}
Ardeshir, M. and Mojtahedi, M. (2019).
\newblock The {$\Sigma_1$}-{P}rovability {L}ogic of {${\sf HA}^*$}.
\newblock {\em Jornal of Symbolic Logic}, 84(3):118--1135.

\bibitem[Artemov and Beklemishev, 2004]{ArtBekProv}
Artemov, S. and Beklemishev, L. (2004).
\newblock Provability logic.
\newblock In Gabbay, D. and Guenthner, F., editors, {\em in Handbook of
  Philosophical Logic}, volume~13, pages 189--360. Springer, 2nd edition.

\bibitem[Beklemishev and Visser, 2006]{VisBek}
Beklemishev, L. and Visser, A. (2006).
\newblock Problems in the logic of provability.
\newblock In {\em Mathematical problems from applied logic. {I}}, volume~4 of
  {\em Int. Math. Ser. (N. Y.)}, pages 77--136. Springer, New York.

\bibitem[Berarducci, 1990]{Berarducci}
Berarducci, A. (1990).
\newblock The {I}nterpretability {L}ogic of {P}eano {A}rithmetic.
\newblock {\em Journal of Symbolic Logic}, 55(3):1059--1089.

\bibitem[Boolos, 1995]{Boolos}
Boolos, G. (1995).
\newblock {\em The Logic of Provability}.
\newblock Cambridge University Press.

\bibitem[de~Jongh and Veltman, 1990]{Velt-Jongh}
de~Jongh, D. and Veltman, F. (1990).
\newblock Provability logics for relative interpretability.
\newblock In Petkov, P., editor, {\em Mathematical Logic}, pages 31--41.
  Springer.

\bibitem[Friedman, 1975]{Friedman75}
Friedman, H. (1975).
\newblock The disjunction property implies the numerical existence property.
\newblock {\em Proc. Nat. Acad. Sci. U.S.A.}, 72(8):2877--2878.

\bibitem[Ghilardi, 1997]{Ghilardi97}
Ghilardi, S. (1997).
\newblock Unification through projectivity.
\newblock {\em J. Log. Comput.}, 7(6):733--752.

\bibitem[Ghilardi, 1999]{Ghil99}
Ghilardi, S. (1999).
\newblock Unification in {I}ntuitionistic {L}ogic.
\newblock {\em Journal of {S}ymbolic {L}ogic}, 64(2):859--880.

\bibitem[Ghilardi, 2000]{Ghil2000modal}
Ghilardi, S. (2000).
\newblock {B}est solving modal equations.
\newblock {\em Annalas of {P}ure and {A}pplied {L}ogic}, 102(3):183--198.

\bibitem[G{\"o}del, 1931]{Godel}
G{\"o}del, K. (1931).
\newblock \"{U}ber formal unentscheidbare {S}\"atze der {P}rincipia
  {M}athematica und verwandter {S}ysteme {I}.
\newblock {\em Monatsh. Math. Phys.}, 38(1):173--198.

\bibitem[G{\"o}del, 1933]{Godel33}
G{\"o}del, K. (1933).
\newblock Eine interpretation des intuitionistischen aussagenkalkuls.
\newblock {\em Ergebnisse eines mathematischen Kolloquiums}, 4:39--40.
\newblock English translation in: S. Feferman etal., editors, Kurt G{\"o}del
  Collected Works, Vol. 1, pages 301-303. Oxford University Press, 1995.

\bibitem[Goris and Joosten, 2011]{goris2011new}
Goris, E. and Joosten, J.~J. (2011).
\newblock A new principle in the interpretability logic of all reasonable
  arithmetical theories.
\newblock {\em Logic Journal of the IGPL}, 19(1):1--17.

\bibitem[Goudsmit and Iemhoff, 2014]{goudsmit2014unification}
Goudsmit, J.~P. and Iemhoff, R. (2014).
\newblock On unification and admissible rules in gabbay--de jongh logics.
\newblock {\em Annals of Pure and Applied Logic}, 165(2):652--672.

\bibitem[Iemhoff, 2001a]{Iemhoff}
Iemhoff, R. (2001a).
\newblock A modal analysis of some principles of the provability logic of
  {H}eyting arithmetic.
\newblock In {\em Advances in modal logic, {V}ol.\ 2 ({U}ppsala, 1998)}, volume
  119 of {\em CSLI Lecture Notes}, pages 301--336. CSLI Publ., Stanford, CA.

\bibitem[Iemhoff, 2001b]{Iemhoff-admissibility}
Iemhoff, R. (2001b).
\newblock {O}n the {A}dmissible {R}ules of {I}ntuitionistic {P}ropositional
  {L}ogic.
\newblock {\em The Journal of Symbolic Logic}, 66(1):281--294.

\bibitem[Iemhoff, 2001c]{IemhoffT}
Iemhoff, R. (2001c).
\newblock {\em Provability Logic and Admissible Rules}.
\newblock PhD thesis, University of Amsterdam.

\bibitem[Iemhoff, 2003]{Iemhoff.Preservativity}
Iemhoff, R. (2003).
\newblock {P}reservativity {L}ogic. ({A}n analogue of interpretability logic
  for constructive theories).
\newblock {\em Mathematical Logic Quarterly}, 49(3):1--21.

\bibitem[Iemhoff, 2016]{iemhoff2016consequence}
Iemhoff, R. (2016).
\newblock Consequence relations and admissible rules.
\newblock {\em Journal of Philosophical Logic}, 45(3):327--348.

\bibitem[Iemhoff et~al., 2005]{Iemhoff2005}
Iemhoff, R., De~Jongh, D., and Zhou, C. (2005).
\newblock Properties of intuitionistic provability and preservativity logics.
\newblock {\em Logic Journal of the IGPL}, 13(6):615--636.

\bibitem[Iemhoff and Metcalfe, 2009]{iemhoff2009proof}
Iemhoff, R. and Metcalfe, G. (2009).
\newblock Proof theory for admissible rules.
\newblock {\em Annals of Pure and Applied Logic}, 159(1-2):171--186.

\bibitem[Je{\v{r}}{\'a}bek, 2005]{jevrabek2005admissible}
Je{\v{r}}{\'a}bek, E. (2005).
\newblock Admissible rules of modal logics.
\newblock {\em Journal of Logic and Computation}, 15(4):411--431.

\bibitem[Leivant, 1979]{Leivant}
Leivant, D. (1979).
\newblock {\em Absoluteness of intuitionistic logic}, volume~73 of {\em
  Mathematical Centre Tracts}.
\newblock Mathematisch Centrum, Amsterdam.

\bibitem[{L}\"{o}b, 1955]{Lob}
{L}\"{o}b, M. (1955).
\newblock {S}olution of a {P}roblem of {L}eon {H}enkin.
\newblock {\em Journal of Symbolic Logic}, 20(2):115--118.

\bibitem[Mojtahedi, 2021]{mojtahedi2021hard}
Mojtahedi, M. (2021).
\newblock Hard provability logics.
\newblock In {\em Mathematics, Logic, and their Philosophies}, pages 253--312.
  Springer.

\bibitem[Mojtahedi, 2024]{Rev-SPLHA}
Mojtahedi, M. (2024).
\newblock {\em The $\Sigma_1$-{P}rovability {L}ogic of {H}{A} {R}evisited},
  pages 89--110.
\newblock Springer International Publishing, Cham.

\bibitem[Mojtahedi, 2025]{PLHA0}
Mojtahedi, M. (2025).
\newblock Relative unification in intuitionistic logic: Towards the provability
  logic of ha.
\newblock {\em The Journal of Symbolic Logic}, page 1–41.

\bibitem[Mojtahedi and Miranda, 2025]{PModels}
Mojtahedi, M. and Miranda, B.~S. (2025).
\newblock Provability models.
\newblock https://arxiv.org/abs/2510.06696.

\bibitem[Myhill, 1973]{Myhill}
Myhill, J. (1973).
\newblock A note on indicator-functions.
\newblock {\em Proceedings of the AMS}, 39:181--183.

\bibitem[Rybakov, 1987a]{Rybakov87}
Rybakov, V.~V. (1987a).
\newblock {B}ases of admissible rules of the modal system {G}rz and of
  intuitionistic logic.
\newblock {\em Mathematics of the USSR-Sbornik}, 56(2):311--331.

\bibitem[Rybakov, 1987b]{Rybakov_1987}
Rybakov, V.~V. (1987b).
\newblock {D}ecidability of admissibility in the modal system {G}rz and in
  intuitionistic logic.
\newblock {\em Mathematics of the {USSR}-Izvestiya}, 28(3):589--608.

\bibitem[Rybakov, 1997]{Rybakov_Book}
Rybakov, V.~V. (1997).
\newblock {\em {A}dmissibility of logical inference rules}.
\newblock Elsevier.

\bibitem[Smory{\'n}ski, 1985]{Smorynski-Book}
Smory{\'n}ski, C. (1985).
\newblock {\em Self-reference and modal logic}.
\newblock Universitext. Springer-Verlag, New York.

\bibitem[Solovay, 1976]{Solovay}
Solovay, R.~M. (1976).
\newblock Provability interpretations of modal logic.
\newblock {\em Israel J. Math.}, 25(3-4):287--304.

\bibitem[Verbrugge, 2017]{sep-logic-provability}
Verbrugge, R.~L. (2017).
\newblock {Provability Logic}.
\newblock In Zalta, E.~N., editor, {\em The {Stanford} Encyclopedia of
  Philosophy}. Metaphysics Research Lab, Stanford University, {F}all 2017
  edition.

\bibitem[Visser, 1981]{VisserThes}
Visser, A. (1981).
\newblock {\em Aspects of Diagonalization and Provability}.
\newblock PhD thesis, Utrecht University.

\bibitem[Visser, 1982]{Visser82}
Visser, A. (1982).
\newblock On the completeness principle: a study of provability in {H}eyting's
  arithmetic and extensions.
\newblock {\em Ann. Math. Logic}, 22(3):263--295.

\bibitem[Visser, 1990]{Visser-Interpretability}
Visser, A. (1990).
\newblock Interpretability logic.
\newblock In {\em Mathematical logic}, pages 175--209. Plenum, New York.

\bibitem[Visser, 1998]{VisserInterpretability}
Visser, A. (1998).
\newblock An overview of interpretability logic.
\newblock In {\em Advances in modal logic, {V}ol.\ 1 ({B}erlin, 1996)},
  volume~87 of {\em CSLI Lecture Notes}, pages 307--359. CSLI Publ., Stanford,
  CA.

\bibitem[Visser, 2002]{Visser02}
Visser, A. (2002).
\newblock Substitutions of {$\Sigma_1^0$} sentences: explorations between
  intuitionistic propositional logic and intuitionistic arithmetic.
\newblock {\em Ann. Pure Appl. Logic}, 114(1-3):227--271.
\newblock Commemorative Symposium Dedicated to Anne S. Troelstra
  (Noordwijkerhout, 1999).

\bibitem[Visser et~al., 1995]{Visser-Benthem-NNIL}
Visser, A., van Benthem, J., de~Jongh, D., and R.~de Lavalette, G.~R. (1995).
\newblock {${\rm NNIL}$}, a study in intuitionistic propositional logic.
\newblock In {\em Modal logic and process algebra ({A}msterdam, 1994)},
  volume~53 of {\em CSLI Lecture Notes}, pages 289--326. CSLI Publ., Stanford,
  CA.

\bibitem[Visser and Zoethout, 2019]{Jetze-Visser}
Visser, A. and Zoethout, J. (2019).
\newblock Provability logic and the completeness principle.
\newblock {\em Annals of Pure and Applied Logic}, 170(6):718--753.

\end{thebibliography}
